\documentclass[a4paper,11pt]{amsart}
\usepackage[margin=3.5cm]{geometry}
\usepackage[english]{babel}
\usepackage{graphicx}
\usepackage{graphics}
\usepackage{amsmath}
\usepackage{amsfonts}
\usepackage{amssymb}
\usepackage{amsthm}
\usepackage{float}	
\usepackage{lmodern} \normalfont
\usepackage[T1]{fontenc}
\usepackage{tikz}
\usepackage{color}
\usepackage{hyperref}
\usepackage{array}
\usepackage{longtable}
\usepackage{changepage}
\usepackage{lscape}
\usepackage{enumitem}
\usetikzlibrary{cd}
\setlist[itemize]{leftmargin=11pt,itemsep=5pt, topsep=5pt}
\usetikzlibrary{positioning,calc}
\usetikzlibrary{shapes}
\DeclareFontShape{T1}{lmr}{bx}{sc} { <-> ssub * cmr/bx/sc }{}
\hypersetup{pdfstartview={XYZ null null 1.00}}

\newcommand\Graph[1]{\Gamma_{\{#1\}}}
\newcommand\typeone{K_5(1)}
\newcommand\typetwoa{K_4^a(0)}
\newcommand\typetwob{K_4^b(0)}

\newtheoremstyle{een}
{11pt}% measure of space to leave above the theorem. E.g.: 3pt
{11pt}% measure of space to leave below the theorem. E.g.: 3pt
{\slshape}% name of font to use in the body of the theorem
{}% measure of space to indent
{\sc}% name of head font
{.}% punctuation between head and body
{1mm}% space after theorem head
{}% Manually specify head

\newtheoremstyle{twee}
{11pt}% measure of space to leave above the theorem. E.g.: 3pt
{11pt}% measure of space to leave below the theorem. E.g.: 3pt
{}% name of font to use in the body of the theorem
{}% measure of space to indent
{\sc}% name of head font
{.}% punctuation between head and body
{1mm}% space after theorem head
{}% Manually specify head

\theoremstyle{een}
\newtheorem{theorem}{\textbf{Theorem}}
\newtheorem{proposition}{\textbf{Proposition}}
\newtheorem{lemma}{\textbf{Lemma}}
\newtheorem{corollary}{\textbf{Corollary}}
\newtheorem{definition}{\textbf{Definition}}

\theoremstyle{twee}

\newtheorem{remark}{\textbf{Remark}}
\newtheorem{remarkdp1}[remark]{\textbf{Remark -- analogy with geometry}}

\newtheorem*{theorem*}{\textbf{Theorem}}
\newtheorem*{proofofpropositionA}{\textbf{Proof of Proposition \ref{A}}}
\newtheorem*{proofofpropositionB}{\textbf{Proof of Proposition \ref{B}}}
\newtheorem*{proofofthm1}{\textbf{Proof of Theorem \ref{main}}}
\newtheorem*{proofofthm2}{\textbf{Proof of Theorem \ref{main2}}}

\newcounter{claimcounter}
\numberwithin{claimcounter}{theorem}

\makeatletter
\renewcommand{\section}{\@startsection
{section}% % the name
{1}% % the level
{0mm}% % the indent
{\baselineskip}% % the before skip
{0.5\baselineskip}% % the after skip
{\normalfont\large\bfseries}} % the style

\renewcommand{\subsection}{\@startsection
{subsection}% % the name
{2}% % the level
{0mm}% % the indent
{\baselineskip}% % the before skip
{0.5\baselineskip}% % the after skip
{\normalfont\large\bfseries}} % the style

\renewcommand{\subsubsection}{\@startsection
{subsubsection}% % the name
{3}% % the level
{0mm}% % the indent
{\baselineskip}% % the before skip
{0.5\baselineskip}% % the after skip
{\normalfont\normalsize\bfseries}} % the style

\newcommand{\addresseshere}{%
  \enddoc@text\let\enddoc@text\relax
}
\makeatother

\DeclareMathOperator{\Aut}{Aut}

\begin{document}

\title{The action of the Weyl group on the $E_8$ root system}

\author{Rosa Winter and Ronald van Luijk}
\address{MPI-MiS, Inselstrasse 22, 04103 Leipzig, Germany}
\email{rosa.winter@mis.mpg.de}
\address{Mathematisch Instituut, Niels Bohrweg 1, 2333 CA Leiden, The Netherlands}
\email{rvl@math.leidenuniv.nl}

\begin{abstract}
Let $\Gamma$ be the graph on the roots of the $E_8$ root system, where any two distinct vertices $e$ and $f$ are connected by an edge with color equal to the inner product of $e$ and $f$. For any set $c$ of colors, let $\Gamma_c$ be the subgraph of $\Gamma$ consisting of all the $240$ vertices, and all the edges whose color lies in $c$. We consider cliques, i.e., complete subgraphs, of $\Gamma$ that are either monochromatic, or of size at most $3$, or a maximal clique in $\Gamma_c$ for some color set $c$, or whose vertices are the vertices of a face of the $E_8$ root polytope. We prove that, apart from two exceptions, two such cliques are conjugate under the automorphism group of $\Gamma$ if and only if they are isomorphic as colored graphs. Moreover, for an isomorphism $f$ from one such clique $K$ to another, we give necessary and sufficient conditions for $f$ to extend to an automorphism of $\Gamma$, in terms of the restrictions of $f$ to certain special subgraphs of~$K$ of size at most~$7$.
\end{abstract}

\maketitle
%\tableofcontents
\section{Introduction}

Let $\Lambda$ be the $E_8$ lattice, that is, the unique positive-definite, even, unimodular lattice of dimension 8. More concretely, let $\Lambda$ be given by 
$$\Lambda=\left\{a\in \mathbb{Z}^8+\left\langle\left(\tfrac{1}{2},\tfrac{1}{2},\tfrac{1}{2},\tfrac{1}{2},\tfrac{1}{2},\tfrac{1}{2},\tfrac{1}{2},\tfrac{1}{2}\right)\right\rangle\;\Biggl\vert\;\sum_{i=1}^8a_i\in 2\mathbb{Z}\right\}.$$
Consider the $E_8$ root system $E$ in $\Lambda$ given by $$E=\{a\in\Lambda\;|\;\|a\|=\sqrt{2}\}.$$ 

In this artice we study a graph on the elements in $E$, which we call \textsl{roots}. By a \textsl{graph} we mean a pair $(V,D)$, where $V$ is a set of elements called \textsl{vertices}, and $D$ a subset of the powerset of $V$ of which every element has cardinality 2; elements in~$D$ are called \textsl{edges}, and the \textsl{size} of the graph is the cardinality of $V$. By a \textsl{colored graph} we mean a graph $(V,D)$ together with a map $\varphi\colon D\longrightarrow C$, where $C$ is any set, whose elements we call \textsl{colors}; for an element $d\in D$ we call $\varphi(d)$ its color. If $(V,D)$ is a colored graph with color function $\varphi$, we define a \textsl{colored subgraph} of $(V,D)$ to be a pair $(V',D')$ with a map $\varphi'$, such that $V'$ is a subset of $V$, while $D'$ is a subset of the intersection of $D$ with the powerset on $V'$, and $\varphi'$ is the restriction of $\varphi$ to $D'$. Finally, we define a \textsl{clique} of a colored graph to be a complete colored subgraph.

\vspace{11pt}

Let $\Gamma$ be the complete colored graph whose vertex set is $E$, of which the color function on the edge set is induced by the dot product. The different colors of the edges in~$\Gamma$ are $-2,-1,0,1$. For a subset $c\subseteq\{-2,-1,0,1\}$, we denote by $\Gamma_c$ the colored subgraph of $\Gamma$ with vertex set $E$ and all edges whose color is an element in~$c$. 

\vspace{11pt}

Let $W$ be the automorphism group of $\Gamma$. It is clear that if two cliques in $\Gamma$ are conjugate under the action of $W$, they must be isomorphic. The converse is not always true, and in general it can be hard to determine whether two cliques in $\Gamma$ are conjugate under the action of $W$. Dynkin and Minchenko studied in \cite{DynMin10} the bases of subsystems of $E_8$, and classified for which isomorphism classes of these bases being isomorphic implies being conjugate. They call these bases \textsl{normal}. In this article, we extend this classification to a large set of cliques in $\Gamma$ (more specifically, cliques of type I, II, III, or IV, as defined below). In Theorem \ref{main} we show that with two exceptions, two such cliques are isomorphic if and only if they are conjugate. One of the exceptions, which is the clique described in Theorem \ref{main} (i), is one of the bases (of the system $4A_1$) that was also found as not being normal in \cite[Theorem 4.7]{DynMin10}. Additionally, in \cite{DynMin10} the authors determine when a homomorphism of two bases of subsytems extends to a homomorphism of the whole root system. We answer the same question for cliques of type I, II, III, or IV in Theorem \ref{main2}.

\vspace{11pt}

Although the classification of different types of cliques and their orbits is a finite problem, because of the size of $\Gamma$ it is practically impossible to naively let a computer find and classify the cliques according to their $W$-orbit. In fact, we avoid using a computer for our computations as much as possible.

\vspace{11pt}

The $E_8$ root polytope is the convex polytope in $\mathbb{R}^8$ whose vertices are the roots in~$E$. By a \textsl{face} of the root polytope we mean a non-empty intersection of a hyperplane in $\mathbb{R}^8$ and the root polytope, such that the root polytope lies entirely on one side of the hyperplane. If the dimension of this intersection is $k$ then we call this a $k$-face, and a $7$-face is called a \textsl{facet}. We study the following cliques in $\Gamma$, and their orbits under the action of $W$.
\begin{itemize}
\item[](I) Monochromatic cliques
\item[](II) Cliques whose vertices are the vertices of a face of the $E_8$ root polytope
\item[](III) Cliques of size at most three
\item[](IV) For all $c\neq\{-1,0,1\}$, the maximal cliques in $\Gamma_c$ 
\end{itemize}

More specifically, we prove the following theorem. 

\begin{theorem}\label{main}
Let $K_1,\;K_2$ be two cliques in $\Gamma$ of types I, II, III, or IV. Then the following hold.
\begin{itemize}
\item[](i) If both $K_1$ and $K_2$ are of type I with color $0$ and of size $4$, then $K_1$ and $K_2$ are conjugate under the action of $W$ if and only if the vertices sum to an element in $2\Lambda$ for both $K_1$ and $K_2$, or for neither.
\item[](ii) If both $K_1$ and $K_2$ are of type I with color $1$ and of size $7$, then $K_1$ and $K_2$ are conjugate under the action of $W$ if and only if the vertices sum to an element in $2\Lambda$ for both $K_1$ and $K_2$, or for neither; this is equivalent to $K_1$ and $K_2$ both being maximal or both being non-maximal, respectively, under inclusion in $\Graph{1}$. 
\item[](iii) In all other cases, $K_1$ and $K_2$ are conjugate under the action of $W$ if and only if they are isomorphic as colored graphs.
\end{itemize}
\end{theorem}

Furthermore, we give conditions for an isomorphism of two cliques of types I, II, III or IV to extend to an automorphism of the lattice $\Lambda$. To this end we introduce the following complete graphs.  

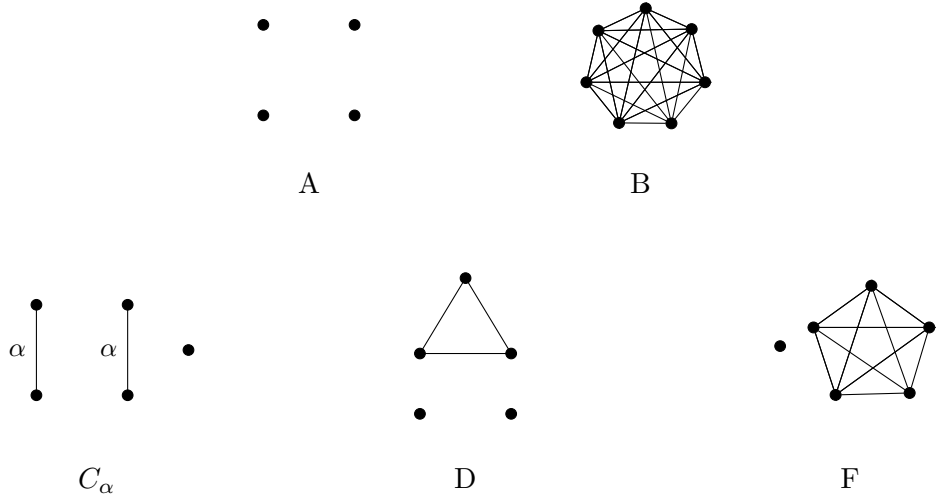
\begin{center}
\resizebox{\columnwidth}{!}{
\begin{tabular}{ccccccccc}
& & \raisebox{1pt}{\begin{tikzpicture} [scale=0.3]
 \node [draw,circle,fill,inner sep=0pt,minimum size=4pt](t1) at (0,2) {};
 \node [draw,circle,fill,inner sep=0pt,minimum size=4pt](t2) at (0,6) {};
 \node [draw,circle,fill,inner sep=0pt,minimum size=4pt](t3) at (4,2) {};
 \node [draw,circle,fill,inner sep=0pt,minimum size=4pt](t3) at (4,6) {};
 \end{tikzpicture}}

&&&&
\raisebox{-2pt}{
\begin{tikzpicture} [scale=0.4]
  \foreach \x /\alph /\name in {347/a/$e_1$, 41/b/$e_2$, 90/c/$e_3$,  141/d/$e_4$,  193/e/$e_5$,  244/f/$e_6$, 295/g/$e_7$}{
  \node[circle,fill,inner sep=0,minimum size=0.01pt,draw,scale=0.35] (\alph) at (\x:2cm) {\name}; }

  \foreach \alpha in {a,b,c,d,e,f,g}%
  {%
  \foreach \alphb in {a,b,c,d,e,f}%
  {%
   \draw (\alpha) -- (\alphb);%
  }}
 \end{tikzpicture}}

 & & \\
 &&&&&&&&\\
 && A &&&& B && \\
&&&&&&&&\\
&&&&&&&&\\
\raisebox{7pt}{
\begin{tikzpicture}[scale=0.4]
\node [draw,circle,fill,inner sep=0pt,minimum size=4pt](t1) at (0,2) {};
 \node [draw,circle,fill,inner sep=0pt,minimum size=4pt](t2) at (0,5) {};
 \node [draw,circle,fill,inner sep=0pt,minimum size=4pt](t3) at (3,2) {};
 \node [draw,circle,fill,inner sep=0pt,minimum size=4pt](t4) at (3,5) {};
 \node [draw,circle,fill,inner sep=0pt,minimum size=4pt](t5) at (5,3.5) {};
\path[every node/.style={font=\sffamily\small}]
     (0,2) edge node [midway,left]{$\alpha$} (0,5)
     (3,2) edge node [midway,left]{$\alpha$} (3,5) ;
\end{tikzpicture} }

&&&&
 \begin{tikzpicture}[scale=0.4] 
\node [draw,circle,fill,inner sep=0pt,minimum size=4pt](t1) at (1,1.5) {};
 \node [draw,circle,fill,inner sep=0pt,minimum size=4pt](t5) at (4,1.5) {};
 \node [draw,circle,fill,inner sep=0pt,minimum size=4pt](t9) at (2.5,4) {};
 \node [draw,circle,fill,inner sep=0pt,minimum size=4pt](t1) at (1,-0.5) {};
 \node [draw,circle,fill,inner sep=0pt,minimum size=4pt](t5) at (4,-0.5) {};
\path[every node/.style={font=\sffamily\small}]
     (1,1.5) edge  (4,1.5)
     (4,1.5) edge (2.5,4)
     (2.5,4) edge (1,1.5)
     ;
\end{tikzpicture} 

& & & &\raisebox{7pt}{
 \begin{tikzpicture} [scale=0.4]
  \node [draw,circle,fill,inner sep=0pt,minimum size=4pt] at (-3,0) {};
  \foreach \x /\alph /\name in {18/a/$e_1$, 90/b/$e_2$, 162/c/$e_3$,  234/d/$e_4$,  309/e/$e_5$}{
  \node[circle,fill,inner sep=0,minimum size=0.01pt,draw,scale=0.35] (\alph) at (\x:2cm) {\name}; }

  \foreach \alpha in {a,b,c,d,e}%
  {%
  \foreach \alphb in {a,b,c,d}%
  {%
   \draw (\alpha) -- (\alphb);%
  }}
  \end{tikzpicture}}
 \\
 &&&&&&&&\\
 $C_{\alpha}$ &&&& D &&&& F \\
&&&&&&&&
 \end{tabular}}
\end{center}
Here $\alpha$ is either $-1$ or $1$, two disjoint vertices have an edge of color 0 between them, and all other edges have color $1$. 

\begin{theorem}\label{main2}
Let $K_1,\;K_2$ be two cliques in $\Gamma$ of types I, II, III, or IV, and let $f\colon K_1\longrightarrow K_2$ be an isomorphism between them. The following hold.
\begin{itemize}
\item[](i) The map $f$ extends to an automorphism of $\Lambda$ if and only if for every ordered sequence $S=(e_1,\ldots,e_r)$ of distinct roots in $K_1$ such that the colored graph on them induced by $\Gamma$ is isomorphic to A, B, $C_{\alpha}$, D, or F, its image $f(S)=(f(e_1),\ldots,f(e_r))$ is conjugate to $S$ under the action of $W$.
\item[](ii) If $S=(e_1,\ldots,e_r)$ is a sequence of distinct roots in $K_1$ such that the colored graph on them induced by $\Gamma$ is isomorphic to either $A$ or $B$, then $S$ and $f(S)$ are conjugate under the action of $W$ if and only if the sets $\{e_1,\ldots,e_r\}$ and $\{f(e_1),\ldots,f(e_r)\}$ are. 
\item[](iii) If $K_1$ and $K_2$ are maximal cliques, both in $\Graph{-1,0}$ or both in $\Graph{-2,-1,0}$, and $S=(e_1,\ldots,e_5)$ is a sequence of roots in $K_1$ such that the colored graph on them induced by $\Gamma$ is isomorphic to $C_{-1}$ with $e_1\cdot e_4=e_2\cdot e_5=-1$, then $S$ and $f(S)$ are conjugate under the action of $W$ if and only if both $e=e_1+e_2+e_3-e_4-e_5$ and $f(e)$ are in the set $\{2f_1+f_2\;|\;f_1,f_2\in E\}$, or neither are.
\item[](iv) If $K_1$ and $K_2$ are maximal cliques in $\Graph{-2,0,1}$, and $S~=~(e_1,\ldots,e_r)$ is a sequence of distinct roots in $K_1$ such that the colored graph $G$ on them induced by $\Gamma$ is isomorphic to $C_1$, $D$, or $F$, then $S$ and $f(S)$ are conjugate under the action of $W$ if and only if the sets $\{e_1,\ldots,e_r\}$ and $\{f(e_1),\ldots,f(e_r)\}$ are, or equivalently, if and only if the following hold.\\
\noindent $\bullet$ If $G\cong C_{1}$, both $\sum_{i=1}^5e_i$ and $\sum_{i=1}^5f(e_i)$ are in the set $\{2f_1+f_2\;|\;f_1,f_2\in~E\}$, or neither are. \\
$\bullet$ If $G\cong D$, both $\sum_{i=1}^5e_i$ and $\sum_{i=1}^5f(e_i)$ are in $\{2f_1+2f_2\;|\;f_1,f_2\in E\}$, or neither are. \\
$\bullet$ If $G\cong F$, then both $\sum_{i=1}^6e_i$ and $\sum_{i=1}^6f(e_i)$ are in $2\Lambda$, or neither are. 
\end{itemize}
\end{theorem}

\begin{remark}
Note that to apply Theorem \ref{main2} (i) to an isomorphism $f$, we have to know whether certain ordered sequences of roots are conjugate. Theorem \ref{main2}~(ii), in combination with Theorem \ref{main} (i) and (ii), tells us how to verify this when the colored graph on the roots in an ordered sequence is isomorphic to $A$ or $B$. Theorem~\ref{main2}~(iii) and (iv) tells us how to verify this when the colored graph on the roots in an ordered sequence is isomorphic to $C_{\alpha}$, $D$, or $F$.
\end{remark}

\begin{remark}
In the proof of Theorem \ref{main2}, we will specify for each type of $K_1$ and $K_2$ which of the graphs $A$, $B$, $C_{\alpha}$, $D$, and $F$ are needed to check whether an isomorphism $f$ extends. Of course one can see this partially from the size and the colors, but it turns out that we can make stronger statements. For example, surprisingly, an isomorphism between two maximal graphs in $\Graph{0,1}$ always extends, and even uniquely (Lemma~\ref{thm2max01}). In the table in Remark \ref{welkedingenwelkeklieken} we show the requirements for each type of $K_1$ and $K_2$. 
\end{remark}

As we mentioned before, because of the size of $\Gamma$ it is practically impossible to naively let a computer find and classify all cliques of the above types according to their $W$-orbit. This holds mainly for the results in Section \ref{maximalcliques}, where we study cliques of type~IV. This is the only section where we use a computer program, but without using results from the previous sections to minimize the computations it would have been practically undoable. Checking that two cliques are isomorphic is easily done by hand for types I, II, and III, since with one exception of size fourteen, they are all of size at most eight (see Sections~\ref{facetsthreecliques} and~\ref{monocliques}). For type IV we give necessary and sufficient invariants to check if two large cliques are isomorphic in Section \ref{maximalcliques}. 

\begin{remark}\label{previouswork}
Apart from the work in \cite{DynMin10} on bases of subsystems of $E_8$, some partial results of Theorems \ref{main} and \ref{main2} were known before. We list them here and compare them to our results. \\
The orbits of the faces of the $E_8$ root polytope under the action of $W$ are described in \cite[Section 7.5]{Cox30}. These include all monochromatic cliques of color~1 (see Proposition~\ref{facesform}). For one of the types of facets, we give a different, more group-theoretical proof of the fact that they form one orbit under the action of $W$, see Corollary~\ref{transsevenface}. \\
The orbits of \textsl{ordered sequences} of the vertices in the faces (except for one type of facets) have been described in \cite[Corollary 26.8]{Man74}. We summarize this result in Proposition \ref{transitief}. \\
Monochromatic cliques of color 0 are orthogonal sets, and their orbits under the action of $W$ are described in \cite[Corollary 3.3]{DynMin10}. We describe the action of $W$ on the \textsl{ordered sequences} of orthogonal roots in Proposition \ref{orbitstuplesgamma0}. \\
Finally, in \cite{CRS04} the authors give a classification of isomorphism types of all maximal \textsl{exceptional} graphs (i.e. connected graphs with least eigenvalue greater or equal to $-2$ that are not generalized line graphs \cite[Section 1.1]{CRS04}). From \cite[Corollary 3.6.4]{CRS04} it follows that these graphs correspond exactly to the maximal cliques in $\Graph{0,1}$. Therefore our classification of isomorphism types of cliques of Type IV for $c=\{0,1\}$ (see Appendices \ref{list} and \ref{29in01}) coincides with the classification of isomorphism types of maximal exceptional graphs in \cite[table p.140 and Table A6]{CRS04}; see Remark \ref{vergelijken} for a comparison between our method and the one in \cite{CRS04}. However, the classification of the isomorphism types is only part of our results on the maximal cliques in $\Graph{0,1}$. We also give invariants for such a clique that determine its isomorphism type, and we show that each isomorphism class is a full orbit under the action of the Weyl group (Propositions \ref{hard} and \ref{29cliques}). Moreover, in Corollary \ref{thm2max01} we show that every isomorphism between two representations of exceptional graphs in $E_8$ extends to an automorphism of $E_8$.
\end{remark}

Our inspiration to study the $E_8$ root system and the cliques in $\Gamma$ is the connection to del Pezzo surfaces of degree one. Such surfaces have exactly 240 lines, and there is a bijection between these lines and a root system that is isomorphic to $E_8$. We have found the maximal number of lines on these surfaces that go through one point using the results in this paper, see \cite{vLW}. This led us to studying cliques in the colored intersection graph on these lines (which is isomorphic to $\Gamma$). A good reference for these surfaces and their lines is \cite[Chapter IV]{Man74}. In Remarks \ref{dp11}, \ref{dp13}, \ref{dp14}, \ref{dp15}, and \ref{dp12}, we explain how some of our results translate to this geometric view.

\vspace{11pt}

We split the article into chapters that deal with one or more of the types I, II, III, or~IV. Note that these four types do not exclude each other, and some results in one section may be part of a result in another section. We ordered the sections such that each section builds as much on the previous ones as possible. \\
Section \ref{Background} states all the needed definitions as well as many known results about $E_8$ and the action of the Weyl group, and the relation with del Pezzo surfaces. We also set up the notation for the rest of this article. The reader who is familiar with root systems, and with $E_8$ in particular, can skip this section. Section \ref{facetsthreecliques} contains all results on the facets of the $E_8$ root polytope, and cliques of type~III. Section~\ref{monocliques} deals with cliques of type I. Section~\ref{maximalcliques} classifies all cliques of type IV. This is the biggest section, and the only section where we use a computer for some of the results (from Section \ref{-10} onwards). The results from this section are summarized in the tables in the appendices. Finally, we prove Theorems \ref{main} and \ref{main2} in Section \ref{proofthm1}.

\vspace{11pt}

All computations were done in \texttt{magma} (\cite{MR1484478}). The code that we used can be found in~\cite{magma}. 

 \section{Background: the Weyl group and the \texorpdfstring{$E_8$}{E8} root polytope}\label{Background}

Let $\Lambda$, $E$, $\Gamma$, and $W$ be as defined in the introduction. In this section we recall some well-known results about these objects, the Weyl group, and the $E_8$ root polytope. We also make a first step in proving Theorems~\ref{main} and \ref{main2}, by showing that for two cliques of type I, II, III, or IV in $\Gamma$ that are isomorphic as colored graphs, there is a type that they both belong to (Lemma \ref{step0}).   

\vspace{11pt}

Useful references for root systems and the Weyl group are \cite[Chapter 6]{Bo81}, and \cite[Chapter III]{Hum72}. 

\vspace{11pt}

The subgroup of the isometry group of $\mathbb{R}^8$ that is generated by the reflections in the hyperplanes orthogonal to the roots in $E$ is called the Weyl group, and denoted by $W_8$. This group permutes the elements in $E$, and since these roots span $\mathbb{R}^8$, the action of $W_8$ on $E$ is faithful. The Weyl group is therefore finite: it has order $696729600=2^{14}\cdot3^5\cdot5^2\cdot7$. It is equal to the automorphism group of the $E_8$ root system \cite[Section~12.2]{Hum72}, hence also to the automorphism group of the root lattice~$\Lambda$, and to the group~$W$.

\begin{lemma}\label{weyltranse8} The Weyl group acts transitively on the $E_8$ root system.\end{lemma}
\begin{proof}\cite[Section 10.4, Lemma C]{Hum72}.\end{proof}

Note that the roots in $E$ are of two types. Either they are of the form $\left(\pm\tfrac{1}{2},\ldots,\pm\tfrac{1}{2}\right)$, where an even number of entries is negative (giving $2^7=128$ roots), or exactly two entries are non-zero, and they can independently be chosen to be $-1$ or $1$ (giving $4\cdot\binom{8}{2}=112$ roots).

\begin{proposition}\label{intersection}
The absolute value of the dot product of any two elements in~$E$ is at most 2. Let $e\in E$ be a root. Then $e$ has dot product 2 only with itself, and dot product $-2$ only with its inverse $-e$. There are exactly 56 roots $f\in E$ with $e\cdot f=1$, there are exactly 56 roots $g\in E$ with $e\cdot g=-1$, and there are exactly 126 roots in $E$ that are orthogonal to $e$. 
\end{proposition}
\begin{proof}
From Cauchy-Schwarz it follows that for $e,e'\in E$ we have $$|e\cdot e'|\leq \|e\|\cdot\|e'\|=2,$$ and equality holds if and only if $e,e'$ are scalar multiples of each other. Since all roots are primitive, it follows that $e\cdot e'=2$ if and only if $e=e'$, and $e\cdot e'=-2$ if and only if $e=-e'$.  
Since $W$ acts transitively on $E$ (Lemma \ref{weyltranse8}), to count the other cases it suffices to prove this for one element in $E$. Take $e=(1,1,0,0,0,0,0,0)\in E$. \\
The roots $f\in E$ with $e\cdot f=1$ are of the form $f=(a_1,\ldots,a_8)$ with $a_1+a_2=1$. So for these roots we either have $a_1=a_2=\tfrac{1}{2}$, which gives $32$ different roots, or $\{a_1,a_2\}=\{0,1\}$, which gives $24$ different roots. This gives a total of $56$ roots.\\
For $f\in E$, we have $e\cdot f=1$ if and only if $e\cdot-f=-1$, so this gives also 56 roots $g\in E$ with $e\cdot g=-1$.\\
The roots in $E$ orthogonal to $e$ are of the form $f=(a_1,\ldots,a_8)$ with $a_1+a_2=0$. So for these roots we have $a_1=a_2=0$, which gives 60 roots, or $\{a_1,a_2\}=\{-1,1\}$, which gives 2 roots, or $\{a_1,a_2\}=\left\{-\tfrac{1}{2},\tfrac{1}{2}\right\}$, which gives 64 roots. This gives a total of $126$ roots.
\end{proof}

We continue with results on the $E_8$ root polytope. Coxeter described all faces of the~$E_8$ root polytope, which he called the $4_{21}$ polytope, in \cite{Cox30}. The faces come in two types: $k$-simplices (for $k\leq7$), given by $k+1$ vertices with angle $\frac{\pi}{3}$ and distance~$\sqrt{2}$ between any two of them, and $k$-crosspolytopes (for $k=7$), given by~$2k$ vertices where every vertex is orthogonal to exactly one other vertex, and has angle $\frac{\pi}{3}$ and distance $\sqrt{2}$ with all the other vertices. We summarize his results in Propositions \ref{facesform} and \ref{facets}. 

\begin{lemma}\label{sqrt2}
Two vertices in the $E_8$ root polytope have distance $\sqrt{2}$ between them if and only if their dot product is one. 
\end{lemma}
\begin{proof}For $e,f\in E$ we have $\|e-f\|^2=e^2-2\cdot e\cdot f+f^2=4-2\cdot e\cdot f.
$\end{proof}

\begin{proposition}\label{facesform}
For $k\leq 7$, the set of $k$-simplices in the $E_8$ root polytope is given by $$\{\{e_1,\ldots,e_{k+1}\}\;|\:\forall i:e_i\in E;\;\forall j\neq i:e_i\cdot e_j=1\},$$ where a $k$-simplex is identified with the set of its vertices. For $k\leq 6$, the $k$-simplices in the $E_8$ root polytope are exactly its $k$-faces.
\end{proposition}
\begin{proof}
The vertices in a $k$-simplex have dot product 1 by the previous lemma. The fact that the $k$-faces are exactly the $k$-simplices for $k\leq 6$ is in \cite[Section~7.5 or the table on page 414]{Cox30}. 
\end{proof}

\begin{proposition}\label{facets}
The facets of the $E_8$ root polytope are exactly the $7$-simplices and the $7$-crosspolytopes contained in it. 
The set of $7$-crosspolytopes is given by 
$$\left\{\left\{\{e_1,f_1\},\ldots,\{e_{7},f_{7}\}\right\}\;\left|\:\begin{array}{l}\forall i\in\{1,\ldots,7\}: e_i,f_i\in E;\;e_i\cdot f_i=0;\\
\forall j\neq i: e_i\cdot e_j=e_i\cdot f_j=f_i\cdot f_j=1.
\end{array}\right.\right\},$$ where a $7$-crosspolytope is identified by the set of its 7 pairs of orthogonal roots.
\end{proposition}
\begin{proof}
The facets are the $7$-simplices and the $7$-crosspolytopes by \cite[Section~7.5 or the table on page 414]{Cox30}.  The dot products follow from Lemma \ref{sqrt2}. 
\end{proof}

\begin{remark}
We also show that the $7$-simplices and the $7$-crosspolytopes in the~$E_8$ root polytope are facets in Remarks \ref{facetsimplexhyperplane} and \ref{facetcrosspolytopehyperplane}.
\end{remark}

\begin{corollary}\label{numberfacets}
The $E_8$ root polytope has 6720 1-faces, 60480 2-faces, 241920 3-faces, 483840 4-faces, 483840 5-faces, 207360 6-faces, 17280 $7$-simplices, and 2160 \\
$7$-crosspolytopes. 
\end{corollary}
\begin{proof}
See \cite[p.414]{Cox30}.
\end{proof}

\begin{remarkdp1}\label{dp11}
Let us give a quick analogy with geometry, which was our motivation to study the $E_8$ root lattice. More on this can be found in \cite[Chapter IV]{Man74}, and in a lot more detail than sketched here. \\
Let $X$ be a del Pezzo surface of degree one over an algebraically closed field $k$. Then~$X$ is isomorphic to the blow up of $\mathbb{P}^2_k$ in eight points in general position (meaning no three on a line, no six on a conic, and no eight on a cubic that is singular at one of them). Let $K_X$ be the class in Pic~$X$ of the anticanonical divisor of $X$, and let $K_X^{\perp}$ be the orthogonal complement of $K_X$ in the lattice Pic~$X$. Let $\left(\mathbb{R}\otimes K_X^{\perp},\langle\cdot,\cdot\rangle\right)$ be the Euclidean vector space with inner product $\langle\cdot,\cdot\rangle$ defined by the \textsl{negative} of the intersection pairing in Pic~$X$. Classes in $K_X^{\perp}$ with self intersection~$-2$ (so inner product 2 in $\mathbb{R}\otimes K_X^{\perp}$) form a root system within this vector space, and this root system is isomorphic to $E_8$. \\
It is well known that Pic~$X$ contains 240 classes $c$ with $c^2=c\cdot K_X=-1$, called exceptional classes. Let $C$ be the set of exceptional classes in Pic~$X$. For $c\in C$ we have $c+K_X\in K_X^{\perp}$ and $\langle c+K_X,c+K_X\rangle=2$, and this gives a bijection between~$C$ and the root system in $\mathbb{R}\otimes K_X^{\perp}$, such that $\langle c_1+K_X,c_2+K_X\rangle=1-c_1\cdot c_2$. Therefore the group of permutations of $C$ that preserves the intersection multiplicity is isomorphic to the Weyl group $W_8$. Moreover, studying the colored intersection graph of $C$, where colors are given by the intersection multiplicities, is equivalent to studying the colored graph of the $E_8$ root system, where colors are given by the dot products. Throughout this article, we will remark on some of the analogies of the results for the set $C$. \\
For example, the vertices of a $k$-simplex in the $E_8$ root polytope correspond to a sequence of $k+1$ exceptional classes in $C$ that have pairwise intersection pairing~0. Moreover, for~$r$ pairwise disjoint exceptional curves $e_1,\ldots,e_r$ (for $1\leq r\leq 7$), the exceptional curves that are disjoint from $e_1,\ldots,e_r$ correspond to the exceptional curves of the del Pezzo surface of degree $r+1$ that is obtained by blowing down $e_1,\ldots,e_r$. We know the number of exceptional curves on del Pezzo surfaces \cite[Table~IV.9]{Man74}, and we can use this to compute the number of $k$-faces of the $E_8$ root polytope for~$k\leq5$. 
\end{remarkdp1}

\begin{remark}\label{aantalexckrommen}
For $k\leq5$, the statement in Corollary \ref{numberfacets} also follows from the last part of Remark~\ref{dp11} and Table (IV.9) in \cite{Man74}: we have $$\frac{240\cdot56}{2}=6720,\;\;\frac{240\cdot56\cdot27}{3!}=60480,\;\;\frac{240\cdot56\cdot27\cdot16}{4!}=241920,$$ and so on. For $k$ equal to 6 and for the 7-simplices, the statement is in Proposition~\ref{max178}. For the 7-crosspolytopes it follows from Lemma \ref{sevenfaces}, see Remark~\ref{numbercrosspolytopes}.
\end{remark}

The following propositions state results about the action of the Weyl group on the faces of the $E_8$ root polytope.

\begin{proposition}\label{orbits67faces}
For $k\leq5$, the group $W$ acts transitively on the set of $k$-faces. There are two orbits of facets. 
\end{proposition}
\begin{proof} In \cite[Section 7.5]{Cox30} it is shown that all $k$-simplices are conjugate for $k\leq 5$, and that any two facets of the same type are conjugate as well. We know that there are two types of facets from Proposition \ref{facets}. 
\end{proof}

\begin{remark}There are two orbits of $6$-faces, which we describe in Proposition~\ref{max178}. See also Remark \ref{6-faces}.
\end{remark} 

We know something even stronger, namely, the action of $W$ on the \textsl{ordered sequences} of roots in faces of the $E_8$ root polytope.

\begin{proposition}\label{transitief}For all $r\leq8$ such that $r\neq 7$, the group $W$ acts transitively on the set $$R_r=\{(e_1,\ldots,e_r)\in E^s\;|\;\forall i\neq j:e_i\cdot e_j=1\}.$$
For $r=7$, there are two orbits under the action of $W$.
\end{proposition}
\begin{proof} In Remark~\ref{dp11} we describe a bijection between $E$ and the set $C$ of 240 exceptional curves on a del Pezzo surface of degree one, where two elements in $E$ have dot product $a$ if and only if the two corresponding elements in $C$ have intersection product $1-a$. This bijection respects the action of $W$, and under this bijection the set $R_r$ corresponds to the set of sequences of length $r$ of disjoint exceptional curves. The statement now follows from \cite[Corollary 26.8]{Man74}.
\end{proof}

The following lemma is the first step in proving Theorems \ref{main} and \ref{main2}.

\begin{lemma}\label{step0}
Let $K_1,K_2$ be two cliques in $\Gamma$ of type I, II, III, or IV that are isomorphic. Then there is a type I, II, III, or IV that they both belong to. 
\end{lemma}
\begin{proof}
If a clique is of type I or III, then any clique that is isomorphic to it is of the same type. If $K_1$ is of type II, then its vertices form a $k$-simplex (for $k\leq7$) or a $k$-crosspolytope (for $k=7$) by Proposition \ref{facesform} and Proposition \ref{facets}. In both cases,~$K_2$ is of the same type, again by Proposition \ref{facesform} and Proposition \ref{facets}. Analogously, if~$K_2$ is of type II then so is $K_1$. Finally, if $K_1$ and $K_2$ are both not of types I, II, or III, then they are automatically both of type~IV.  
\end{proof}

We conclude this section by stating a lemma that will be used throughout this article.

\begin{lemma}\label{action}
Let $H$ be a group, let $A,B$ be $H$-sets, and $f\colon A\longrightarrow B$ a morphism of $H$-sets. Then the following hold.
\begin{itemize}
\item[](i) If $H$ acts transitively on $A$, then $H$ acts transitively on $f(A)$.
\item[](ii) If $H$ acts transitively on $B$, then all fibers of $f$ have the same cardinality.
\item[](iii) If $H$ acts transitively on $A$ and $A$ is finite, then all non-empty fibers of $f$ have the same cardinality, say $n$, and $|f(A)|=\frac{|A|}{n}$.
\item[](iv) If $H$ acts transitively on $f(A)$, and there is an element $b\in f(A)$ such that its stabilizer $H_b$ acts transitively on $f^{-1}(b)$, then $f$ acts transitively on $A$. 
\end{itemize}
\end{lemma}
\begin{proof}
\begin{itemize}
\item[] 
\item[](i) Take $f(a),f(a')\in f(A)$ with $a,a'\in A$. Assume that $H$ acts transitively on $A$, then there is an $h\in H$ such that $ha=a'$. Since $f$ is a morphism of $H$-sets, we have $hf(a)=f(ha)=f(a')$, so $H$ acts transitively on $f(A)$.  
\item[](ii) Take $b,b'\in B$. Since $H$ acts transitively on $B$, there is an $h\in H$ such that $hb=b'$, so $|f^{-1}(b')|=|f^{-1}(hb)|=|hf^{-1}(b)|=|f^{-1}(b)|$.
\item[](iii) Take $b,b'\in B$ such that $f^{-1}(b)$ and $f^{-1}(b')$ are non-empty. Then we have $b,b'\in f(A)$. Since $H$ acts transitively on $f(A)$ by $(i)$, it follows from $(ii)$ that $f^{-1}(b)$ and $f^{-1}(b')$ have the same cardinality, say $n$. It is now immediate that $|A|=|f^{-1}(B)|=\sum_{b\in f(A)}n=n|f(A)|$, so $|f(A)|=\frac{|A|}{n}$.
\item[] (iv) Take $b\in f(A)$ such that $H_b$ acts transitively on $f^{-1}(b)$. Take $a,a'\in A$. Since $H$ acts transitively on $f(A)$, there are $h,h'\in H$ such that $hf(a)=b$ and $h'f(a')=b$. Then $ha$ and $h'a'$ are contained in $f^{-1}(b)$. Since $H_b$ acts transitively on $f^{-1}(b)$, there is an element $g\in H_b$ with $gha=h'a'$. So we have $h'^{-1}gha=a'$ and $H$ acts transitively on $A$.\qedhere
\end{itemize}
\end{proof}

\section{Facets of the \texorpdfstring{$E_8$}{E8} root polytope and cliques of size at most three}\label{facetsthreecliques}

In this section we study the cliques in $\Gamma$ of type III, and the facets of the $E_8$ root polytope. We give an alternative proof for the fact that $W$ acts transitively on the set of facets that are $7$-crosspolytopes (Corollary~\ref{transsevenface}), and we prove the following propositions.

\begin{proposition}\label{A} For $a\in\{\pm1,-2,0\}$, the group $W$ acts transitively on the set $$\{(e_1,e_2)\in E^2\;|\;e_1\cdot e_2=a\}.$$
\end{proposition}

\begin{proposition}\label{B}
For $a,b,c\in\{-2,-1,0,1\}$, the group $W$ acts transitively on the set $$\{(e_1,e_2,e_3)\in E^3\;|\;e_1\cdot e_2=a,\;e_2\cdot e_3=b,\;e_1\cdot e_3=c\},$$ in all cases where it is not empty.
\end{proposition}

Note that these two propositions describe the orbits under the action of $W$ of \textsl{sequences} of the vertices of cliques in $\Gamma$, hence they also prove Theorem \ref{main2} for cliques of type III; see Corollary \ref{corB}. The proof of Proposition \ref{A} can be found below Proposition~\ref{pair}, and the proof of Proposition~\ref{B} below Lemma \ref{driehoek}. Throughout this section we do not use any computer programs. More background on the $E_8$ root polytope can be found in \cite{Cox30} and \cite{Cox48}. 

\vspace{11pt}

We start with some results on the facets of the $E_8$ root polytope that are $7$-simplices. The results on the facets that are $7$-crosspolytopes are in Lemmas \ref{transsevenface} and~\ref{sevengenerate}. Consider the set $$U=\{(e_1,e_2,e_3,e_4,e_5,e_6,e_7,e_8)\in E^8\;|\;\forall i\neq j:e_i\cdot e_j=1\}.\label{facet1}$$ 
Note that an element in $U$ is a sequence of eight roots that form a $7$-simplex. Define the following roots, and note that $(u_1,\ldots,u_8)$ is an element in $U$.
\begin{align*}
&u_1=(1,1,0,0,0,0,0,0);\; & &u_5=(1,0,0,0,0,1,0,0);\\
&u_2=(1,0,1,0,0,0,0,0);\; & &u_6=(1,0,0,0,0,0,1,0);\\
&u_3=(1,0,0,1,0,0,0,0);\; & &u_7=(1,0,0,0,0,0,0,1);\\
&u_4=(1,0,0,0,1,0,0,0);\; & &u_8=\left(\tfrac{1}{2},\tfrac{1}{2},\tfrac{1}{2},\tfrac{1}{2},\tfrac{1}{2},\tfrac{1}{2},\tfrac{1}{2},\tfrac{1}{2}\right).\label{utjes}
\end{align*}

\begin{lemma}\label{voortbrengen1}
Every element in $U$ generates a sublattice of index 3 of the root lattice~$\Lambda$, and the group $W$ acts freely on $U$.
\end{lemma}
\begin{proof}
By Proposition \ref{transitief}, it is enough to check the first statement for one element in~$U$. The matrix whose $i$-th row is $u_i$ for $i\in\{1,\ldots,8\}$ has determinant 3, so $u_1,\ldots,u_8$ are linearly independent and generate a sublattice of rank 8 and index 3 in $\Lambda$. Take $w\in W$ such that there is an element $u\in U$ with $w(u)=u$. Then $w$ fixes the sublattice generated by $u$, so for all $x\in\Lambda$ we have $3w(x)=w(3x)=3x$. Since~$\Lambda$ is torsion free, this implies that $w$ fixes all of $\Lambda$. It follows that $w$ is the identity. We conclude that the action of $W$ on $U$ is free.
\end{proof}

\begin{corollary}\label{normal} Let $u=(e_1,\ldots,e_8)$ be an element in $U$. Then $\frac{1}{3}\sum_{i=1}^8e_i$ is contained in~$\Lambda$. 
%Moreover, there is a bijection $$\{\{e_1,\ldots,e_8\}\;|\;\forall i:e_i\in E;\;\forall i\neq j:e_i\cdot e_j=1\}\leftrightarrow\{\in\Lambda\;|\;\|v\|=8,\;v\mbox{ is primitive}\},$$ $$\{e_1,\ldots,e_8\}\longmapsto \frac{1}{3}\sum_{i=1}^8e_i,$$ $$v\longmapsto \{e\in E\;|\;e\cdot v=3\}.$$ 
\end{corollary}
\begin{proof}By Lemma \ref{voortbrengen1}, we know that the roots $e_1,\ldots,e_8$ generate a lattice $M$ of index 3 in $\Lambda$. Set $v=\frac13\sum_{i=1}^8e_i$. Since $v\cdot e_i=3$ for $i\in\{1,\ldots,8\}$, we have $\frac{1}{3}v\in M^{\vee}$, where $M^{\vee}$ is the dual lattice of $M$. But the dual lattice $\Lambda^{\vee}$ has index 3 in $M^{\vee}$, so it follows that $3\cdot\frac{1}{3}v=v$ is contained in $\Lambda^{\vee}$. Since $\Lambda$ is unimodular, it is self dual, so $v$ is contained in $\Lambda$. 
\end{proof}

\begin{remarkdp1}\label{dp13}Let $X$ be a del Pezzo surface of degree 1 and $K_X$ its canonical divisor, see Remark \ref{dp11}. Lemma \ref{voortbrengen1} can be stated in terms of $X$ as follows. For every set of eight exceptional classes $c_1,\ldots,c_8$ that have pairwise intersection pairing 0, there exists a unique class $l$ such that we have $K_X=-3l+\sum_{i=1}^8c_i$ and $(l,c_1,\ldots,c_8)$ is a basis for Pic~$X$; one can blow down the exceptional curves corresponding to $c_1,\ldots,c_8$ to eight points in $\mathbb{P}^2$, such that $l$ is the class of the pullback of a line in $\mathbb{P}^2$ that does not contain any of these eight points.\end{remarkdp1}

\begin{remark} \label{facetsimplexhyperplane}Let $u=(e_1,\ldots,e_8)$ be an element in $U$. We know that $e_1,\ldots,e_8$ define a facet of the $E_8$ root polytope. This also follows from from Corollary \ref{normal}. Indeed, for $v=\frac{1}{3}\sum_{i=1}^8e_i$ we have $v\cdot e_i=3$ for $i\in\{1,\ldots,8\}$, and we have $$v\cdot e=\frac13\sum_{i=1}^8e_i\cdot e\leq\frac13\sum_{i=1}^81=\frac83<3$$ for $e\in E\setminus\{e_1,\ldots,e_8\}$.  This implies that the whole $E_8$ root polytope lies on one side of the hyperplane given by $v\cdot x=3$, and the intersection of the polytope with this hyperplane, which is exactly given by the convex combinations of $e_1,\ldots,e_8$, lies in the boundary of the polytope. Hence $e_1,\ldots,e_8$ generate a facet of the $E_8$ root polytope, and $v$ is the normal vector to this facet. \end{remark}

We will now prove part of Proposition \ref{A}.

\begin{lemma}\label{twosets}
For any $a\in \{-2,\pm1\}$, the group $W$ acts transitively on the set $$A_a=\{(e_1,e_2)\in E^2\;|\;e_1\cdot e_2=a\}.$$
\end{lemma}
\begin{proof}
The group $W$ acts transitively on $A_1$ by Proposition \ref{transitief}. There is a bijection between the $W$-sets $A_1$ and $A_{-1}$ given by $$f\colon A_1\longrightarrow A_{-1},\;(e_1,e_2)\longmapsto (e_1,-e_2).$$ It follows from Lemma \ref{action} that $W$ acts transitively on $A_{-1}$, too. Finally, we have a bijection $$E\longrightarrow A_{-2},\;e\longmapsto (e,-e),$$ so $W$ acts transitively on $A_{-2}$ by Proposition \ref{transitief} and by Lemma \ref{action}.
\end{proof}

\vspace{11pt}

Before we prove the rest of Proposition \ref{A}, we prove Proposition \ref{B} for the cases $(a,b,c)=(-1,-1,-1)$ (Corollary \ref{trans-1-1-1}) and $(a,b,c)=(0,0,1)$ (Lemma \ref{lemma}), which we will use later. 

\begin{lemma}\label{unique-1-1-1}
For $e_1,e_2\in E$ with $e_1\cdot e_2=-1$ there is a unique element $e\in E$ with $e\cdot e_1=e\cdot e_2~=~-1$, which is given by $e=-e_1-e_2$.
\end{lemma}
\begin{proof}
Take $e_1,e_2,e\in E$ with $e_1\cdot e_2=-1$ and $e\cdot e_1=e\cdot e_2~=~-1$. Set $f=e_1+e_2+e$. Then we have $\|f\|=0$, hence $f=0$, so $e=-e_1-e_2$. Therefore $e$ is unique if it exists. Moreover, we have $\|-e_1-e_2\|=\sqrt{2}$, so $-e_1-e_2$ is an element in $E$ that satisfies the conditions.
\end{proof}

\begin{corollary}\label{trans-1-1-1}
The group $W$ acts transitively on the $W$-set $$\{(e_1,e_2,e_3)\in E^3\;|\;e_1\cdot e_2=e_2\cdot e_3=e_1\cdot e_3=-1\}.$$
\end{corollary}
\begin{proof}
By Lemma \ref{unique-1-1-1} there is a bijection between the sets $$\{(e_1,e_2)\in E^2\;|\;e_1\cdot e_2=-1\}$$ and $$\{(e_1,e_2,e_3)\in E^3\;|\;e_1\cdot e_2=e_2\cdot e_3=e_1\cdot e_3=-1\},$$ given by $(e_1,e_2)\longmapsto(e_1,e_2,-e_1-e_2)$. The statement now follows from Lemma \ref{twosets} and Lemma \ref{action}.
\end{proof}

\begin{lemma}\label{intersection2}
Take $e_1,\;e_2\in E$ such that $e_1\cdot e_2=1$. Then there are exactly 72 elements of $E$ orthogonal to $e_1$ and $e_2$.
\end{lemma}
\begin{proof} By Lemma \ref{twosets} it is enough to check this for fixed $e_1,e_2\in E$ with $e_1\cdot e_2=1$. Set $e_1=\left(1,1,0,0,0,0,0,0\right)$, $e_2=\left(1,0,1,0,0,0,0,0\right)$. Then $e_1\cdot e_2=1$. An element $f\in E$ with $f\cdot e_1=f\cdot e_2=0$ is of the form $f=\left(a_1,\ldots,a_8\right)$ with $a_1+a_2=0$ and $a_1+a_3=0$, hence $a_1=-a_2$ and $a_2=a_3$. If $f$ is of the form $\left(\pm\tfrac{1}{2},\ldots,\pm\tfrac{1}{2}\right)$, then there are $32$ such possibilities. If $f$ has two non-zero entries, given by~$\pm1$, then $a_1,a_2,a_3$ should all be zero, which gives $40$ possibilities. We find a total of $72$ possibilities for $f$.  
\end{proof}

\begin{lemma}\label{lemma}Consider the set 
$$B=\{(e_1,e_2,e_3)\in E^3\;|\;e_1\cdot e_2=e_2\cdot e_3=0;\;e_1\cdot e_3=1\}.$$ We have $|B|=967680$, and the following hold.
\begin{itemize}
\item[](i) The group $W$ acts transitively on $B$.
\item[](ii) For every element $b=(e_1,e_2,e_3)\in B$, there are exactly 6 roots that have dot product 1 with $e_1,e_2$ and $e_3$. These 6 roots, together with $e_1$ and $e_3$, form a facet in the set $U$.
\end{itemize}
\end{lemma}
\begin{proof}From Proposition \ref{intersection} and Lemma \ref{intersection2} we have $$|B|=240\cdot56\cdot72=967680.$$ Set $e_1=(1,1,0,0,0,0,0,0),\;e_2=(0,0,1,1,0,0,0,0),$ and $e_3=(1,0,0,0,1,0,0,0)$. Then $b=(e_1,e_2,e_3)$ is an element in $B$. Let $W_b$ be its stabilizer in $W$ and $Wb$ its orbit in $B$. Let $U_{b}$ be the set $$U_{b}=\{e\in E\;|\;e\cdot e_1=e\cdot e_2=e\cdot e_3=1\}.$$ For an element $e=(a_1,\ldots,a_8)\in U_b$, we have $a_1+a_2=a_3+a_4=a_1+a_5=1$. From this we find 
$$U_b=\left\{\begin{array}{c}(1,0,0,1,0,0,0,0)\\(1,0,1,0,0,0,0,0)\\\left(\tfrac{1}{2},\tfrac{1}{2},\tfrac{1}{2},\tfrac{1}{2},\tfrac{1}{2},\tfrac{1}{2},\tfrac{1}{2}\tfrac{1}{2}\right)\\\left(\tfrac{1}{2},\tfrac{1}{2},\tfrac{1}{2},\tfrac{1}{2},\tfrac{1}{2},-\tfrac{1}{2},-\tfrac{1}{2},\tfrac{1}{2}\right)\\\left(\tfrac{1}{2},\tfrac{1}{2},\tfrac{1}{2},\tfrac{1}{2},\tfrac{1}{2},-\tfrac{1}{2}\tfrac{1}{2},-\tfrac{1}{2}\right)\\\left(\tfrac{1}{2},\tfrac{1}{2},\tfrac{1}{2},\tfrac{1}{2},\tfrac{1}{2},\tfrac{1}{2},-\tfrac{1}{2},-\tfrac{1}{2}\right)\\\end{array}\right\}.$$ We conclude that there are 6 roots that have dot product 1 with $e_1,e_2$, and $e_3$. It is obvious that these 6 elements, together with $e_1$ and $e_3$, form an element of the set~$U$ that is defined above Lemma \ref{voortbrengen1}.\\
We have $\frac{|W|}{|W_{b}|}=|Wb|\leq|B|$. We want to show that the latter is an equality. 
The group $W_b$ acts on $U_b$. Let $w$ be an element of $W_b$ that fixes all the roots in $U_b$. Since the roots in $\{e_1,e_3\}\cup U_b$ form an element in $U$, by Lemma \ref{voortbrengen1} this implies that $w$ is the identity. Therefore the action of $W_b$ on $U_b$ is faithful. This implies that $W_b$ injects into $S_6$, so $|W_b|\leq720$. We now have $$967680=\frac{|W|}{720}\leq\frac{|W|}{|W_{b}|}=|Wb|\leq|B|=967680,$$ so we have equality everywhere and therefore we have $Wb=B$. We conclude that~$W$ acts transitively on $B$, proving (i). Part (ii) clearly holds for the element $b$, and from part (i) it follows that it holds for all elements in $B$.
\end{proof}

We proceed to prove the rest of Proposition~\ref{A}. 

\begin{lemma}\label{32}
For $e_1=\left(1,1,0,0,0,0,0,0\right),e_2=\left(0,0,1,1,0,0,0,0\right)\in E$, there are 32 elements $e$ in $E$ such that $e\cdot e_1=0$ and $e\cdot e_2=1$. 
\end{lemma}
\begin{proof}Take $e\in E$ with $e\cdot e_1=0$ and $e\cdot e_2=1$. Then we have $e=\left(a_1,a_2,a_3,a_4,\ldots,a_8\right)$ with $a_1+a_2=0$ and $a_3+a_4=1$. If $e$ is of the form $\left(\pm\tfrac{1}{2},\ldots,\pm\tfrac{1}{2}\right)$, then $a_1=-a_2$ and $a_3=a_4=\tfrac{1}{2}$. There are $16$ such possibilities. If $e$ has two non-zero entries given by $\pm1$, then either $a_3=1,\;a_1=a_2=a_4=0$, or $a_4=1,\;a_1=a_2=a_3=0$. This gives $16$ possibilities. We find a total of $32$ possibilities for $e$. 
\end{proof}

\begin{proposition}\label{pair}
The group $W$ acts transitively on the set $$A_0=\{(e_1,e_2)\in E^2\;|\;e_1\cdot e_2=0\}.$$
\end{proposition}
\begin{proof}
Consider the set $B'=\{(e_1,e_2,e_3)\in E^3\;|\;e_1\cdot e_2=e_1\cdot e_3=0;\;e_2\cdot e_3=1\}.$ Note that there is a bijection between the $W$-set $B'$ and the $W$-set $B$ in Lemma~\ref{lemma}, given by $(e,f,g)\longmapsto (f,e,g)$. Therefore, the group $W$ acts transitively on $B'$ and we have $|B'|=967680$ by Lemma \ref{lemma}. We have a projection $\lambda\colon B'\longrightarrow A_0$ to the first two coordinates. We show that $\lambda$ is surjective. Fix the roots $e_1=\left(1,1,0,0,0,0,0,0\right)$ and $e_2=\left(0,0,1,1,0,0,0,0\right)$ in $E$. Then $(e_1,e_2)$ is an element of $A_0$. Take $e\in E$, then $(e_1,e_2,e)$ is in $B'$ if and only if $e\cdot e_1=0$ and $e\cdot e_2=1$. By Lemma \ref{32} this gives 32 possibilities for $e$,  so $|\lambda^{-1}((e_1,e_2))|=32$. Since $W$ acts transitively on~$B'$, it follows from Lemma \ref{action} that all non-empty fibers of $\lambda$ have cardinality 32, and $|\lambda(B')|=\frac{|B'|}{32}=30240$. By Proposition \ref{intersection} we have $|A_0|=240\cdot 126=30240$. We conclude that $\lambda(B')=A_0$. Hence $\lambda$ is surjective. Therefore, the group $W$ acts transitively on $A_0$ by Lemma~\ref{action}. 
\end{proof}

\begin{proofofpropositionA}\label{prp1}
This follows from Proposition \ref{pair} together with Lemma~\ref{twosets}. 
\end{proofofpropositionA}

Before we continue proving Proposition \ref{B}, we complete our study of the facets of the $E_8$ root polytope. Define the set
$$C=\left\{\left\{\{e_1,f_1\},\ldots,\{e_7,f_7\}\right\}\;\left|\:\begin{array}{l}\forall i\in\{1,\ldots,7\}: e_i,f_i\in E;\;e_i\cdot f_i=0;\\
\forall j\neq i: e_i\cdot e_j=e_i\cdot f_j=f_i\cdot f_j=1.
\end{array}\right.\right\}.$$\label{defC}

Elements in $C$ are facets that are $7$-crosspolytopes by Proposition \ref{facesform}. 
We define the following elements $c_{1},\ldots,c_7,d_1,\ldots,d_{7}$. Note that $\{\{c_1,d_1\},\ldots,\{c_7,d_7\}\}$ is an element in $C$. 

\begin{align*}
&c_1=\left(1,1,0,0,0,0,0,0\right), & &d_1=(0,0,1,1,0,0,0,0),\\ 
&c_2=\left(1,0,1,0,0,0,0,0\right), & &d_2=(0,1,0,1,0,0,0,0),\\
&c_3=(1,0,0,1,0,0,0,0), & &d_3=(0,1,1,0,0,0,0,0),\\
&c_4=\left(\tfrac{1}{2},\tfrac{1}{2},\tfrac{1}{2},\tfrac{1}{2},\tfrac{1}{2},\tfrac{1}{2},\tfrac{1}{2},\tfrac{1}{2}\right), & & d_4=\left(\tfrac{1}{2},\tfrac{1}{2},\tfrac{1}{2},\tfrac{1}{2},-\tfrac{1}{2},-\tfrac{1}{2},-\tfrac{1}{2},-\tfrac{1}{2}\right),\\
&c_5=\left(\tfrac{1}{2},\tfrac{1}{2},\tfrac{1}{2},\tfrac{1}{2},-\tfrac{1}{2},-\tfrac{1}{2},\tfrac{1}{2},\tfrac{1}{2}\right), & &d_{5}=\left(\tfrac{1}{2},\tfrac{1}{2},\tfrac{1}{2},\tfrac{1}{2},\tfrac{1}{2},\tfrac{1}{2},-\tfrac{1}{2},-\tfrac{1}{2}\right),\\
&c_{6}=\left(\tfrac{1}{2},\tfrac{1}{2},\tfrac{1}{2},\tfrac{1}{2},-\tfrac{1}{2},\tfrac{1}{2},-\tfrac{1}{2},\tfrac{1}{2}\right), & &d_{6}=\left(\tfrac{1}{2},\tfrac{1}{2},\tfrac{1}{2},\tfrac{1}{2},\tfrac{1}{2},-\tfrac{1}{2},\tfrac{1}{2},-\tfrac{1}{2}\right),\\
&c_{7}=\left(\tfrac{1}{2},\tfrac{1}{2},\tfrac{1}{2},\tfrac{1}{2},-\tfrac{1}{2},\tfrac{1}{2},\tfrac{1}{2},-\tfrac{1}{2}\right), & &d_{7}=\left(\tfrac{1}{2},\tfrac{1}{2},\tfrac{1}{2},\tfrac{1}{2},\tfrac{1}{2},-\tfrac{1}{2},-\tfrac{1}{2},\tfrac{1}{2}\right).\\ 
\end{align*}

\begin{lemma}\label{sevenfaces}
For $e_1,e_2\in E$ with $e_1\cdot e_2=0$, there are exactly 12 elements $e\in E$ with $e\cdot e_1=e\cdot e_2=1$. These 12 elements, together with $e_1$ and $e_2$, form an element in $C$, and this is the unique element in $C$ containing ${e_1,e_2}$.
\end{lemma}
\begin{proof}
By Proposition \ref{pair}, it is enough to check this for fixed $e_1,e_2\in E$ with $e_1\cdot e_2=0$. Take $e_1=c_1,\;e_2=d_1$ in $E$. For a root $e=(a_1,\ldots,a_8)$ in $E$ with $e\cdot c_1~=~e\cdot d_1=1$, we have either $a_1=a_2=a_3=a_4=\tfrac{1}{2}$, which implies $e\in\{c_4,\ldots,c_7,d_4,\ldots,d_7\}$, or $\{a_1,a_2\}=\{a_3,a_4\}=\{0,1\}$, which implies $e\in\{c_2,c_3,d_2,d_3\}$. Therefore there are exactly 12 possibilities $\{c_2,\ldots,c_7,d_2,\ldots,d_7\}$ for $e$, and we conclude that the set $\{\{c_1,d_1\},\ldots,\{c_7,d_7\}\}$ is the unique element in $C$ containing ${c_1,d_1}$.  
\end{proof}

\begin{remark}\label{numbercrosspolytopes} Since elements in $C$ correspond to $7$-crosspolytopes, we know that we have $|C|=2160$ from Corollary \ref{numberfacets}. This also follows from the previous lemma. Recall the set $A_0=\{(e_1,e_2)\in E^2\;|\;e_1\cdot e_2=0\}$. By Lemma \ref{sevenfaces}, for every element $(e_1,e_2)$ in $A_0$ there is a unique element in $C$ containing $e_1,e_2$. But every element in $C$ contains seven pairs $f_1,f_2$ such that $(f_1,f_2)$ and $(f_2,f_1)$ are in $A_0$, so the map $A_0\longrightarrow C$ is fourteen to one. Hence we have $|C|=\frac{|A_0|}{14}=\frac{240\cdot126}{14}=2160$.\end{remark}

\begin{corollary}\label{transsevenface}
The group $W$ acts transitively on $C$.
\end{corollary}
\begin{proof}
Consider the set $A_0=\{(e_1,e_2)\in E^2\;|\;e_1\cdot e_2=0\}$. By Proposition \ref{pair}, the group $W$ acts transitively on $A_0$. By Lemma \ref{sevenfaces} there is a map $A_0\longrightarrow C$, sending $(e_1,e_2)$ to the unique element in $C$ that contains $e_1$ and $e_2$. This map is clearly surjective. It follows from Lemma \ref{action} that $W$ acts transitively on $C$. 
\end{proof}

\begin{lemma}\label{sevengenerate}
Every element in $C$ generates a sublattice of finite index in~$\Lambda$.
\end{lemma}
\begin{proof}
By Corollary \ref{transsevenface}, it is enough to check this for one element in $C$. Take the element $\{\{c_1,d_1\},\ldots,\{c_7,d_7\}\}$ in $C$, where the $c_i,d_i$ are defined above Lemma \ref{sevenfaces}. The matrix whose rows are the vectors $c_1,\ldots,c_7,d_1,\ldots,d_7$ has rank 8, so these 14 elements generate a sublattice $L$ of finite index in $\Lambda$.
\end{proof}

\begin{remark}\label{facetcrosspolytopehyperplane}
Let $\{\{e_1,f_1\},\ldots,\{e_7,f_7\}\}$ be an element in $C$, and let $c$ be given by $c=\{e_1,\ldots,e_7,f_1,\ldots,f_7\}$. We know that the elements in $c$ are the vertices of a facet of the $E_8$ root polytope. We show how this also follows from the previous lemma. Take $i\in\{1,\ldots,7\}$, then we have $(e_i+f_i)\cdot e=2$ for all $e\in c$. Since the elements in $c$ generate a full rank sublattice, this implies that $e_i+f_i=e_j+f_j$ for all $i,j\in\{1,\ldots,7\}$. So the vector $n=\frac17\sum_{i=1}^7(e_i+f_i)=e_1+f_1$ is an element in $\Lambda$ with $n\cdot e=2$ for $e\in s$. Take $e\in E\setminus s$, and note that $e$ cannot have dot product 1 with both $e_1$ and $f_1$ by Lemma \ref{sevenfaces}. It follows that we have $n\cdot e<2$, so the entire $E_8$ root polytope lies on one side of the affine hyperplane given by $n\cdot x=2$. Moreover, this hyperplane intersects the $E_8$ root polytope in its boundary, and exactly in the convex combinations of the roots $e_1,\ldots,e_7,f_1,\ldots,f_7$. Therefore these roots are the vertices of a facet of the $E_8$ root polytope with normal vector $n$. \end{remark}

We continue with Proposition \ref{B}, and prove it for $(a,b,c)=(0,0,0)$. Consider the sets $$V_3=\{\left(e_1,e_2,e_3\right)\in E^3\;|\;\forall i\neq j:e_i\cdot e_j=0\}$$ and $$V_4=\{\left(e_1,e_2,e_3,e_4\right)\in E^4\;|\;\forall i\neq j:e_i\cdot e_j=0\}.$$ 

\vspace{11pt}

We begin by studying $V_4$. To this end, recall the set $U$ defined above Lemma \ref{voortbrengen1}, and define the set $$Z=\{\left(\{e_1,e_2\},\{e_3,e_4\},\{e_5,e_6\},\{e_7,e_8\}\right)\;|\;\forall i: e_i\in E;\;\forall j\neq i:e_i\cdot e_j=1\}.$$  

\begin{remark}\label{remark}
We have a surjective map $U\longrightarrow Z$ by simply forgetting the order of~$e_i$ and $e_{i+1}$ for $i\in\{1,3,5,7\}$. Since $W$ acts transitively on $U$ (Proposition \ref{transitief}), it follows from Lemma \ref{action} that $W$ acts transitively on $Z$. 
By Lemma \ref{voortbrengen1}, the action of $W$ on $U$ is free, so we have $|U|=|W|$, and $|Z|=\frac{|U|}{2^4}=\frac{|W|}{2^4}=2^{10}\cdot3^5\cdot5^2\cdot7$.
\end{remark}

We want to define a map $\alpha\colon Z\longrightarrow V_4$. To do this we use the following lemma. 

\begin{lemma}\label{mapalpha}
For an element $z=(\{e_1,e_2\},\{e_3,e_4\},\{e_5,e_6\},\{e_7,e_8\})$ in $Z$, there are unique roots $f_1,f_2,f_3,f_4\in E$ with 
\begin{align*}
f_1\cdot e_i=0,\;\; &f_1\cdot e_j=1\;\;\mbox{for }i\in\{1,2\},\;j\notin\{1,2\};\\
f_2\cdot e_i=0,\;\; &f_2\cdot e_j=1\;\;\mbox{for }i\in\{3,4\},\;j\notin\{3,4\};\\
f_3\cdot e_i=0,\;\; &f_3\cdot e_j=1\;\;\mbox{for }i\in\{5,6\},\;j\notin\{5,6\};\\
f_4\cdot e_i=0,\;\; &f_4\cdot e_j=1\;\;\mbox{for }i\in\{7,8\},\;j\notin\{7,8\}.
\end{align*}
For these $f_1,f_2,f_3,f_4$ we have $f_i\cdot f_j=0$ for $i\neq j$, and $3\sum_{i=1}^4 f_i=\sum_{i=1}^8e_i$. 
\end{lemma} 
\begin{proof}
By Lemma \ref{voortbrengen1}, the elements $e_1,\ldots,e_8$ generate a full rank sublattice of $\Lambda$, so an element $f\in E$ is uniquely determined by the intersection numbers $f\cdot e_i$ for $i$ in $\{1,\ldots,8\}$. We will show existence. Set $v=\frac{1}{3}\sum_{i=1}^8e_i$. 
By Corollary \ref{normal}, the vector $v$ is an element in $\Lambda$. We have $\|v\|=\sqrt{8}$, and $v\cdot e_i=3$ for $i\in\{1,\ldots,8\}$. For $i\in \{1,2,3,4\}$, set $f_i=v-e_{2i-1}-e_{2i}$. Then $\|f_i\|=\sqrt{2}$, so $f_i\in E$. Moreover, $f_1,f_2,f_3,f_4$ satisfy the conditions in the lemma. \end{proof}

\vspace{5pt}

We now define a map $\alpha\colon Z\longrightarrow V_4,\;(\{e_1,e_2\},\ldots,\{e_7,e_8\})\longmapsto (f_1,f_2,f_3,f_4)$, where $f_1,f_2,f_3,f_4$ are the unique elements found in Lemma \ref{mapalpha}.

\begin{corollary}\label{primitiveY}
If $(f_1,f_2,f_3,f_4)$ is an element in the image of $\alpha$, then $x=\sum_{i=1}^4f_i$ is a primitive element of $\Lambda$ with norm $\sqrt{8}$.
\end{corollary}
\begin{proof}
Take $(f_1,f_2,f_3,f_4)$ in the image of $\alpha$, and let $(\{e_1,e_2\},\ldots,\{e_7,e_8\})\in Z$ be such that $(f_1,f_2,f_3,f_4)=\alpha((\{e_1,e_2\},\ldots,\{e_7,e_8\}))$. Set $x=\sum_{i=1}^4f_i$. Then we have $3x=\sum_{i=1}^8e_i$ by Lemma \ref{mapalpha}. It follows that $\|3x\|^2=72$, hence $\|x\|^2=8$. Moreover, for any $i\in\{1,\ldots,8\}$ we have $3x\cdot e_i=9$, hence $x\cdot e_i=3$. This implies that if we have $x=m\cdot x'$ for some $m\in\mathbb{Z},\;x'\in\Lambda$, then $m|2$ and $m|3$, so $m=1$ and~$x$ is primitive.
\end{proof}

\begin{remarkdp1}\label{dp14}
Let $X$ be a del Pezzo surface of degree one over an algebraically closed field, and $C$ the set of exceptional classes in Pic~$X$. The map $\alpha$ has a nice description in the geometric setting, through the bijection $$C\longrightarrow E,\;c\longmapsto c+K_X.$$ Take $z=(\{e_1,e_2\},\{e_3,e_4\},\{e_5,e_6\},\{e_7,e_8\})$ an element in $Z$. The roots $e_1,\ldots,e_8$ correspond to classes $c_1,\ldots,c_8$ in $C$ with $c_i\cdot c_j=0$ for all $i\neq j\in\{1,\ldots,8\}$. These classes correspond to pairwise disjoint curves on $X$ that can be blown down to points $P_1,\ldots,P_8$ in $\mathbb{P}^2$ such that $c_i$ is the class of the exceptional curve above $P_i$ for $i$ in $\{1,\ldots,8\}$ (see \cite{Man74}). The conditions for $f_i$ in Lemma \ref{mapalpha} are equivalent to $f_i$ being the strict transform on $X$ of the line in $\mathbb{P}^2$ through $P_{2i-1}$ and $P_{2i}$ for $i\in\{1,2,3,4\}$. This geometrical argument immediately proves the uniqueness of $f_i$.
\end{remarkdp1}

Let $\pi\colon V_4\longrightarrow V_3$ be the projection to the first three coordinates. From the maps $\pi$ and $\alpha$, transitivity on $V_3$ will follow (Proposition~\ref{three}). Let $Y$ be the image of $\alpha$. We will show that $V_4$ has two orbits under the action of $W$, given by $Y$ and $V_4\setminus Y$ (Proposition~\ref{orbits}). The following commutative diagram shows the maps and sets that are defined. 

\begin{center}
\begin{tikzcd} \label{diagram}
U \arrow[d, two heads] && \\
Z \arrow[r, "\alpha"] \arrow[dr, two heads] &  V_4 \arrow[r,"\pi"] &V_3\\
& Y\arrow[u, hook] \arrow[ur]&
\end{tikzcd}
\end{center}

\begin{lemma}\label{cardfiber}
The map $\alpha$ is injective. 
\end{lemma}
\begin{proof}
Consider the roots in $E$ given by
\begin{align*}&f_1=\left(1,1,0,0,0,0,0,0\right),  & &f_3=\left(0,0,0,0,1,1,0,0\right),\\ &f_2=\left(0,0,1,1,0,0,0,0\right),  & &f_4=\left(1,-1,0,0,0,0,0,0\right).\end{align*} Then $v=(f_1,f_2,f_3,f_4)$ is an element in $V_4$. Let $\left(\{e_1,e_2\},\{e_3,e_4\},\{e_5,e_6\},\{e_7,e_8\}\right)$ be an element in the fiber of $\alpha$ above $v$. Then we have \begin{align}
e_1\cdot f_1=e_2\cdot f_1=0 &\mbox{ and }e_1\cdot f_i=e_2\cdot f_i=1\mbox{ for all }i\neq 1; \\
e_3\cdot f_2=e_4\cdot f_2=0 &\mbox{ and }e_3\cdot f_i=e_4\cdot f_i=1\mbox{ for all }i\neq 2; \nonumber \\
e_5\cdot f_3=e_6\cdot f_3=0 &\mbox{ and }e_5\cdot f_i=e_6\cdot f_i=1\mbox{ for all }i\neq 3; \nonumber\\
e_7\cdot f_4=e_8\cdot f_4=0 &\mbox{ and }e_7\cdot f_i=e_8\cdot f_i=1\mbox{ for all }i\neq 4. \nonumber
 \end{align}
Write $e_1=(a_1,\ldots,a_8)$. Then (1) implies that we have $a_1+a_2=0$ and $a_1-a_2=1$, and $a_3+a_4=a_5+a_6=1$. It follows that $e_1$ is equal to $\left(\tfrac{1}{2},-\tfrac{1}{2},\tfrac{1}{2},\tfrac{1}{2},\tfrac{1}{2},\tfrac{1}{2},-\tfrac{1}{2},\tfrac{1}{2}\right)$ or $\left(\tfrac{1}{2},-\tfrac{1}{2},\tfrac{1}{2},\tfrac{1}{2},\tfrac{1}{2},\tfrac{1}{2},\tfrac{1}{2},-\tfrac{1}{2}\right)$, and $e_2$ equals the other. Analogously we find:
\begin{align*}
&\left\{e_3,e_4\right\} =\left\{\left( 1, 0, 0, 0, 0, 1, 0, 0 \right),\left(1, 0, 0, 0, 1, 0, 0, 0 \right)\right\},\\
&\left\{e_5,e_6\right\}=\left\{\left(1, 0, 0, 1, 0, 0, 0, 0 \right),\left( 1, 0, 1, 0, 0, 0, 0, 0 \right)\right\},\\
&\left\{e_7,e_8\right\}=\left\{\left(\tfrac{1}{2},\tfrac{1}{2},\tfrac{1}{2},\tfrac{1}{2},\tfrac{1}{2},\tfrac{1}{2},-\tfrac{1}{2},-\tfrac{1}{2}\right),\left(\tfrac{1}{2},\tfrac{1}{2},\tfrac{1}{2},\tfrac{1}{2},\tfrac{1}{2},\tfrac{1}{2},\tfrac{1}{2},\tfrac{1}{2}\right)\right\}.\end{align*}
Hence the fiber above $v$ has cardinality one. Since $W$ acts transitively on $Z$, we conclude from Lemma \ref{action} that all non-empty fibers of $\alpha$ have cardinality one, so~$\alpha$ is injective. 
\end{proof}

\begin{remark}\label{image}By the previous lemma, there is a bijection between the set $Z$ and the set $\alpha(Z)=Y$. Since $\alpha$ is a $W$-map, it follows that $Y$ is a $W$-set, and that $W$ acts transitively on $Y$ by Lemma \ref{action}. 
\end{remark}

We state two more lemmas before we prove that $W$ acts transitively on $V_3$.

\begin{lemma}\label{emptyfibers} 
Consider the elements in $E$ given by 
\begin{align*}
&e_1=(1,1,0,0,0,0,0,0); & &f_1=(0,0,0,0,0,0,1,1)\\
&e_2=(0,0,1,1,0,0,0,0); & &f_2=(0,0,0,0,0,0,-1,-1).\\
&e_3=(0,0,0,0,1,1,0,0); & &
\end{align*}
Then $v=(e_1,e_2,e_3,f_1)$ and $v'=(e_1,e_2,e_3,f_2)$ are elements in $V_4$ that are not in $Y$. 
\end{lemma}
\begin{proof}It is easy to check that $v$ and $v'$ are in $V_4$. We have $$e_1+e_2+e_3+f_1=2\cdot\left(\tfrac{1}{2},\tfrac{1}{2},\tfrac{1}{2},\tfrac{1}{2},\tfrac{1}{2},\tfrac{1}{2},\tfrac{1}{2},\tfrac{1}{2}\right)$$ and $$e_1+e_2+e_3+f_2=2\cdot\left(\tfrac{1}{2},\tfrac{1}{2},\tfrac{1}{2},\tfrac{1}{2},\tfrac{1}{2},\tfrac{1}{2},-\tfrac{1}{2},-\tfrac{1}{2}\right),$$ hence both $e_1+e_2+e_3+f_1$ and $e_1+e_2+e_3+f_2$ are not primitive elements in $\Lambda$ and therefore not contained in $Y$ by Corollary \ref{primitiveY}.
\end{proof}

\begin{lemma}\label{60}
For two elements $e_1,e_2\in E^2$ with $e_1\cdot e_2=0$, there are exactly 60 roots $e\in E$ such that $e_1\cdot e=e_2\cdot e =0$.
\end{lemma}
\begin{proof}
By Proposition \ref{pair}, it is enough to check this for two orthogonal roots $e_1,e_2$ in $E$. Set $e_1=\left(1,1,0,0,0,0,0,0\right)$, $e_2=\left(0,0,1,1,0,0,0,0\right)$. An element $f\in E$ with $f\cdot e_1=f\cdot e_2=0$ is of the form $f=\left(a_1,a_2,a_3,\ldots,a_8\right)$ with $a_1=-a_2$ and $a_3=-a_4$. If $f$ is of the form $\left(\pm\tfrac{1}{2},\ldots,\pm\tfrac{1}{2}\right)$, then there are $32$ such possibilities. If $f$ has two non-zero entries, given by $\pm1$, then there are $28$ possibilities. We find a total of $60$ possibilities for $f$.  
\end{proof}

The following graph summarizes the results in Proposition \ref{intersection} and Lemmas \ref{unique-1-1-1}, \ref{32} and~\ref{60}. Vertices are roots, and the number in a subset is its cardinality. The number on an edge between two subsets is the dot product of two roots, one from each subset. 

\vspace{22pt}

\begin{center}
\begin{tikzpicture}\label{graph}
\node[label=above right:$e_1$,draw,circle,fill,inner sep=0pt,minimum size=4pt](E1) at (5,5){};
\node[label=above:$-e_1$,draw,circle,fill,inner sep=0pt,minimum size=4pt](minE1) at (7,5.5) {};
\node[draw,circle,inner sep=1pt] (A) at (3.8,1.4) {
    \begin{tikzpicture} 
	  \node (a1) at (0,0) {126};
      \node [label=above:$e_2$,draw,circle,fill,inner sep=0pt,minimum size=4pt](E2) at (0,-1){};
      \node [label=below:$-e_2$,draw,circle,fill,inner sep=0pt,minimum size=4pt](minE2) at (0,-2.5){};
      \node [draw,circle] (a2) at (1,-1.75) {60};
      \node [draw,circle] (a3) at (-1.2,-0.7) {32};
      \node [draw,circle] (a4) at (-1.2,-2.8) {32};
      \path[every node/.style={font=\sffamily\small}] 
       (E2) edge node [midway,above] {$0$} (a2)
       (E2) edge node [midway,right] {$-2$} (minE2)
       (E2) edge node [midway,above] {$1$} (a3)
       (minE2) edge node [midway,below] {$1$} (a4)
	   (E2) edge node [midway,below] {$-1$} (a4)
       (minE2) edge node [midway,above] {$-1$} (a3)       
       (minE2) edge node [midway,below] {$0$} (a2);
     \end{tikzpicture}};
\node[draw,circle,inner sep=0pt] (B) at (8,3.5) {
  \begin{tikzpicture} 
     \node (b1) at (0,0) {56};
     \node [label=below:$-e_3$,draw,circle,fill,inner sep=0pt,minimum size=4pt](minE3) at (-0.8,-0.8) {};
     \node [label={[xshift=0cm, yshift=-1.3 cm]$-e_1-e_3$},draw,circle,fill,inner sep=0pt,minimum size=4pt](minE1E3) at (0.8,-0.8) {};
     \path[every node/.style={font=\sffamily\small}] 
       (minE3) edge node [midway,above] {$-1$} (minE1E3);
   \end{tikzpicture}};
 \node[draw,circle,inner sep=2pt] (C) at (2,5) {
   \begin{tikzpicture} 
     \node (c1) at (0,0) {56};
     \node [label=below:$e_3$,draw,circle,fill,inner sep=0pt,minimum size=4pt](E3) at (0,-0.5) {};
   \end{tikzpicture}
    };
\path[every loop/.style={}]   
    (E1) edge [in=150,out=90,loop, above, distance= 12mm] node {2} (E1);
   \path[every node/.style={font=\sffamily\small}]
     (E1) edge node [midway,above]{$-2$} (minE1)
     (E1) edge node [midway,above]{$-1$} (B)
     (E1) edge node [midway,left]{$0$} (A)
     (E1) edge node [midway,above]{$1$} (C);
\end{tikzpicture}
\end{center}

\vspace{22pt}

\begin{proposition}\label{three}
Let $v=(f_1,f_2,f_3)$ be an element of $V_3$. The following hold.
\begin{itemize}
\item[](i) We have $|V_3|=1814400$, and the group $W$ acts transitively on $V_3$.
\item[](ii) We have $|\pi^{-1}(v)|=26$, and $|\pi^{-1}(v)\cap Y|=24$. 
\item[](iii) For $\{(f_1,f_2,f_3,u),(f_1,f_2,f_3,u')\}=\pi^{-1}(v)\setminus Y$, we have $u=-u'$, and for $(f_1,f_2,f_3,e)\in \pi^{-1}(v)\cap Y$, we have $e\cdot u=e\cdot u'=0$. 
\end{itemize}
\end{proposition}
\begin{proof}
From Proposition \ref{intersection} and Lemma \ref{60} it follows that $$|V_3|=240\cdot126\cdot60=1814400.$$ Consider the map $\lambda=\pi\circ \alpha\colon Z\rightarrow V_3$. Note that $\lambda$ is a $W$-map, since both $\pi$ and $\alpha$ are. We want to show that $\lambda$ is surjective. Set $$f_1=(1,1,0,0,0,0,0,0),\;f_2=(0,0,1,1,0,0,0,0),\;f_3=(0,0,0,0,1,1,0,0).$$ Then we have $v=(f_1,f_2,f_3)\in V_3$. Define the roots
\begin{align*}
&e_1=\left(\tfrac{1}{2},-\tfrac{1}{2},\tfrac{1}{2},\tfrac{1}{2},\tfrac{1}{2},\tfrac{1}{2},-\tfrac{1}{2},\tfrac{1}{2}\right), & &e_5=\left(1, 0, 0, 1, 0, 0, 0, 0 \right),\\
&e_2=\left(\tfrac{1}{2},-\tfrac{1}{2},\tfrac{1}{2},\tfrac{1}{2},\tfrac{1}{2},\tfrac{1}{2},\tfrac{1}{2},-\tfrac{1}{2}\right), & &e_6=\left( 1, 0, 1, 0, 0, 0, 0, 0 \right),\\
&e_3=\left( 1, 0, 0, 0, 0, 1, 0, 0 \right), & & e_7=\left(\tfrac{1}{2},\tfrac{1}{2},\tfrac{1}{2},\tfrac{1}{2},\tfrac{1}{2},\tfrac{1}{2},-\tfrac{1}{2},-\tfrac{1}{2}\right)\\
&e_4=\left(1, 0, 0, 0, 1, 0, 0, 0 \right), & &e_8=\left(\tfrac{1}{2},\tfrac{1}{2},\tfrac{1}{2},\tfrac{1}{2},\tfrac{1}{2},\tfrac{1}{2},\tfrac{1}{2},\tfrac{1}{2}\right).\end{align*} 
Note that for $i\neq j$ we have $e_i\cdot e_j=1$, so $\left(\{e_1,e_2\},\{e_3,e_4\},\{e_5,e_6\},\{e_7,e_8\}\right)$ is an element in $Z$.
We have 
\begin{align*}
f_1\cdot e_1=f_1\cdot e_2=0 &\mbox{ and }f_1\cdot e_i=1\mbox{ for all }i\not\in\{1,2\}; \\
f_2\cdot e_3=f_2\cdot e_4=0 &\mbox{ and }f_2\cdot e_i=1\mbox{ for all }i\not\in\{3,4\}; \\
f_3\cdot e_5=f_3\cdot e_6=0 &\mbox{ and }f_3\cdot e_i=1\mbox{ for all }i\not\in\{5,6\},
\end{align*}
so $\lambda\left(\left(\{e_1,e_2\},\{e_3,e_4\},\{e_5,e_6\},\{e_7,e_8\}\right)\right)=v$. Hence the fiber of $\lambda$ above $v$ is not empty, and we want to compute its cardinality. We first compute the cardinality of the fiber of $\pi$ above $v$. For an element $f=\left(a_1,\ldots,a_8\right)\in E$, we have $(f_1,f_2,f_3,f)\in$~$V_4$ if and only if $a_1+a_2=a_3+a_4=a_5+a_6~=~0$. This gives $16$ possibilities for~$f$ with $a_i\in\left\{\pm\tfrac{1}{2}\right\}$ for $i\in\{1,\ldots,8\}$, and $10$ possibilities for~$f$ where the two non-zero entries are $\pm1$. We conclude that $|\pi^{-1}(v)|=26$. Set $g_1~=~(0,0,0,0,0,0,1,1)$ and $g_2=(0,0,0,0,0,0,-1,-1)$,
then $u=(f_1,f_2,f_3,g_1)$ and $u'=(f_1,f_2,f_3,g_2)$ are both elements in $\pi^{-1}(v)$. By Lemma \ref{emptyfibers}, we know that the fibers of $\alpha$ above $u$ and $u'$ are empty. Since $\alpha$ is injective, this implies $|\lambda^{-1}(v)|\leq 24.$ Since $\lambda^{-1}(v)$ is not empty, by Lemma~\ref
{action}, we have $|\lambda(Z)|=\frac{|Z|}{|\lambda^{-1}(v)|}$. Combining this, we find $$\frac{|Z|}{24}\leq\frac{|Z|}{|\lambda^{-1}(v)|}=|\lambda(Z)|\leq |V_3|=1814400=\frac{|Z|}{24}.$$So we have equality everywhere, hence $|\lambda^{-1}(v)|=24$, and $|\lambda(Z)|=|V_3|,$ so $\lambda$ is surjective. Since $W$ acts 
transitively on $Z$, we conclude from Lemma~\ref{action} that $W$ acts transitively on $V_3$, too. This proves (i). To prove (ii), note that we showed that $|\pi^{-1}(v)|=26$ and $|\lambda^{-1}(v)|=24$, and since $\alpha$ is injective, we have the equality $|\pi^{-1}(v)\cap Y|=|\lambda^{-1}(v)|=24$. Since $\pi$ is a $W$-map, and $W$ acts transitively on $V_3$, the result holds for all elements in $V_3$. Finally, (iii) is an easy check for the element~$v$, after writing down the 26 elements in $\pi^{-1}(v)$. Since $W$ acts transitively on $V_3$, this holds for all elements in $V_3$.
\end{proof}

\begin{proposition}\label{orbits}
The set $V_4$ has two orbits under the action of $W$, which are $Y$ and $V_4\setminus Y$. We have $|Y|=43545600$ and $|V_4\setminus Y|=3628800$. An element $(e_1,\ldots,e_4)$ is in $V_4\setminus Y$ if and only if $\sum_{i=1}^4e_i\in2\Lambda$. 
\end{proposition}
\begin{proof}
From Remark \ref{image} it follows that $Y$ is an orbit under the action of $W$ on~$V_4$. Therefore $O=V_4\setminus Y$ is also a $W$-set. Let $e_1,e_2,e_3,f_1,f_2$ be as in Lemma \ref{emptyfibers} and set $v=(e_1,e_2,e_3)$, $u=(e_1,e_2,e_3,f_1)$, and $u'=(e_1,e_2,e_3,f_2)$. Consider the restriction $\pi|_O$ of $\pi$ to $O$. We have $v\in V_3$, and $u,u'\in \pi|_O^{-1}(v)$ by Lemma~\ref{emptyfibers}. From Proposition~\ref{three} we know that $|\pi^{-1}(v)\cap Y|=24$, so $\left|\pi|_O^{-1}(v)\right|=2$. This implies $\pi|_O^{-1}(v)=\{u,u'\}$. Consider the element $r$ in $W$ given by the reflection in the hyperplane that is orthogonal to $f_1$. Since $e_1,e_1,e_3$ are contained in this hyperplane, the reflection $r$ is contained in the stabilizer $W_v$ in $W$ of $v$. Moreover, since $f_2=-f_1$, the reflection $r$ interchanges $f_1$ and $f_2$, hence $W_v$ acts transitively on $\pi|_O^{-1}(v)$. Since $W$ acts transitively on $V_3$ by Proposition \ref{three}, we conclude that~$W$ acts transitively on $O$ from Lemma \ref{action}. From Proposition \ref{three} it follows that $|Y|=|V_3|\cdot24=43545600$, and $|O|=|V_3|\cdot2=3628800$. By Corollary \ref{primitiveY}, for every element $(g_1,g_2,g_3,g_4)$ in $Y$ the sum $\sum_{i=1}^4g_i$ is primitive. On the other hand, $u$ is an element in $O$, and $e_1+e_2+e_3+f_1=2\cdot\left(\tfrac{1}{2},\tfrac{1}{2},\tfrac{1}{2},\tfrac{1}{2},\tfrac{1}{2},\tfrac{1}{2},\tfrac{1}{2},\tfrac{1}{2}\right)$. So the sum of the coordinates of $u$ is an element in $2\Lambda$, and since $W$ acts transitively on $O$, this holds for every element in $O$. \end{proof}

Now that we proved that $W$ acts transitively on $V_3$, there is one last case of Proposition~\ref{B} that we prove separately (Lemma \ref{driehoek}). We state two auxiliary lemmas first.

\begin{lemma}\label{semidirect product}Let $r$ be a positive integer, and let $G$ be a graph with vertex set given by $\{v_1,\ldots,v_r,w_1\ldots,w_r\}$ and edge set given by $\{\{v_i,w_i\}\;|\;i\in\{1,\ldots,r\}\}$. Let $A$ be the automorphism group of $G$. For an element $a\in A$ and for $i\in\{1,\ldots,r\}$, define an integer $a_i$ by $a_i=1$ if $a(v_i)\in\{v_1,\ldots,v_r\}$, and $a_i=-1$ otherwise. There exists an isomorphism $\varphi\colon A\xrightarrow{\sim} \mu_2^r \rtimes S_r$, where $\mu_2$ is the multiplicative group with two elements and $S_r$ the symmetric group on $r$ elements, acting on $\mu_2^r$ by permuting the coordinates, given by $$\varphi(a)=(\left(a_1,\ldots,a_r\right),\left(i\mapsto j \mbox{ for }a(v_i)\in\{v_j,w_j\}\right)).$$ 
\end{lemma} 
\begin{proof}Let $a$ be an element in $A$. Note that for all $i$, the image $a(v_i)$ of $v_i$ is only connected to $a(w_i)$, so there is a $j$ such that $\{a(v_i),a(w_i)\}=\{v_j,w_j\}$. Therefore we have a group homomorphism $\gamma\colon A\longrightarrow S_r$, given by 
$$a\longmapsto \left(i\mapsto j \mbox{ for }a(v_i)\in\{v_j,w_j\}\right).$$
Note that $\gamma$ is surjective, and its kernel consists of all elements $a\in A$ such that, for all $i\in\{1,\ldots,r\}$, either $a(v_i)=v_i$, or $a(v_i)=w_i$.  We conclude that the kernel of $\gamma$ is isomorphic to the group $\mu_2^r$. So we have a short exact sequence
$$1\longrightarrow\mu_2^r\longrightarrow A\overset{\gamma}{\longrightarrow} S_r\longrightarrow1.$$ Moreover, we have a section $S_r\longrightarrow A$ given by $g\longmapsto \{v_i\mapsto v_{g(i)},\;w_i\mapsto w_{g(i)}\}$, so the statement follows. 
\end{proof}

\begin{lemma}\label{index2}Let $c=\{\{e_1,f_1,\},\ldots,\{e_7,f_7\}\}$ be an element in the set $C$ that is defined above Lemma \ref{sevenfaces}, and let~$s$ be the set $\{e_1,\ldots,e_7,f_1,\ldots,f_7\}$. Let $A$ the automophism group of the colored graph associated to $s$, and let $\varphi\colon A\xrightarrow{\sim}\mu_2^7\rtimes S_7$ be the isomorphism from Lemma \ref{semidirect product}.
Let $W_s$ be the stabilizer in $W$ of $s$. Then there is an injective map $W_s\longrightarrow A$, whose image has index 2 in $A$, and its image after composing with $\varphi$ is given by $$\left\{((m_1,\ldots,m_7),g)\in \mu_2^7\rtimes S_7\;|\;\prod_{i=1}^7m_i=1\right\}.$$
\end{lemma}
\begin{proof}
Elements in $W_s$ respect the dot product, so we have a map $\beta\colon W_s\longrightarrow A$. If an element $w\in W_s$ fixes every element in $s$, then it fixes a sublattice of $\Lambda$ of finite index by Lemma \ref{sevengenerate}, and since $\Lambda$ is torsion free this implies that $w$ is the identity. So the action of $W_s$ on $s$ is faithful, hence $\beta$ is injective, and $|\beta(W_s)|=|W_s|$. Since~$W$ acts transitively on $C$ by Corollary \ref{transsevenface}, and $|C|=2160$ by Remark \ref{numbercrosspolytopes}, we have $|W_s|=|W_c|=\frac{|W|}{|C|}=\frac{|W|}{2160}=322560$. Moreover, we have $|A|=2^7\cdot7!=645120$, so $|\beta(W_s)|=|W_s|=322560=\frac{1}{2}\cdot |A|$. Hence $\beta(W_s)$ is a subgroup of index two in~$A$. We will now determine which subgroup. Note that $\|e_1-e_2\|=\sqrt{2}$, so $e_1-e_2$ is an element $e\in E$, and the reflection in the hyperplane orthogonal to $e$ gives an element in $W$, say $r_{12}$. Note that $e_1+f_1=e_2+f_2$ by Remark \ref{facetcrosspolytopehyperplane}, so $e_1-e_2=f_2-f_1$. Therefore $r_{12}$ interchanges $e_1$ with $e_2$ and $f_1$ with $f_2$. Moreover, since all roots in $\{e_3,\ldots,e_7,f_3,\ldots,f_7\}$ are orthogonal to $e$, the element $r_{12}$ acts trivially on them. Analogously, for $i,j\in\{1,\ldots,7\}$, $i\neq j$, the reflection $r_{ij}$ is an element in $W_s$ that interchanges $e_i$ and $e_j$, and $f_i$ with $f_j$. Let $\gamma\colon A\longrightarrow S_7$ be the projection of $\varphi(A)$ to~$S_7$, then it follows that $\gamma(\beta(W_s))=S_7$. Now consider for $i,j\in\{1,\ldots,7\}, i\neq j$, the element $e_i-f_j$. Again, this is an element in $E$, and the reflection $t_{ij}$ in the hyperplane orthogonal to it is an element in $W_s$ interchanging $e_i$ with $f_j$, and $e_j$ with $f_i$, and leaving all other roots in $s$ fixed. It follows that the composition $t_{ij}\circ r_{ij}$ is an element in $W_s$ with $\varphi(\beta(t_{ij}\circ r_{ij}))=((-1,-1,1,1,1,1,1),\mbox{id})\in \mu_2^7\rtimes S_7$. By composing these automorphisms $t_{ij}\circ r_{ij}$ for different $i,j$, we see that $\varphi(\beta(W_c))$ contains all elements $((m_1,\ldots,m_7),g)\in \mu_2^7\rtimes S_7$ with $\prod_{i=1}^7m_i=1$. Therefore, the reflections  $r_{ij},t_{ij}$ generate a subgroup of $A$ of order $7!\cdot 2^6=\tfrac{1}{2}|A|$, and we conclude that this is all of $W_s$. 
\end{proof}

\begin{corollary}\label{thm2crosspolytopes}
Let $K_1$ and $K_2$ be two cliques in $\Gamma$ whose vertices correspond to a $7$-crosspolytope in the $E_8$ root polytope. Let $f\colon K_1\longrightarrow K_2$ be an isomorphism between them. Then~$f$ extends to an automorphism of $\Lambda$ if and only if for every subclique $S=\{e_1,\ldots,e_7\}$ of $K_1$ of 7 vertices that are pairwise connected with edges of color~1, the vectors $\sum_{i=1}^7e_i$ and $\sum_{i=1}^7f(e_i)$ are either both in $2\Lambda$, or neither are.
\end{corollary}
\begin{proof}
Consider the set $H=\{c_1\ldots,c_7,d_1,\ldots,d_7\}$, where the elements are defined above Lemma \ref{sevenfaces}. Note that the vertices in $H$ correspond to a 7-crosspolytope, and since $W$ acts transitively on the set of cliques corresponding to 7-crosspolytopes (Corollary \ref{transsevenface}), there are elements $\alpha$, $\beta$ in $W$ such that $\alpha(K_1)=\beta(K_2)=H$. So $\beta\circ f\circ\alpha^{-1}$ is an element in the automorphism group $\Aut(H)$ of $H$. Of course, $f$ extends to an element in $W$ if and only if $\beta\circ f\circ\alpha^{-1}$ does. Moreover, since $\alpha$ and~$\beta$ are automorphisms of $\Lambda$, the two sums $\sum_{i=1}^7f(e_i)$ and $\sum_{i=1}^7(\beta\circ f\circ\alpha^{-1})(e_i)$ are either both in or both not in $2\Lambda$. We conclude that we can reduce to the case where $K_1=K_2=H$, and $f$ is an element in $\Aut(H)$. \\
Let $W_H$ be the stabilizer of $H$ in $W$. By Lemma \ref{index2}, there exists an injective map $\psi\colon W_H\longrightarrow\Aut(H)$, whose image has index 2 in $\Aut(H)$. Of course, for all elements~$w$ in the image of $\psi$, and for all cliques $S=\{s_1,\ldots,s_7\}$ as in the statement, the sums $\sum_{i=1}^7s_i$ and $\sum_{i=1}^7w(s_i)$ are either both in, or both not in $2\Lambda$. We will show that this completely determines the image of $\psi$, that is, we will show that every element in $\Aut(H)\setminus\psi(W_H)$ does not have this property for all cliques $S$ as in the statement. To this end, consider the element $h$ in $\Aut(H)$ that exchanges $c_1$ and~$d_1$, and fixes all other vertices. Since $h$ exchanges an odd number of $c_i$ with $d_i$, it is not in the image of $\psi$. Note that $S=\{c_1,\ldots,c_7\}$ is a clique as in the statement. The sum $\sum_{i=1}^7c_i=(5,3,3,3,-1,1,1,1)$ is an element in $2\Lambda$, and its image under $h$, which is $\sum_{i=1}^7h(c_i)=d_1+\sum_{i=2}^7c_i=(4,2,4,4,-1,1,1,1)$, is not. Since all elements in $\Aut(H)\setminus\psi(W_H)$ are compositions of $h$ with elements in $W_H$, we conclude that for all elements $a$ in $\Aut(H)\setminus\psi(W_H)$, the sum $\sum_{i=1}^7a(c_i)$ is not an element in $2\Lambda$. Since the image of $\psi$ consists exactly of those elements in $\Aut(H)$ extending to an element in $W$, this finishes the proof.
\end{proof}

\begin{lemma}\label{driehoek}
The group $W$ acts transitively on the set $$B=\{(e_1,e_2,e_3)\in E^3\;|\;e_1\cdot e_2=0,\;e_2\cdot e_3=e_1\cdot e_3=1\}.$$
\end{lemma}
\begin{proof}
By Proposition \ref{intersection} and Lemma \ref{sevenfaces}, we have $|B|=240\cdot126\cdot12=362880$. 
Let $c,s,A$ be as defined in Lemma \ref{index2}, and note that $b=(e_1,f_1,e_2)$ is an element in $B$. 
Let $W_b$ be the stabilizer in $W$ of $b$. 
Then we have $$|W_b|=\frac{|W|}{|Wb|}\geq\frac{|W|}{|B|}=1920.$$ We want to show that this is an equality. \\
Since $c$ is the unique element in $C$ containing $e_1,f_1$ by Lemma~\ref{sevenfaces}, the stabilizer $W_{b}$ of $b$ acts on the set $s$. 
If an element $w\in W_b$ fixes all the roots in $s$, then it fixes a full rank sublattice of finite index in $\Lambda$, and since $\Lambda$ is torsion free this implies that $w$ is the identity. Therefore the action of $W_b$ on $s$ is faithful, so there is an injective map $W_b\longrightarrow W_s$. Note that $f_2$ is uniquely determined in $s$ as the root that is orthogonal to $e_2$, so every element in $W_b$ fixes $e_1,e_2,f_1,f_2$, hence $W_b$ acts faithfully on $s'=\{e_3,\ldots,e_7,f_3,\ldots,f_7\}$. Let $A'$ be the automorphism group of the colored graph associated to $s'$. We know there is an isomorphism $\varphi':A'\longrightarrow\mu_2^5\rtimes S_5$ by Lemma \ref{semidirect product}. Since elements in $W_b$ respect the dot product, we have an injective map $\beta'\colon W_b\longrightarrow A'$. Let $\beta\colon W_s\longrightarrow A$ be the injective map from Lemma \ref{index2}. together with the injective maps $W_b\longrightarrow W_s$ and $A'\longrightarrow A$, we have the following commutative diagram.

\begin{center}
\begin{tikzcd}
W_b \arrow[r, hook, "\beta'"] \arrow[d, hook]
& A' \arrow[d,hook] \arrow[r,"\varphi'" ,"\sim"']& \mu_2^5\rtimes S_5 \arrow[d,hook]\\
W_s \arrow[r,hook,"\beta"] & A\arrow[r, "\varphi" ,"\sim"']& \mu_2^7\rtimes S_7
\end{tikzcd}
\end{center}

By Lemma \ref{index2}, the image $\varphi(\beta(W_s))$ is a subset of index 2 in $\mu_2^7\rtimes S_7$, given by subset $\left\{((m_1,\ldots,m_7),g)\in \mu_2^7\rtimes S_7\;|\;\prod_{i=1}^7m_i=1\right\}.$ Intersecting this subset with $\mu_2^5\rtimes S_5$ gives a subset of index 2 in $\mu_2^5\rtimes S_5$, so by the diagram above, the image $\varphi'(\beta'(W_b))$ has index at least 2 in $\mu_2^5\rtimes S_5$. We find $|W_b|\leq\frac{1}{2}\cdot2^5\cdot5!=1920$, so together with the inequality above we conclude that $|W_b|=1920$. So we find $|Wb|=\frac{|W|}{|W_b|}=362880=|B|$, and $W$ acts transitively on $B$. \end{proof}

We can now prove Proposition \ref{B}. 

\begin{proofofpropositionB}\label{prp2}
Note that for $a,b,c$ fixed and $\sigma$ any permutation of them, there is a bijection between the sets $V_{a,b,c}$ and $V_{\sigma(a),\sigma(b),\sigma(c)}$, so if we prove that $W$ acts transitively on one of them, then $W$ also acts transitively on the other by Lemma \ref{action}. Therefore, we only consider the sets $V_{a,b,c}$ where $a\leq b\leq c$.\\
There are 4 different sets with $a=b=c$. There are 12 different sets where two of $a,b,c$ are equal to each other and unequal to the third, and 4 different sets with $a,b,c$ all distinct. So there are 20 different sets $V_{a,b,c}$ with $a\leq b\leq c$.  
\begin{itemize}
\item[]$\bullet$ If $V_{a,b,c}$ is a non-empty set with $a=-2$, then every element $(e_1,e_2,e_3)$ in $V_{a,b,c}$ has $e_1=-e_2$, so $b=-c$. Therefore the set $V_{a,b,c}$ is empty for $(a,b,c)$ in \begin{align*}\{(-2,-2,-2),&(-2,-2,-1),(-2,-2,0),(-2,-2,1),\\
&(-2,-1,-1),(-2,-1,0),(-2,0,1),(-2,1,1)\}.
\end{align*}
\item[]$\bullet$ We have proved that $W$ acts transitively on the sets $V_{-1,-1,-1}$ (Corollary \ref{trans-1-1-1}), $V_{0,0,0}$ (Proposition \ref{three}), $V_{0,0,1}$ (Lemma \ref{lemma}), $V_{0,1,1}$ (Lemma \ref{driehoek}), and $V_{1,1,1}$ (Proposition~\ref{transitief}). 
\item[]$\bullet$ We have the following bijections.
\begin{align*}
&\{(e_1,e_2)\in E^2\;|\;e_1\cdot e_2=-1\}\longrightarrow V_{-2,-1,1},\;&(e_1,e_2)\longmapsto (-e_1,e_1,e_2);\\
&\{(e_1,e_2)\in E^2\;|\;e_1\cdot e_2=0\}\longrightarrow V_{-2,0,0},\;&(e_1,e_2)\longmapsto (-e_1,e_1,e_2);\\
&V_{0,1,1}\longrightarrow V_{-1,-1,0},\; &(e_1,e_2,e_3)\longmapsto (e_1,-e_3,e_2);\\
&V_{1,1,1}\longrightarrow V_{-1,-1,1},\; &(e_1,e_2,e_3)\longmapsto (e_1,-e_2,e_3);\\
&V_{0,0,1}\longrightarrow V_{-1,0,0},\;&(e_1,e_2,e_3)\longmapsto (-e_1,e_3,e_2);\\
&V_{0,1,1}\longrightarrow V_{-1,0,1},\; &(e_1,e_2,e_3)\longmapsto (-e_3,e_2,-e_1);\\
&V_{-1,-1,-1}\longrightarrow V_{-1,1,1},\; &(e_1,e_2,e_3)\longmapsto (e_1,e_2,-e_3).
\end{align*}
We proved that $W$ acts transitively on the six different sets on the left-hand sides. From Lemma \ref{action} it follows that $W$ acts transitively on $V_{-2,-1,1}$, $V_{-2,0,0}$, $V_{-1,-1,0}$, $V_{-1,-1,1}$, $V_{-1,0,0}$, $V_{-1,0,1}$, and $V_{-1,1,1}$, too.  
\end{itemize}
Since we proved that $V_{a,b,c}$ is either empty or $W$ acts transitively on it for 20 different sets, we conclude that we proved the proposition.
\end{proofofpropositionB}

The following corollary proves Theorem \ref{main2} for cliques of Type III. 

\begin{corollary}\label{corB}
Let $K_1$ and $K_2$ be two cliques of type III, and let $f\colon K_1\longrightarrow K_2$ be an isomorphism between them. Then $f$ extends to an automorphism of $\Lambda$.
\end{corollary}
\begin{proof}
Since the group $W$ acts transitively on the set of ordered sequences of $n$ roots for $1\leq n \leq3$ by , there exists an automorphism $w\in W$ of $\Lambda$ such that $w$ restricted to $K_1$ equals $f$.
\end{proof}

\section{Monochromatic cliques}\label{monocliques}

In this section we study the cliques of type I, that is, cliques in $\Graph{-2},\;\Graph{-1},\;\Graph{0},\;$ and~$\Graph{1}$. We describe the orbits under the action of $W$ of \textsl{sequences} of roots that form a clique, thus obtaining the results in Theorem \ref{main2} for cliques of type I (see Corollaries~\ref{thm2mono0} and~\ref{thm2mono1}). We also describe all maximal cliques per color. For $\Graph{-2}$ and $\Graph{-1}$, everything follows from the previous sections. For $\Graph{1}$ we already have Proposition~\ref{transitief}; we show moreover that there are no cliques of size bigger than eight, and describe the maximal cliques in Proposition~\ref{max178}. Finally, in this section we prove that $W$ acts transitively on ordered sequences of orthogonal roots of length~$r$ for $r\geq5$. The result is in Proposition \ref{orbitstuplesgamma0}. Throughout this section we do not use any computer. 

\vspace{11pt}

\textbf{Cliques in $\Graph{-2}$}\\
The maximal size of a clique in $\Graph{-2}$ is two, since such a maximal clique consists of an element in $E$ and its inverse (see Proposition \ref{intersection}). There are therefore 120 such cliques. In Lemma \ref{twosets} we showed that $W$ acts transitively on the set of ordered pairs $\{(e_1,e_2)\in E^2\;|\;e_1=-e_2\}$, so $W$ acts transitively on the set of maximal cliques in~$\Graph{-2}$.

\vspace{11pt}

\textbf{Cliques in $\Graph{-1}$}\\
In $\Graph{-1}$, the maximal size of a clique is three, and there are no maximal cliques of smaller size, by Lemma \ref{unique-1-1-1}. From Proposition \ref{intersection} and Lemma \ref{unique-1-1-1} it follows that there are $\frac{240\cdot56}{3!}=2240$ maximal cliques. By Corollary \ref{trans-1-1-1}, the group $W$ acts transitively on the set of sequences $\{(e_1,e_2,e_3)\in E^3\;|\;e_1\cdot e_2=e_2\cdot e_3=e_1\cdot e_3=-1\}$, so $W$ acts transitively on the set of maximal cliques in $\Graph{-1}$. By Lemma \ref{twosets}, the group $W$ acts transitively on the set $\{(e_1,e_2)\in E^2\;|\;e_1\cdot e_2=-1\}$, so $W$ acts also transitively on the set of cliques of size two in $\Graph{-1}$, of which there are $\frac{240\cdot56}{2}=6720$ (Proposition~\ref{intersection}).

\vspace{11pt}

\textbf{Cliques in $\Graph{0}$}\\
Cliques in $\Graph{0}$ are studied in \cite{DynMin10}, where they are called \textsl{orthogonal subsets}. In their article, the authors show that the maximal size of cliques in $\Graph{0}$ is eight \cite[Table~1]{DynMin10}, that two cliques of the same size $r$ are conjugate if $r\neq4$, and that there are two orbits of cliques of size 4 \cite[Corollary 2.3]{DynMin10}. In the previous section we showed that $W$ acts transitively on the set of \textsl{ordered sequences} of length at most~3 of orthogonal roots, and that there are two orbits of sequences of length~4. In this section we use this to conclude the same results as in \cite{DynMin10} for cliques of size $r\leq4$, and we compute the number of these cliques. Moreover, we study the action of $W$ on ordered sequences of length $\geq5$ of orthogonal roots (Proposition~\ref{orbitstuplesgamma0}), and compute the number of cliques of size $\geq5$ (Proposition \ref{gamma0}).

\vspace{11pt}

The following proposition deals with the cliques of size at most 4. 

\begin{proposition}\label{orbitskleinerdan4}
\begin{itemize}
\item[]
\item[](i) There are $15120$ cliques of size two in $\Graph{0}$, and the group $W$ acts transitively on the set of all of them.
\item[](ii) There are $302400$ cliques of size three in $\Graph{0}$, and the group $W$ acts transitively on the set of all of them.
\item[](iii) There are $1965600$ cliques of size four in $\Graph{0}$, and they form two orbits under the action of $W$: one of size $151200$, in which all cliques have vertices whose roots sum up to a vector in $2\Lambda$, and one of size $1814400$, in which all cliques have vertices whose roots sum op to a vector that is not in $2\Lambda$. 
\end{itemize}
\end{proposition}
\begin{proof}
\begin{itemize}
\item[]
\item[](i) We have shown that the group $W$ acts transitively on the set $$A_0~=~\{\left(e_1,e_2\right)\in E^2\;|\;e_1\cdot e_2=0\}$$ (Proposition \ref{pair}), and $|A_0|=240\cdot126=30240$ (Proposition \ref{intersection}). It follows that there are $\frac{30240}{2}=15120$ cliques of size two in $\Graph{0}$, and the group $W$ acts transitively on the set of all of them.
\item[](ii) The group $W$ acts transitively on the set $$V_3~=~\{\left(e_1,e_2,e_3\right)\in E^3\;|\;\forall i\neq j:e_i\cdot e_j=0\},$$ and we have $|V_3|=1814400$ (Proposition \ref{three} (i)). It follows that there are $\frac{1814400}{6}=302400$ cliques of size three in $\Graph{0}$, and the group $W$ acts transitively on the set of all of them.
\item[](iii) By Proposition \ref{orbits} there are two orbits under the action of $W$ on the set$$V_4=\{\left(e_1,e_2,e_3,e_4\right)\in E^4\;|\;\forall i\neq j:e_i\cdot e_j=0\};$$ one of size $3628800$ where all elements have coordinates that sum up to a vector that is in $2\Lambda$, and one orbit of size $43545600$ where all elements have coordinates that sum up to a vector that is not in $2\Lambda$. Since the orbit in which an element is contained does not depend on the order of its coordinates, we conclude that this also gives two orbits with the same properties under the action of $W$ on the set of all cliques of size four in $\Graph{0}$, of sizes $\frac{3628800}{4!}=151200$ and $\frac{43545600}{4!}=1814400$, respectively.\qedhere
\end{itemize}
\end{proof}

We continue by studying the sequences of orthogonal roots of length greater than four. Recall the set~$V_4$ and its orbits under the action of $W$, given by $Y$ of size $43545600$ and $O=V_4\setminus Y$ of size $3628800$ (Proposition \ref{orbits}).

\begin{lemma}\label{ydisjoint}
For an element $y=(e_1,\ldots,e_4)\in Y$, define the set $$C_y=\{e\in E\;|\;e\cdot e_i=0\mbox{ for }i\in\{1,2,3,4\} \}.$$ The following hold. 
\begin{itemize}
\item[] (i) The set $C_y$ is the union of four sets $\{f_1,-f_1\},\{f_2,-f_2\},\{f_3,-f_3\},\{f_4,-f_4\}$ with $f_i\cdot f_j=0$ for $i\neq j$. For such a set $\{f_i,-f_i\}$, there is exactly one triple $\{e_{i_1},e_{i_2},e_{i_3}\}$ of elements in $y$ such that the permutations of $(e_{i_1},e_{i_2},e_{i_3},f_i)$ (or equivalently of $(e_{i_1},e_{i_2},e_{i_3},-f_i)$) form elements in $O$. Moreover, for $j\neq i$, and $j_1,j_2,j_3$ such that the permutations of $(e_{j_1},e_{j_2},e_{j_3},f_j)$ form elements in $O$, the sets $\{e_{i_1},e_{i_2},e_{i_3}\}$ and $\{e_{j_1},e_{j_2},e_{j_3}\}$ are different.
\item[](ii) The stabilizer of $y$ is generated by the reflections in the hyperplanes orthogonal to $f_i$ for $i\in\{1,2,3,4\}$.
\end{itemize}  
\end{lemma}
\begin{proof}
Since $W$ acts transtively on $Y$, it suffices to show this for a fixed element $y\in Y$. Set \begin{align*}
&e_1=\left(1,1,0,0,0,0,0,0\right),  & &e_3=\left(0,0,0,0,1,1,0,0\right),\\ &e_2=\left(0,0,1,1,0,0,0,0\right),  & &e_4=\left(1,-1,0,0,0,0,0,0\right).\end{align*}
Then $(e_1,e_2,e_3,e_4)$ is an element in $V_4$ and since $\sum_{i=1}^4 e_i\notin 2\Lambda$, it is an element in~$Y$ as well by Proposition \ref{orbits}. Take $e=(a_1,\ldots,a_8)\in E$ such that $e\cdot e_i=0$ for $i\in \{1,2,3,4\}$. Then we have $a_1+a_2=a_1-a_2=a_3+a_4=a_5+a_6=0$. We find the following possibilities.
\begin{align*}
&\pm f_1=\pm\left(0,0,0,0,0,0,1,-1\right),  & &\pm f_3=\pm\left(0,0,1,-1,0,0,0,0\right),\\ 
&\pm f_2=\pm\left(0,0,0,0,1,-1,0,0\right),& &\pm f_4=\pm\left(0,0,0,0,0,0,1,1\right).\end{align*}
It is an easy check that $f_i\cdot f_j=0$ for $i\neq j$, and for $i,k\in {1,2,3,4}$, the sum $\left(\sum_{j\neq i}e_j\right)\pm f_k$ is contained in $2\Lambda$ if and only if $i=k$. This proves (i). We continue with (ii). Take $i\in\{1,2,3,4\}$. Since $f_i$ is orthogonal to the elements in $y$ the reflection $r_i$ in the hyperplane orthogonal to $f_i$ is an element of $W_y$. For $i\neq j$, the reflections $r_i$ and $r_j$ commute, since $f_i$ and $f_j$ are orthogonal. Therefore the elements $r_1,r_2,r_3,r_4$ generate a subgroup of $W_y$ of order 16. Since we have $$|W_y|=\frac{|W|}{|Y|}=\frac{696729600}{43545600}=16,$$ they generate the whole group $W_y$. 
\end{proof}

\begin{corollary}\label{uniqueY}
Set $n_5=1,\;n_6=3,\;n_7=7,$ and $n_8=14$. Let $K$ be a clique of size $r\in\{5,6,7,8\}$ in $\Graph{0}$. Then the number of sets of four vertices $e_1,e_2,e_3,e_4$ in $K$ such that the permutations of $(e_1,e_2,e_3,e_4)$ are elements in $O$ is equal to $n_r$.
\end{corollary}
\begin{proof} First let $K$ be a clique of size $5$ in $\Graph{0}$. Assume in contradiction that there are two distinct subsets, say $y_1,y_2$, of four vertices in $K$ that form an element in~$O$. Then there are three vertices of $K$, say $e_1,e_2,e_3$, that are contained both in $y_1$ and~$y_2$. Write $y_1=\{e_1,e_2,e_3,f_1\},\;y_2=\{e_1,e_2,e_3,f_2\}$. By applying Proposition~\ref{three}~(iii) to the triple $(e_1,e_2,e_3)$, it follows that $f_1=-f_2$, so $f_1\cdot f_2=-2$. But this gives a contradiction, since $f_1,f_2$ are both in $K$. So the number of sets of four vertices in~$K$ that form an element in $O$ is at most $1$, which means that there is at least one subset $\{g_1,g_2,g_3,g_4\}$ of $K$ of four roots such that $(g_1,g_2,g_3,g_4)$ is an element in $Y$. For the fifth element in $K$, say $g_5$, it follows from the previous lemma that there is exactly one triple $\{g_{\alpha},g_{\beta},g_{\gamma}\}$ of elements in $\{g_1,\ldots,g_4\}$ that it forms an element in~$O$ with. We conclude that there is exactly 1 set of four vertices in $K$ that form an element in $O$; this proves the statement for $r=5$.\\
We proceed by induction. Take $s\in\{6,7,8\}$. Assume that the statement holds for $5\leq r<s$, and let $K~=~\{e_1,\ldots,e_s\}$ be a clique of size $s$ in $\Graph{0}$. By induction we know that $\{e_1,\ldots,e_{s-1}\}$ contains $n_{s-1}$ subsets of size four that form an element in~$O$. That means that there are $\binom{s-1}{4}-n_{s-1}$ subsets of size four in $\{e_1,\ldots,e_{s-1}\}$ that form an element in $Y$. By Lemma \ref{ydisjoint}, each of these $\binom{s-1}{4}-n_{s-1}$ subsets contains exactly three elements that, together with $e_s$, form an element in~$O$. Let $d_1,d_2,d_3$ be three elements in $\{e_1,\ldots,e_{s-1}\}$ such that $(d_1,d_2,d_3,e_s)$ is an element in~$O$. Then for every element $d\in\{e_1,\ldots,e_{s-1}\}\setminus\{d_1,d_2,d_3\}$, the set $\{d_1,d_2,d_3,e_s,d\}$ forms a clique of size 5 in~$\Graph{0}$, and since $n_5=1$, it follows that $(d_1,d_2,d_3,d)$ is an element in~$Y$. This means that every set of three roots in $\{e_1,\ldots,e_{s-1}\}$ that forms an element in $O$ with $e_s$ forms an element in $Y$ with all other roots in $\{e_1,\ldots,e_{s-1}\}$. Since every set of three roots in $\{e_1,\ldots,e_{s-1}\}$ is contained in $(s-1)-3$ subsets of size four of $\{e_1,\ldots,e_{s-1}\}$, this gives $\tfrac{\binom{s-1}{4}-n_{s-1}}{s-4}$ distinct sets of three that form an element in $O$ with $e_s$. In total this gives $n_{s-1}+\tfrac{\binom{s-1}{4}-n_{s-1}}{s-4}$ sets of four vertices in $K$ that form an element in $O$. This is exactly equal to $n_s$ for $s=6,7,8$.
\end{proof}

For $1\leq r\leq 8$, let $V_r$ be the set $$V_r=\{(e_1,\ldots,e_r)\in E^r\;|\;\forall i\neq j:e_i\cdot e_j=0\}.$$

\begin{proposition}\label{orbitstuplesgamma0}
For $5\leq r\leq8$, two elements $(e_1,\ldots,e_r)$, $(f_1,\ldots,f_r)$ in $V_r$ are in the same orbit under the action of $W$ if and only if for all $1\leq i<j<k<l\leq r$, the elements $(e_i,e_j,e_k,e_l)$ and $(f_i,f_j,f_k,f_l)$ are conjugate in $V_4$ under the action of~$W$. 
\end{proposition}
\begin{proof}
For $5\leq r\leq 8$, define the relation $\sim$ on $V_r$ by $(e_1,\ldots,e_r)\sim(f_1,\ldots,f_r)$ if and only if for all $1\leq i<j<k<l\leq r$, the elements $(e_i,e_j,e_k,e_l)$ and $(f_i,f_j,f_k,f_l)$ are conjugate in $V_4$. Note that $\sim$ is an equivalence relation on $V_r$, and the group $W$ acts on the equivalence classes. Our goal is to show that each equivalence class is an orbit in $V_r$ under the action of $W$. We do this by induction on $r$.\\
For $r=5$, let $X_5\subset V_5$ be an equivalence class with respect to $\sim$. We distinguish two cases. If for every element in $X_5$ the first four coordinates form an element in $Y$, we let $p\colon X_5\longrightarrow Y$ be the projection to the first four coordinates. Note that this is surjective by Lemma \ref{ydisjoint}. Set $y=(y_1,\ldots,y_4)\in Y$. Since the elements in the fiber $p^{-1}(y)$ are equivalent under $\sim$, there are exactly two elements $(y_1,\ldots,y_4,f),(y_1,\ldots,y_4,-f)$ in $p^{-1}(y)$ by Lemma \ref{ydisjoint} (i). Moreover, the stabilizer $W_y$ acts transitively on these two elements by Lemma \ref{ydisjoint} (ii). From Lemma \ref{action} it follows that $W$ acts transitively on $X_5$. If, on the other hand, for every element in~$X_5$ the first four coordinates form an element in $O$, then the last four coordinates of every element in $X_5$ form an element in $Y$ by Corollary \ref{uniqueY}. We now let $p\colon X_5\longrightarrow Y$ be the projection to the last four coordinates, and the proof is the same. \\    
Now assume that $r>5$, and that each equivalence class in $V_{r-1}$ is an orbit under the action of $W$. Let $X_r$ be an equivalence class in $V_r$, and $p_r\colon X_r\longrightarrow V_{r-1}$ the projection on the first $r-1$ coordinates. Then $W$ acts on $p_r(X_r)$, and $p_r(X_r)$ is contained in an equivalence class $X_{r-1}$ with respect to $\sim$ in $V_{r-1}$. Since $W$ acts transitively on $X_{r-1}$ by hypothesis, it follows that $p_r(X_r)=X_{r-1}$, and $W$ acts transitively on $p_r(X_r)$. Since $r>5$, by Corollary \ref{uniqueY} there exist $i,j,k,l\in\{1,\ldots,r-1\}$ such that for all elements $(e_1,\ldots,e_r)\in X_r$ we have $(e_i,e_j,e_k,e_l)\in Y$. Fix such $i,j,k,l$, and let $v=(v_1,\ldots,v_{r-1})$ be an element in $p_r(X_r)$. Then $(v_i,v_j,v_k,v_l)$ is an element in $Y$. Let $(v_1,\ldots,v_{r-1},f)$, $(v_1,\ldots,v_{r-1},g)$ be elements in the fiber $p_r^{-1}(v)$. Since $(v_1,\ldots,v_{r-1},f)$ is equivalent to $(v_1,\ldots,v_{r-1},g)$ with respect to $\sim$, by applying Lemma \ref{ydisjoint} to $(v_i,v_j,v_k,v_l)$ we see that $f=-g$, and the fiber $p_r^{-1}(v)$ consists of the two elements $(v_1,\ldots,v_{r-1},f)$ and $(v_1,\ldots,v_{r-1},-f)$. Moreover, the reflection in the hyperplane orthogonal to $f$ fixes $v_1,\ldots,v_{r-1}$, hence is an element in the stabilizer of $v$ that switches $f$ and $-f$. So the stabilizer of $v$ acts transitively on $p_r^{-1}(v)$, and again from Lemma \ref{action} we conclude that $W$ acts transitively on $X_r$.    
\end{proof}

\begin{corollary}\label{thm2mono0}
Let $K_1$ and $K_2$ be two cliques in $\Graph{0}$, and $f\colon K_1\longrightarrow K_2$ an isomorphism between them. Then $f$ extends to an automorphism of $\Lambda$ if and only if for every subclique $S$ of size 4 in $K_1$, the image $f(S)$ in $K_2$ is conjugate to $S$ under the action of~$W$.
\end{corollary}
\begin{proof}
If $K_1$ and $K_2$ have size $\leq3$, then $f$ extends always by Corollary \ref{corB}. From Proposition \ref{orbitstuplesgamma0} it follows that if $K_1$ and $K_2$ have size at least four, the isomorphism~$f$ extends to an element in $W$ exactly when $f$ sends every sequence of four roots that form an element in $V_4$ to a conjugate element in $V_4$. By Proposition \ref{orbits}, there are two orbits of ordered sequences of four pairwise orthogonal roots, that do not depend on the order of the roots. We conclude that if $S$ and $f(S)$ are conjugate under the action of $W$ for every set $S$ of four vertices in $K_1$, there exists an automorphism~$w\in W$ of $\Lambda$ such that $w$ restricted to $K_1$ equals~$f$.
\end{proof}

\begin{theorem}\label{gamma0}
In $\Graph{0}$, the following hold.
\begin{itemize}
\item[] (i) There are no maximal cliques of size smaller than eight. 
\item[](ii) There are $3628800$ cliques of size five, $3628800$ cliques of size six, $2073600$ cliques of size seven, and $518400$ cliques of size eight.
\item[](iii) The group $W$ acts transitively on the cliques of size $5$.
\end{itemize}
\end{theorem}
\begin{proof}
\begin{itemize}\item[]
\item[](i) We know that every root in $E$ is orthogonal to 126 other roots (Proposition~\ref{intersection}). Moreover, we know that in $\Graph{0}$ every clique of size 2 extends to a clique of size~3 (Lemma \ref{60}), and every clique of size 3 extends to a clique of size 4 (Proposition~\ref{three}~(ii)). Since $n_5=1<\binom54$ by Corollary \ref{uniqueY}, every clique of size 5 in $\Graph{0}$ contains both a subclique whose vertices form an element in $O$, and a subclique whose vertices form an element in $Y$. Since $W$ acts transitively on $O$ and on $Y$, and $V_4=O\cup Y$, this means that every clique of size 4 in $\Graph{0}$ extends to a clique of size~5. Moreover, by Lemma \ref{ydisjoint} (i), every extension of a clique of size 4 whose vertices form an element in $Y$ can be extended to a clique of size 8. Since every clique of size at least 5 contains a clique of size~4 whose vertices form an element in $Y$, there are no maximal cliques of size smaller than 8.
\item[](ii) By Lemma \ref{ydisjoint}, if we fix an element $y=(e_1,e_2,e_3,e_4)\in Y$, there are exactly 8 elements in $V_5$, and $8\cdot6$ elements in $V_6$, and $8\cdot6\cdot4$ elements in~$V_7$, and  $8\cdot6\cdot4\cdot2$ elements in $V_8$, that have $e_i$ as the $i^{\mbox{\tiny{th}}}$ coordinate. We call this number~$m_r$ for $r=5,6,7,8$. For all $5\leq r\leq8$, for $S$ a clique of size $r$, it follows from Corollary~\ref{uniqueY} that $S$ contains $\binom{r}{4}-n_r$ cliques of size 4 that, together, form $4!\cdot(\binom{r}{4}-n_r)$ different elements in $Y$; for such a subclique of size 4 in $S$, the other $r-4$ elements can be permuted in $(r-4)!$ ways. For all $5\leq r\leq 8$, let $D_r$ be the set of cliques of size $r$ in $\Graph{0}$. It follows that we have $$|D_r|=\frac{|Y|\cdot m_r}{4!\cdot(\binom{r}{4}-n_r)\cdot(r-4)!}.$$ We find the following results. $$|D_5|=\frac{|Y|\cdot 8}{4!\cdot4}=3628800,\;|D_6|=\frac{|Y|\cdot 8\cdot6}{4!\cdot12\cdot2}=3628800,$$
$$|D_7|=\frac{|Y|\cdot 8\cdot6\cdot4}{4!\cdot28\cdot3!}=2073600,\;|D_8|=\frac{|Y|\cdot 8\cdot6\cdot4\cdot2}{4!\cdot56\cdot4!}=518400.$$ 
\item[](iii) Let $K_1=\{e_1,\ldots,e_5\},\;K_2=\{f_1,\ldots,f_5\}$ be two cliques in $\Graph{0}$. From Corollary \ref{uniqueY} we have $n_5=1$, so without loss of generality we can assume that $e_1,e_2,e_3,e_4$ and $f_1,f_2,f_3,f_4$ are the unique four elements in $K_1$ and $K_2$, respectively, that form an element in $O$. Then $(e_1,e_2,e_3,e_4,e_5)$ and $(f_1,f_2,f_3,f_4,f_5)$ are conjugate under the action of $W$ by Proposition~\ref{orbitstuplesgamma0}, hence so are $K_1$ and~$K_2$.\qedhere
\end{itemize}
\end{proof}

\textbf{Cliques in $\Graph{1}$}\\
We know that cliques in $\Graph{1}$ form $k$-simplices that are $k$-faces of the $E_8$ root polytope (Proposition \ref{facesform}), hence Corollary \ref{numberfacets} states how many cliques of size $n$ there are in $\Graph{1}$ for $n\leq 8$. Moreover, we know that $W$ acts transitively on these cliques for $n\leq8,\;n\neq7$ (Proposition \ref{transitief}). Proposition \ref{max178} shows that there are no cliques of size bigger than eight in $\Graph{1}$, and that there are two orbits of cliques of size seven (which was already known, for example by \cite{Cox30} and \cite{Man74}); it shows that all maximal cliques are of size 7 or 8. 

\begin{proposition}\label{max178}In $\Graph{1}$, the following hold. 
\begin{itemize}
\item[](i) There are only maximal cliques of size 7 and 8.
\item[](ii) There are two orbits of cliques of size 7 in $\Graph{1}$; one of size 138240, which is given by non-maximal cliques, and one of size 69120, which is given by maximal cliques. A clique of size seven in $\Graph{1}$ is maximal if and only if the sum of its vertices is an element in $2\Lambda$.
\item[](iii) There are 17280 cliques of size 8.
\end{itemize}
\end{proposition}
\begin{proof}
Consider the clique of size six in $\Graph{1}$ given by $\{e_1,\ldots,e_6\}$, where we define 
\begin{align*}
e_1=\left(1,1,0,0,0,0,0,0\right),\;\;\;\;\; & e_4=\left(1,0,0,0,1,0,0,0\right)\\
e_2=\left(1,0,1,0,0,0,0,0\right),\;\;\;\;\; & e_5=\left(1,0,0,0,0,1,0,0\right)\\
e_3=\left(1,0,0,1,0,0,0,0\right),\;\;\;\;\; & 
e_6=\left(1,0,0,0,0,0,1,0\right).
\end{align*}
Since $W$ acts transitively on the set of cliques of size smaller than $6$ in $\Graph{1}$ by Proposition \ref{transitief}, it follows that every clique of size smaller than $6$ in $\Graph{1}$ is contained in a clique of size 6 in $\Graph{1}$. The elements in $E$ that have dot product one with all $e_1,\ldots,e_6$ are given by:
\begin{align*}
c_1=\left(\tfrac{1}{2},\tfrac{1}{2},\tfrac{1}{2},\tfrac{1}{2},\tfrac{1}{2},\tfrac{1}{2},\tfrac{1}{2},\tfrac{1}{2}\right), \;\; c_2=(1, 0, 0, 0, 0, 0, 0, 1 ), \;\; c_3=( 1, 0, 0, 0, 0, 0, 0, -1 ).
\end{align*} 
Note that $c_1\cdot c_2=1$ and $c_3\cdot c_1=c_3\cdot c_2=0$, so $\{e_1,\ldots,e_6,c_1,c_2\}$ is a maximal clique of size 8 in $\Graph{1}$, and $\{e_1,\ldots,e_6,c_3\}$ is a maximal clique of size 7 in $\Graph{1}$. 
Since~$W$ acts transitively on the cliques of size 6 in $\Graph{1}$ by Proposition \ref{transitief}, all maximal cliques in $\Graph{1}$ are of size 7 or 8. This proves part (i). Moreover, it follows that every non-maximal clique of size 7 is contained in a unique clique of size 8, so there are $\frac{138240}{8}=17280$ cliques of size 8. This proves part (iii). We will now prove (ii). From part (i) it follows that there exist maximal and non-maximal cliques of size 7 in $\Graph{1}$. It is obvious that they can not be in the same orbit under the action of $W$. Moreover, there are two orbits of ordered sequences of length 7, hence at most two orbits of cliques of size 7 by Proposition \ref{transitief}. We conclude that the orbits are given exactly by the maximal cliques and the non-maximal cliques. Since there are 483840 cliques of size 6 (Corollary \ref{numberfacets}), from the above it follows that there are $\frac{483840\cdot 2}{7}=138240$ non-maximal cliques, and $\frac{483840\cdot1}{7}=69120$ maximal cliques. Now consider the set $\{e_1,\ldots,e_7\}$, where the elements are defined above Lemma \ref{sevenfaces}. This is a clique of size~7 in $\Graph{1}$, and it is not hard to check that it is maximal. Moreover, we have $$\sum_{i=1}^7e_i=(5,3,3,3,1,1,1,1)\in2\Lambda.$$ Since $W$ acts transitively on all maximal cliques of size~7 in $\Graph{1}$, for all such cliques the sum of the vertices is an element in $2\Lambda$. On the other hand, consider the set $d=\{d_1,\ldots,d_7\}$ as defined above Lemma \ref{sevenfaces}. This is a non-maximal clique of size~7 in $\Graph{1}$, since the union of $d$ with the root $\left(-\frac12, \frac12, \frac12, \frac12, \frac12, -\frac12, -\frac12, -\frac12\right)$ is a clique of size 8 in $\Graph{1}$. Moreover, we have $$\sum_{i=1}^7d_i=(2,4,4,4,1,-1,-1,-1)\not\in2\Lambda.$$ Since $W$ acts transitively on all non-maximal cliques of size 7 in $\Graph{1}$, for all such cliques the sum of the vertices is not an element in $2\Lambda$.
\end{proof}

\begin{remark}
Note that $138240+69120=207360$, which is the total number of cliques of size 7 by Corollary \ref{numberfacets}.
\end{remark}

\begin{corollary}\label{thm2mono1}
Let $K_1$ and $K_2$ be two cliques in $\Graph{1}$, and $f\colon K_1\longrightarrow K_2$ an isomorphism between them. If $K_1$ and $K_2$ have size unequal to 7, then $f$ extends to an automorphism of $\Lambda$. If $K_1$ and $K_2$ have size 7, then $f$ extends if and only if the sum of the vertices of $K_1$ and the sum of the vertices of $K_2$ are either both in $2\Lambda$, or both not.
\end{corollary}
\begin{proof}Another way of saying that the morphism $f$ extends, is that for $\{e_1,\ldots,e_r\}$ the roots in $K_1$, the sequences $(e_1,\ldots,e_r)$ and $(f(e_1),\ldots,f(e_r))$ are conjugate. 
By Proposition \ref{transitief}, for $r\leq8,\;r\neq7$, there is only one orbit of ordered sequences of length $r$ of roots that have pairwise dot product 1. This implies that $f$ extends to an element in $W$ if $K_1$, $K_2$ have size unequal to 7. Furthermore, by the same proposition, there are two orbits of ordered sequences of roots of length~7. By Proposition~\ref{max178}, there two orbits of cliques of size 7, that are distinguished by whether the sum of the 7 roots is an element in $2\Lambda$ or not. We conclude that the two orbits of ordered sequences are distinguished in the same way. This implies that $f$ extends if and only if the sum of the vertices in $K_1$ and the sum of the vertices in $f(K_1)=K_2$ are both in $2\Lambda$ or both not. 
\end{proof}

\begin{remark}\label{6-faces}
We know that the cliques of size 7 in $\Graph{1}$ are 6-faces of the $E_8$ root polytope. We can describe the two orbits of these cliques in this framework as well. A $6$-face of the polytope is an intersection of two facets. There are two types of facets of the $E_8$ root polytope: $7$-crosspolytopes and $7$-simplices (Proposition \ref{facets}). Since the maximal cliques of size 7 in $\Graph{1}$ are not contained in a $7$-simplex, these are exactly the intersections of two $7$-crosspolytopes. \\
Consider the set $c=\{c_1,\ldots,c_7,d_1,\ldots,d_7\}$ as defined above Lemma \ref{sevenfaces}. Note that $d=\{d_1,\ldots,d_7\}$ is a non-maximal clique of size 7 in $\Graph{1}$ that is contained in the $7$-crosspolytope with vertices in $c$, but also in the $7$-simplex with the eight vertices 
$d\;\cup\;\left\{\left(-\frac12, \frac12, \frac12, \frac12, \frac12, -\frac12, -\frac12, -\frac12\right)\right\}$. It follows that every non-maximal clique of size~7 in~$\Graph{1}$ is the intersection of a $7$-crosspolytope with a $7$-simplex.\\
From this it also follows that two $7$-simplices in the $E_8$ root polytope never intersect. 
\end{remark}

\begin{remarkdp1}\label{dp15}
Let $X$ be a del Pezzo surface of degree one over an algebraically closed field, and $C$ the set of exceptional classes in Pic $X$. Through the bijection between $C$ and $E$, cliques in $\Graph{1}$ are related to sets of exceptional classes that are pairwise disjoint (see Remark \ref{dp11}). These are called \textsl{exceptional sets}, and can be blown down so that we obtain a del Pezzo surface of higher degree \cite[Chapter IV]{Man74}. Since a del Pezzo surface can have degree at most 9 (in which case it is $\mathbb{P}^2$), it is clear that the maximal size of a clique in $\Graph{1}$ is eight. We can also describe the two orbits of size seven in this setting; cliques that are maximal correspond to exceptional sets that blow down to a del Pezzo surface of degree eight that is isomorphic to $\mathbb{P}^1\times\mathbb{P}^1$, and cliques that are not maximal correspond to exceptional sets that blow down to a del Pezzo surface of degree eight that is isomorphic to $\mathbb{P}^1$ blown up in one point \cite[remark below Corollary~26.8]{Man74}. 
\end{remarkdp1}

\section{Maximal cliques}\label{maximalcliques}

In this section we describe all maximal cliques in $\Gamma_c$ for $c\neq\{-1,0,1\}$ (cliques of type~IV), and their orbits under the action of $W$. Note that $\Gamma_{-1,0,1}$ is the graph $\Gamma$ after removing all edges between roots and their inverses. This means that the maximal cliques in $\Gamma_{-1,0,1}$ are all of size 120: for each root you can either choose the root or its inverse. There are therefore $2^{120}$ maximal cliques in $\Gamma_{-1,0,1}$, and at least $\left\lceil \frac{2^{120}}{|W|} \right\rceil =1907810427151244719477695595$ orbits in the set of maximal cliques under the action of $W$. Because of the size of these cliques and their orbits, we did not compute the orbits.

\vspace{11pt}

In the first two subsections of this section we describe all maximal cliques in $\Graph{-2}$, $\Graph{-1}$, $\Graph{0}$, $\Graph{1}$, $\Graph{-2,-1}$, $\Graph{-2,1}$, $\Graph{-2,0}$, and $\Graph{-2,-1,0,1}=\Gamma$. Cliques in $\Graph{-2,-1}$ and $\Graph{-2,1}$ are monochromatic (Lemma \ref{mono-1-2}), and maximal cliques in $\Graph{-2,0}$ are in bijection with maximal cliques in $\Graph{0}$ (Lemma \ref{maxin-20}). Therefore, everything before Section \ref{-10} follows from results in Section \ref{monocliques} and is done without a computer. From Section \ref{-10} onwards, we used \texttt{magma} for some computations. The code that we used can be found in \cite{magma}.

\vspace{11pt}

Our motivation to study the cliques in $\Gamma$ comes from del Pezzo surfaces of degree one (see Remark \ref{dp11}), and because of that, the maximal cliques in $\Graph{-2,0}$ and $\Graph{-1,0}$ are of special interest to us, which is explained in Remark \ref{dp12}. For these two graphs we have some extra results. We compute the structure of the largest cliques in the graphs, see Propositions \ref{trans16} and \ref{maxnopairs}. We also show that for these largest cliques, their stabilizer in $W$ acts transitively on the clique itself (Corollaries \ref{eenK} and \ref{corollary11}). The techniques in Sections \ref{-20} and~\ref{-10} show how one could prove similar results for graphs with other colors. 

\vspace{11pt}

The main results of this section are summarized in the tables in Appendix \ref{list} and Remark \ref{welkedingenwelkeklieken}. 

\vspace{11pt}

\noindent\textbf{Notation}
To denote cliques of $\Gamma$ in a compact way, we order the root system $E$ as follows. Roots of the form $\left(\pm\tfrac12,\ldots,\pm\tfrac12\right)$ are ordered lexicographically and denoted by numbers $1-128$; for example, $\left(-\tfrac12,\ldots,-\tfrac12\right)$ is number 1, and $\left(\tfrac12,\ldots,\tfrac12\right)$ number 128. Roots that are permutations of $(\pm1,\pm1,0,0,0,0,0,0)$ are ordered lexicographically and denoted by the numbers $129-240$; for example, $(-1,-1,0,0,0,0,0,0)$ is number 129, and $(1,1,0,0,0,0,0,0)$ is number 240.

\vspace{11pt}

The table in Appendix \ref{list} contains the following information.

\begin{itemize}\label{legenda}
\item Graph: a graph $\Gamma_c$ where $c$ is a set of colors in $\{-2,-1,0,1\}.$
\item $K$: a clique in $\Gamma_c$; we denote vertices by their index as in the notation above.
\item $|K|$: the size of $K$.
\item $|W_K|$: the size of the stabilizer of clique $K$ in the group $W$.
\item $|\Aut(K)|$: the size of the automorphism group of $K$ as a colored graph.
\item $\#O$: the number of orbits of the set of all maximal cliques of size $|K|$ in $\Gamma_c$ under the action of $W$.
\end{itemize}

For each graph $\Gamma_c$, the list of cliques in $\Gamma_c$ in the table in Appendix \ref{list} gives exactly one representative for each orbit of the set of maximal cliques in $\Gamma_c$ under the action of $W$. The proofs of these results are in Proposition \ref{nomagma}, Corollary \ref{G-20}, Proposition~\ref{G-10}, Lemma \ref{easy}, Proposition \ref{-2-10}, and Proposition~\ref{hard}. 

\vspace{11pt}

The following remark shows the connection between del Pezzo surfaces and cliques in $\Graph{-2,0}$ and $\Graph{-1,0}$. 

\begin{remarkdp1}\label{dp12}
Let $X$ be a del Pezzo surface of degree one over an algebraically closed field, and let $C$ be the set of exceptional classes in~Pic~$X$. The question that led us to study the $E_8$ root system was how many elements of $C$ can go through the same point on $X$. The linear system $|-2K_X|$ realizes $X$ as a double cover of a cone in $\mathbb{P}^3$, ramified over a smooth sextic curve $B$ that does not contain the vertex of the cone. There are 120 hyperplanes that are tritangent to $B$, and such a hyperplane pulls back to the sum of two elements in $C$ that intersect with multiplicity three. It follows that two elements in $C$ intersecting with multiplicity three correspond to curves on $X$ intersecting in three points on the ramification curve. Conversely, if an element $c$ in $C$ corresponds to a curve on~$X$ that goes through a point $P$ on the ramification curve, then the unique element $c'\in C$ with $c\cdot c'=3$ corresponds to a curve on $X$ going through $P$ as well. \\
Through the bijection $C\longrightarrow E,\;c\longmapsto c+K_X$, two elements in $C$ that intersect with multiplicity $a$ correspond to two roots $e_1,e_2\in E$ with $e_1\cdot e_2=1-a$. Therefore, cliques in $\Gamma$ that correspond to sets of pairwise intersecting lines on $X$ have edges of colors $-2,-1,0$. Since elements in $C$ with intersection multiplicity 3 correspond to two roots in $E$ with dot product $-2$, it follows that a set of lines on $X$ that all go through one point on the ramification curve forms a clique in $\Graph{-2,0}$, and a set of lines on $X$ that all go through one point outside the ramification curve forms a clique in $\Graph{-1,0}$. This motivates why we have studied these two graphs extensively, and especially the biggest cliques in them (with respect to number of vertices).
\end{remarkdp1}

\subsection{Maximal cliques in \texorpdfstring{$\Graph{-2},\Graph{-1},\Graph{1},\Graph{-2,-1},\Graph{-2,1},$ and $\Graph{-2,-1,0,1}$}{G(-2),G(-1),G(0),G(1),G(-2,-1),G(-2,1), and G(-2,-1,0,1)}}

\begin{lemma}\label{mono-1-2}
Cliques in $\Graph{-2,-1}$ and in $\Graph{-2,1}$ are monochromatic. 
\end{lemma}
\begin{proof}
For an element $e\in E$, its inverse $-e$ is the unique element intersecting it with multiplicity $-2$ (Proposition \ref{intersection}). Take $e,f\in E$ with $e\cdot f=-1$, then $-e\cdot f=1$, hence $e,f,-e$ do not form a clique in $\Graph{-2,-1}$. Therefore all cliques in $\Graph{-2,-1}$ are monochromatic. Analogously, the cliques in $\Graph{-2,1}$ are monochromatic. \end{proof}

\begin{proposition}\label{nomagma}
For $$c\in\{\{-2\},\{-1\},\{1\},\{-2,-1\},\{-2,1\},\{-2,-1,0,1\}\},$$ the table in Appendix \ref{list} gives the complete list of orbits of the maximal cliques in~$\Gamma_c$, as well as a correct representative for each orbit, the size of its stabilizer in~$W$, and the size of its automorphism group.
\end{proposition}
\begin{proof}
We showed in Section \ref{monocliques} that all maximal cliques in $\Graph{-2}$ have size 2, and that they form one orbit of size 120. We also showed that all maximal cliques in $\Graph{-1}$ have size 3, and they form one orbit of size 2240. In Proposition \ref{max178} we showed that there are two orbits of maximal cliques in $\Graph{1}$; one of size $69120$, which consists of cliques of size 7, and one of size $17280$, which consists of cliques of size 8. For $\Graph{-2,-1}$ and $\Graph{-2,1}$ we proved that all cliques are monochromatic in Lemma \ref{mono-1-2}, so the maximal cliques and their orbits are found by looking at the monochromatic subgraphs $\Graph{-2}$, $\Graph{-1}$, and $\Graph{1}$. \\
It is an easy check that for these five graphs, the cliques in the table are correct representatives of the orbits. The sizes of their stabilizers are found by dividing the order of $W$ by the size of their orbit. Since all the cliques in these five graphs are monochromatic, their automorphism group is the permutation group on their vertices. \\
Finally, note that $\Graph{-2,-1,0,1}=\Gamma$. The only maximal clique in $\Graph{-2,-1,0,1}$ is therefore the whole graph, which forms an orbit of size 1 under the action of $W$.\end{proof}

\subsection{Cliques in \texorpdfstring{$\Graph{0}$}{G(0)} and  \texorpdfstring{$\Graph{-2,0}$}{G(-2,0)}}\label{-20}

The following lemma describes the maximal cliques in $\Graph{-2,0}$.

\begin{lemma}\label{maxin-20}
In $\Graph{-2,0}$, the following hold. 
\begin{itemize}
\item[](i) The maximal size of a clique in $\Graph{-2,0}$ is 16, and there are no maximal cliques of smaller size. 
\item[](ii) The set of maximal cliques in $\Graph{-2,0}$ is given by $$\left\{\{e_1,\ldots,e_8,-e_1,\ldots,-e_8\}\;|\;\forall i:e_i\in E;\;\forall i\neq j: e_i\cdot e_j=0\right\}.$$
\end{itemize}
\end{lemma}
\begin{proof}
By Theorem \ref{gamma0}, all maximal cliques in $\Graph{0}$ are of size 8. Let $\{e_1,\ldots,e_8\}$ be a maximal clique in $\Graph{0}$. Then $\{e_1,\ldots,e_8,-e_1,\ldots,-e_8\}$ is a clique in $\Graph{-2,0}$ of size~16. Now assume that $\{c_1,\ldots,c_r\}$ is a clique in $\Graph{-2,0}$ of size bigger than 16. Since edges of color $-2$ connect a root and its inverse, the clique $\{c_1,\ldots,c_r\}$ contains a subclique of size at least $\left\lceil\frac{r}{2}\right\rceil$ with only edges of color $0$. But this would give a clique in $\Graph{0}$ of size at least $\left\lceil\frac{17}{2}\right\rceil=9$, contradicting Theorem \ref{gamma0}. We conclude that the maximal size of a clique in $\Graph{-2,0}$ is 16. Now assume that $S$ is a maximal clique in $\Graph{-2,0}$ of size smaller than 16. Let $K$ be the biggest (with respect to number of vertices) subclique of $S$ with only edges of color 0. Let $K'$ be a maximal clique in $\Graph{0}$ containing $K$, so $K'$ has size 8. Then the clique consisting of all vertices of $K'$ and their inverses is a clique in $\Graph{-2,0}$ of size 16 that strictly contains $S$, contradicting the maximality of~$S$. We conclude that there are no maximal cliques of size smaller than 16 in $\Graph{-2,0}$, concluding the proof of (i). Part (ii) is now obvious.
\end{proof}

To show that the group $W$ acts transitively on the maximal cliques in $\Graph{-2,0}$, we use the following lemma, which builds on results in previous sections. Recall the set $Y$ as defined above Lemma \ref{mapalpha}. 

\begin{lemma}\label{maxclique}The following hold.
\begin{itemize}
\item[](i) For every element $y=(e_1,\ldots,e_4)\in Y$, there is a unique maximal clique in $\Graph{-2,0}$ containing $e_1,\ldots,e_4$.
\item[](ii) Every maximal clique in $\Graph{-2,0}$ contains 896 distinct subsets of four roots $e_1,\ldots,e_4$ such that $(e_1,\ldots,e_4)$ is an element in $Y$.
\end{itemize}
\end{lemma}
\begin{proof}\begin{itemize}
\item[]
\item[](i) From Lemma \ref{ydisjoint} it follows that an element in $Y$ is contained in a unique clique of size 8 in $\Graph{0}$. But such a clique extends uniquely to a maximal clique in $\Graph{-2,0}$ by adding all inverses of the roots.
\item[](ii) By Lemma \ref{maxin-20}, a maximal clique in $\Graph{-2,0}$ consists of eight pairwise orthogonal roots and their inverses. Let $K$ be such a clique. Eight pairwise orthogonal roots in $K$ contain $\binom84-14=56$ distinct subsets of four roots that form an element in $Y$ by Corollary \ref{uniqueY}. Let $D=\{e_1,e_2,e_3,e_4\}$ be such a subset. If we replace a root in $D$ by its inverse, then the roots in $D$ still form an element in $Y$. This gives $56\cdot 2^4=896$ distinct subsets of $K$ of that form an element in $Y$. Since a set of four roots that contains both a root and its inverse never forms an element in $Y$, these are all of them.    \qedhere \end{itemize} \end{proof}

Let $\mathcal{S}$ be the set of all cliques of size 16 in $\Graph{-2,0}$. By Lemma \ref{maxin-20}, this is exactly the set of maximal cliques in $\Graph{-2,0}$. By Lemma \ref{maxclique} we have a surjective map $$s\colon Y\longrightarrow\mathcal{S}.$$

\begin{corollary}\label{trans16}
The group $W$ acts transitively on $\mathcal{S}$, and we have $|\mathcal{S}|=2025$.
\end{corollary}
\begin{proof}
Since the map $s$ is surjective and $W$ acts transitively on $Y$ (Proposition~\ref{orbits}), it follows from Lemma \ref{action} that $W$ acts transitively on $\mathcal{S}$. From Lemma \ref{maxclique} it follows that $|\mathcal{S}|=\frac{|Y|}{|896\cdot4!|}=2025$.  
\end{proof}

Let $K$ be an element of $\mathcal{S}$, and $W_K$ its stabilizer in $W$. Now that we fully described all maximal cliques in $\Graph{-2,0}$ and the action of $W$ on the set of these maximal cliques, we finish the study of $\Graph{-2,0}$ by studying the action of $W_K$ on $K$, and concluding that $W$ acts transitively on cliques of sizes $6,7,8$ in $\Graph{0}$ in Proposition~\ref{G0,678}. Consider the sets 
$$I=\{(e_1,e_2,e_3)\in K^3\;|\;e_1\cdot e_2=e_1\cdot e_3=e_2\cdot e_3=0\},$$
and 
$$J=\{(e_1,e_2)\in K^2\;|\;e_1\cdot e_2=0\}.$$

\begin{proposition}\label{drieK}
The group $W_K$ acts transitively on $I$.
\end{proposition}
\begin{proof}
From Lemma \ref{maxin-20} we know that $K$ consists of eight pairwise orthogonal roots and their inverses, so we have $|I|=16\cdot14\cdot12=2688$. Fix an element $\iota=(e_1,e_2,e_3)$ in $I$. We want to show that its orbit $W_K\iota$ has size $2688$, hence is equal to $I$. Let $W_{K,\iota}$ be the stabilizer in $W_K$ of $\iota$. We have $|W_K\iota|=\frac{|W_K|}{|W_{K,\iota}|}$, and 
$$\frac{|W|}{|W_{K,\iota}|}=\frac{|W|}{|W_K|}\cdot\frac{|W_K|}{|W_{K,\iota}|}.$$ 
By Corollary \ref{trans16} we have $\frac{|W|}{|W_K|}=|WK|=2025$. Moreover, we have $$\frac{|W|}{|W_{K,\iota}|}=\frac{|W|}{|W_\iota|}\cdot\frac{|W_\iota|}{|W_{\iota,K}|}.$$ By Proposition \ref{three} we have $\frac{|W|}{|W_\iota|}=|W \iota|=240\cdot126\cdot60=1814400$. We now compute $\frac{|W_\iota|}{|W_{\iota,K}|}=|W_{\iota}K|$. From Proposition \ref{three} we know that there are 24 roots $e\in E$ such that $(e_1,e_2,e_3,e)$ is an element in $Y$. Since $W_{\iota}$ acts transitively on those~24 roots by Proposition \ref{orbits}, the orbit $W_{\iota}K$ contains the cliques $s((e_1,e_2,e_3,e))$ for all 24 roots~$e$. Now fix $e$ and set $y=(e_1,e_2,e_3,e)$, and $L=s(y)$. From Lemma (i) we know that $L$ contains exactly eight roots $f$ such that $(e_1,e_2,e_3,f)$ is an element in $Y$. Therefore, they determine the same unique clique of size sixteen as $e$. We conclude that there are $\frac{24}{8}=3$ different cliques containing $\iota$. So we have $|W_{\iota}K|\geq3$, and we find $\frac{|W|}{|W_{K,\iota}|}\geq1814400\cdot3=5443200.$ It follows that $\frac{|W_K|}{|W_{K,\iota}|}\geq\frac{5443200}{2025}=2688$. Since on the other hand we have $\frac{|W_K|}{|W_{K,\iota}|}=|W_K\iota|\leq |I|=2688$, we have equality everywhere and we conclude that $W_K\iota=I$. This finishes the proof. 
\end{proof}

\begin{corollary}\label{tweeK}
The group $W_K$ acts transitively on $J$.
\end{corollary}
\begin{proof} 
We have a projection map $\lambda\colon I\longrightarrow J$ on the first two coordinates. Since $K$ consists of eight pairwise orthogonal roots and their inverses, if we fix two elements $e_1,e_2$ such that $(e_1,e_2)\in J$, there are $16-4=12$ elements $e\in K$ such that $(e_1,e_2,e)$ is contained in $I$. Therefore, $\lambda$ is surjective. From Proposition~\ref{drieK} and Lemma \ref{action}, it follows that $W_K$ acts transitively on $J$. 
\end{proof} 

\begin{corollary}\label{eenK}
The group $W_K$ acts transitively on $K$. 
\end{corollary}
\begin{proof}
We have a projection map $\lambda\colon J\longrightarrow K$ on the first coordinate. For every element $e$ in $K$ there are 14 elements $c$ such that $(e,c)\in J$, so $\lambda$ is surjective. From Corollary \ref{tweeK} and Lemma \ref{action} it follows that $W_K$ acts transitively on $K$. 
\end{proof}

\begin{proposition}\label{transitivepairs}For $n\in\{2,3,5,6,7,8\}$, the group $W$ acts transitively on the set $$D_n=\left\{\{e_1,\ldots,e_n,-e_1,\ldots,-e_n\} \;|\;\forall i:e_i\in E;\;\forall i\neq j:e_i\cdot e_j=0\right\}.$$ \end{proposition}\begin{proof}For $n=2,3,5$, this follows from the fact that $W$ acts transitively on the cliques of size $n$ in $\Graph{0}$ (Propositions \ref{orbitskleinerdan4} and \ref{gamma0}), and the fact that there is a surjective map from the set of cliques in $\Graph{0}$ of size $n$ to $D_n$. The case $n=8$ is Corollary \ref{trans16}. From Proposition \ref{tweeK}, it follows that the stabilizer $W_K$ in $W$ of $K$ acts transitively on the set$$\{(e_1,e_2,-e_1,-e_2)\in K^4\;|\;e_1\cdot e_2=0\}$$ Since $K$ consists of eight pairwise orthogonal roots and their inverses, the cliques of six pairwise orthogonal roots and their inverses in $K$ are the complements of the cliques of two orthogonal roots and their inverses in $K$, so this implies that $W_K$ acts transitively on the set of cliques of six pairwise orthogonal roots and their inverses contained in $K$, too. From Corollary \ref{trans16}, the statement now follows for $n=6$. The case $n=7$ is proved analogously since we showed that $W_K$ acts transitively on~$K$. \end{proof}

\begin{remark}
There are two orbits under the action of $W$ on the set $$\left\{\{e_1,\ldots,e_4,-e_1,\ldots,-e_4\} \;|\;\forall i:e_i\in E;\;\forall i\neq j:e_i\cdot e_j=0\right\}.$$ Indeed, this follows from Proposition \ref{orbitskleinerdan4} and the fact that there is a surjective map from the set of cliques of size $4$ in $\Graph{0}$ to this set. 
\end{remark}

As we mentioned before, the fact that $W$ acts transitively on the set of cliques of size $r$ for $1\leq r\leq8$ in $\Graph{0}$ is in \cite{DynMin10}. The following proposition shows how it follows from our results about $\Graph{-2,0}$ as well. 

\begin{proposition}\label{G0,678}
for $n=6,7,8$, the group $W$ acts transitively on the cliques of size $n$ in $\Graph{0}$. 
\end{proposition}
\begin{proof}
We know that $W$ acts transitively on the set $$D_n=\left\{\{e_1,\ldots,e_n,-e_1,\ldots,-e_n\} \;|\;\forall i:e_i\in E;\;\forall i\neq j:e_i\cdot e_j=0\right\}$$ from Proposition \ref{transitivepairs}. Let $F_n$ be the set of cliques of size $n$ in $\Graph{0}$. We have an obvious map $f\colon F_n\longrightarrow D_n$ which adds adds the inverses to all roots in an element in~$F_n$. Let $D=\{e_1,\ldots e_n,-e_1,\ldots,-e_n\}$ be an element in $D_n$ and consider its fiber $f^{-1}(D)$ in~$F_n$. This consists of all cliques $\{\pm e_1,\ldots,\pm e_n\}$, where for each root either itself or its inverse is chosen. The stabilizer $W_{D}$ of $D$ acts on $f^{-1}(D)$. Note that for $i\in\{1,\ldots,n\}$, the reflection in the hyperplane orthogonal to $e_i$ switches $e_i$ and $-e_i$ and fixes all other roots in $D$, hence it is an element in $W_D$. Therefore, $W_D$ acts transitively on $f^{-1}(D)$, and by Lemma \ref{action}, $W$ acts transitively on $F_n$.
\end{proof}

\begin{corollary}\label{G-20}
The table in Appendix \ref{list} gives the complete list of orbits of the maximal cliques in $\Graph{0}$ and $\Graph{-2,0}$, as well as a correct representative for each orbit, the size of its stabilizer in $W$, and the size of its automorphism group.
\end{corollary}
\begin{proof}Al maximal cliques in $\Graph{-2,0}$ are of size 16 (Lemma \ref{maxin-20}) and there is only one orbit of them, of size 2025 (Corollary \ref{trans16}). It is an easy check that the clique in the table is a representative of this orbit. Its stabilizer size is $\frac{|W|}{|2025|}=344064$. Its automorphism group is isomorphic to $\mu_2^8\rtimes S_8$ by Lemma \ref{semidirect product}, hence has size $2^8\cdot 8!$. In Theorem \ref{gamma0} we showed that all maximal cliques in $\Graph{0}$ have size 8, and that there are $518400$ of them. In Proposition \ref{G0,678} we showed that $W$ acts transitively on the set of these cliques. Therefore the stabilizer of the clique in the table has size $\frac{|W|}{518400}=1344$. Its automorphism group is the symmetric group on the 8 vertices. \end{proof}

We finish this subsection by proving Theorem \ref{main2} for maximal cliques in $\Graph{-2,0}$.

\begin{lemma}\label{thm2max-20}
Let $K_1$ and $K_2$ be two maximal cliques in $\Graph{-2,0}$, and let $f\colon K_1\longrightarrow K_2$ be an isomorphism between them. Then $f$ extends to an automorphism of $\Lambda$ if and only if for every subclique $S$ of four pairwise orthogonal roots in $K_1$, the image $f(S)$ in $K_2$ is conjugate to $S$ under the action of $W$.
\end{lemma}
\begin{proof}
By Corollary \ref{trans16}, the group $W$ acts transitively on the set of maximal cliques in $\Graph{-2,0}$. Therefore there is an element $\alpha$ in $W$ such that $\alpha(K_1)=K_2$. So $\alpha^{-1}\circ f$ is an element in the automorphism group $\Aut(K_1)$ of $K_1$. Of course, $f$ extends to an element in $W$ if and only if $\alpha^{-1}\circ f$ does. Moreover, for every set $S$ of four pairwise orthogonal roots, $f(S)$ and $(\alpha^{-1}\circ f)(S)$ are conjugate. We conclude that we can reduce to the case where $K_1=K_2$, and $f$ is an element in $\Aut(K_1)$. \\
By Lemma \ref{maxin-20}, we can choose a subclique $H=\{e_1,\ldots,e_8\}$ of $K_1$ of eight pairwise orthogonal roots, such that we have $K_1=\{e_1,\ldots,e_8,-e_1,\ldots,-e_8\}$. Let $\Aut(H)$ be the automorphism group of $H$ as colored graph, and let $(\Aut(K_1))_{H}$ be the stabilizer of $H$ in $\Aut(K_1)$. Since for every element $e\in K_1$ we have $e\in H$ or $-e\in H$, an element in $\Aut(H)$ determines a unique element in $(\Aut(K_1))_{H}$, and conversely, every element in $(\Aut(K_1))_{H}$, when restricted to $H$, determines a unique element in $\Aut(H)$. So we have an isomorphism $\varphi\colon \Aut(H)\xrightarrow{\sim} (\Aut(K_1))_{H}$. Let $f$ be an element in $\Aut(K_1)$. Using Lemma \ref{semidirect product}, write $f=a\circ r|_{K_1}$, where $a$ is an element in $\varphi(\Aut(H))$, and $r$ is a composition of reflections $r_{i}$ in the hyperplanes orthogonal to $e_i$ for certain $i$ in $\{1,\ldots,8\}$. By definition, $r|_{K_1}$ extends to the element $r$ in $W$, and $r(S)$ and $S$ are conjugate for all cliques $S$ of four orthogonal roots, so the statement in the lemma is true for $f$ if and only if it is true for $a$. Of course, if $a$ extends to an automorphism of~$\Lambda$, then $a$ and $a(S)$ are conjugate for all subcliques $S$ of $K_1$ of four orthogonal roots. Conversely, assume that $a(S)$ and $S$ are conjugate for all such $S$. Then in particular, for every subclique $S'$ of size 4 in $H$, the sets $a|_H(S')$ and $S'$ are conjugate. From Corollary~\ref{thm2mono0} it follows that $a|_H$ extends to an element in $W$. Write $w$ for an element in $W$ with $w|_H=a|_H$. Then $w|_{K_1}$ and $a$ are both elements in $(\Aut(K_1))_{H}$, that are identical on $H$, hence also on $K_1$. We conclude that $w|_{K_1}$ and $a$ are the same, so $a$ extends to $w\in W$. This finishes the proof. 
\end{proof}

\subsection{Cliques in \texorpdfstring{$\Graph{-1,0}$}{G(-1,0)}}\label{-10}

Consider the following twelve elements in $E$. \label{kliek12}
\begin{align*}
&t_1=( 1, 1, 0, 0, 0, 0, 0, 0 ); & &t_7=\left(-\tfrac{1}{2}, \tfrac{1}{2}, \tfrac{1}{2}, -\tfrac{1}{2}, -\tfrac{1}{2}, -\tfrac{1}{2}, -\tfrac{1}{2},-\tfrac{1}{2}\right)\\
&t_2=(0, 0, 1, 1, 0, 0, 0, 0);	& &t_8=\left(-\tfrac{1}{2}, \tfrac{1}{2}, -\tfrac{1}{2}, \tfrac{1}{2}, \tfrac{1}{2}, -\tfrac{1}{2}, \tfrac{1}{2}, -\tfrac{1}{2}\right)\\
&t_3=(0, 0, 0, 0, 1, 1, 0, 0); & &t_9=\left(-\tfrac{1}{2}, -\tfrac{1}{2}, \tfrac{1}{2}, -\tfrac{1}{2}, \tfrac{1}{2}, -\tfrac{1}{2}, \tfrac{1}{2}, \tfrac{1}{2}\right)\\
&t_4=( 0, 0, 0, 0, 0, 0, -1, 1);	& &t_{10}=\left(\tfrac{1}{2}, -\tfrac{1}{2}, -\tfrac{1}{2}, -\tfrac{1}{2}, \tfrac{1}{2}, -\tfrac{1}{2}, -\tfrac{1}{2}, -\tfrac{1}{2}\right)\\
&t_5=\left(-\tfrac{1}{2}, -\tfrac{1}{2}, -\tfrac{1}{2}, \tfrac{1}{2}, -\tfrac{1}{2}, \tfrac{1}{2}, -\tfrac{1}{2}, -\tfrac{1}{2}\right); & &t_{11}=\left(\tfrac{1}{2}, -\tfrac{1}{2}, -\tfrac{1}{2}, \tfrac{1}{2}, -\tfrac{1}{2}, -\tfrac{1}{2}, \tfrac{1}{2}, \tfrac{1}{2}\right)\\
&t_6=\left(-\tfrac{1}{2}, \tfrac{1}{2}, -\tfrac{1}{2}, -\tfrac{1}{2}, -\tfrac{1}{2}, \tfrac{1}{2}, \tfrac{1}{2}, \tfrac{1}{2}\right); & &t_{12}=\left(\tfrac{1}{2}, -\tfrac{1}{2}, \tfrac{1}{2}, -\tfrac{1}{2}, -\tfrac{1}{2}, \tfrac{1}{2}, \tfrac{1}{2}, -\tfrac{1}{2}\right)
\end{align*}

One can easily check that these twelve elements form a clique in $\Graph{-1,0}$, depicted below (where edges of color 0 are not drawn). We call this clique $T$.

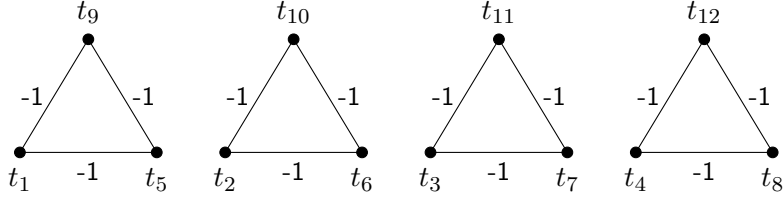
\begin{center}
\begin{tikzpicture}[scale=0.5]
 \node [label=below:$t_1$,draw,circle,fill,inner sep=0pt,minimum size=4pt](t1) at (1,1.5) {};
 \node [label=below:$t_5$,draw,circle,fill,inner sep=0pt,minimum size=4pt](t5) at (4,1.5) {};
 \node [label=above:$t_9$,draw,circle,fill,inner sep=0pt,minimum size=4pt](t9) at (2.5,4) {};
 \node [label=below:$t_3$,draw,circle,fill,inner sep=0pt,minimum size=4pt](t3) at (10,1.5) {};
 \node [label=below:$t_7$,draw,circle,fill,inner sep=0pt,minimum size=4pt](t7) at (13,1.5) {};
 \node [label=above:$t_{11}$,draw,circle,fill,inner sep=0pt,minimum size=4pt](t11) at (11.5,4) {};
 \node [label=below:$t_4$,draw,circle,fill,inner sep=0pt,minimum size=4pt](t4) at (14.5,1.5) {};
 \node [label=below:$t_8$,draw,circle,fill,inner sep=0pt,minimum size=4pt](t8) at (17.5,1.5) {};
 \node [label=above:$t_{12}$,draw,circle,fill,inner sep=0pt,minimum size=4pt](t12) at (16,4) {};
 \node [label=below:$t_2$,draw,circle,fill,inner sep=0pt,minimum size=4pt](t2) at (5.5,1.5) {};
 \node [label=below:$t_6$,draw,circle,fill,inner sep=0pt,minimum size=4pt](t6) at (8.5,1.5) {};
 \node [label=above:$t_{10}$,draw,circle,fill,inner sep=0pt,minimum size=4pt](t10) at (7,4) {};
\path[every node/.style={font=\sffamily\small}]
     (1,1.5) edge node [midway,below]{$-1$} (4,1.5)
     (4,1.5) edge node [midway,right]{$-1$} (2.5,4)
     (2.5,4) edge node [midway,left]{$-1$} (1,1.5)
     (10,1.5) edge node [midway,below]{$-1$} (13,1.5)
     (13,1.5) edge node [midway,right]{$-1$} (11.5,4)
     (11.5,4) edge node [midway,left]{$-1$} (10,1.5)
     (14.5,1.5) edge node [midway,below]{$-1$} (17.5,1.5)
     (17.5,1.5) edge node [midway,right]{$-1$} (16,4)
     (16,4) edge node [midway,left]{$-1$} (14.5,1.5) 
     (5.5,1.5) edge node [midway,below]{$-1$} (8.5,1.5)
     (8.5,1.5) edge node [midway,right]{$-1$} (7,4)
     (7,4) edge node [midway,left]{$-1$} (5.5,1.5) ;
\end{tikzpicture}
\end{center}

The existence of this clique implies that the maximal size of cliques in $\Graph{-1,0}$ is at least twelve. We will show that this is in fact the maximum. Moreover, we will show that all cliques of size twelve in $\Graph{-1,0}$ are isomorphic, and that $W$ acts transitively on the set of cliques of size twelve (Propositions \ref{maxnopairs} and \ref{trans12}). To describe all maximal cliques of smaller size in $\Graph{-1,0}$ and their orbits under the action of $W$, we use \texttt{magma} for part of the computations.

\begin{lemma}\label{blablabla}
Take $e_1,e_2,e_3\in E$ with $e_1\cdot e_2=e_2\cdot e_3=e_1\cdot e_3=-1$. For $e\in E$ with $e\neq e_1,e_2,e_3$, we have $e\cdot e_i\neq 1$ for all $i=1,2,3$ if and only if $e\cdot e_1=e\cdot e_2=e\cdot e_3=0$. 
\end{lemma}
\begin{proof}
Take $e_1,e_2,e_3\in E$ such that $e_1\cdot e_2=e_2\cdot e_3=e_1\cdot e_3=-1$. Then we have $\|e_1+e_2+e_3\|=0$, so $e_1+e_2+e_3=0$. For an element $e\in E$ with $\;e\neq e_1,e_2,e_3$ we have $e\cdot e_i\in\{-2,-1,0,1\}$ for $i=1,2,3$, so $e\cdot e_i\neq1$ for $i=1,2,3$ implies $e\cdot e_i\leq0$ for $i=1,2,3$. But $e\cdot(e_1+e_2+e_3)=e\cdot0=0$, so we have $e\cdot e_i\neq1$ for $i=1,2,3$ if and only if $e\cdot e_i=0$ for $i=1,2,3$. 
\end{proof}

\begin{lemma}\label{no221}
The maximum size of a clique in $\Graph{-1,0}$ that contains $e_1,e_2,e_3\in E$ with $e_1\cdot e_2=0$ and $e_1\cdot e_3=e_2\cdot e_3=-1$, is ten.
\end{lemma}
\begin{proof} 
Consider the elements $e_1=(1,1,0,0,0,0,0,0),\;e_2=(0,0,1,1,0,0,0,0)$, and $e_3=(-1,0,-1,0,0,0,0,0)$. By Lemma \ref{driehoek}, it is enough to prove that the maximal size of all cliques in $\Graph{-1,0}$ containing $e_1,e_2,e_3$ is ten. Let $A$ be the set $$\{e\in E\:|\:\mbox{for }i\in \{1,2,3\}: e\cdot e_i\in\{-1,0\}\}.$$For an element $e=(a_1\ldots,a_8)$ in $A$, we have $a_1+a_2\in\{-1,0\},\;a_3+a_4\in\{-1,0\}$, and $-a_1-a_3\in\{-1,0\}$. This gives the following possibilities for $(a_1,a_2,a_3,a_4)$:
\begin{align*}
(a_1,a_2,a_3,a_4) = & \left(-\tfrac{1}{2},\pm\tfrac{1}{2},\tfrac{1}{2},-\tfrac{1}{2}\right) & & \mbox{ (16 roots)}\\
& \left(\tfrac{1}{2},-\tfrac{1}{2},\pm\tfrac{1}{2},-\tfrac{1}{2}\right) & & \mbox{ (16 roots)}\\
& \left(\tfrac{1}{2},-\tfrac{1}{2},-\tfrac{1}{2},\tfrac{1}{2}\right) & & \mbox{ (8 roots)}\\ 
& ( 0, -1, 0, -1) & & \mbox{ (1 roots)}\\
& (0,0,1,-1) & & \mbox{ (1 root)}\\
& (1,-1,0,0) & & \mbox{ (1 root)}\\
& (0,-1,0,0) & & \mbox{ (8 roots)}\\
& (0,0,0,-1) & & \mbox{ (8 roots)}\\
& (0,0,0,0) & & \mbox{ (24 roots)}
\end{align*}
We conclude that the cardinality of $A$ is $83$. As it is too tedious to compute the maximal size of the cliques in $\Graph{-1,0}$ with only vertices in $A$ by hand, we compute this with \texttt{magma}. This number is seven, which implies that the maximal size of a clique in $\Graph{-1,0}$ containing $e_1,\;e_2$ and $e_3$ is ten.    
\end{proof}

\begin{lemma}\label{no51}
The maximum size of a clique in $\Graph{-1,0}$ that contains a clique of five pairwise orthogonal vertices is ten.
\end{lemma}
\begin{proof}
Consider the set $$V_5=\left\{\{e_1,\ldots,e_5\} \;|\;
\forall i: e_i\in E;\;\forall i\neq j:e_i\cdot e_j=0\right\}.$$ 
The group $W$ acts transitively on $V_5$ by Theorem \ref{gamma0}, so it suffices to take \begin{align*} 
&e_1=(1,1,0,0,0,0,0,0);  & &e_4=(0,0,0,0,0,0,1,1); \\
&e_2=(0,0,1,1,0,0,0,0);  & &e_5=(0,0,0,0,0,0,1,-1), \\
&e_3=(0,0,0,0,1,1,0,0); & &
\end{align*} and show that a clique in $\Graph{-1,0}$ containing $e_1,\ldots,e_5$ has size at most ten. Let $A$ be the set $$\{e\in E\:|\:\mbox{for }i\in \{1,\ldots,5\}: e\cdot e_i\in\{-1,0\}\}.$$ For an element $e=(a_1,\ldots,a_8)\in A$, we have $a_i+a_{i+1}\in\{-1,0\}$ for $i\in\{1,3,5,7\}$, and $a_7-a_8\in\{-1,0\}$. If $e$ is of the form $\left(\pm\tfrac{1}{2},\ldots,\pm\tfrac{1}{2}\right)$, then the fact that $a_7+a_8$ and $a_7-a_8$ are contained in $\{-1,0\}$ implies $a_7=-\tfrac{1}{2}$. Moreover, for $i\in\{1,3,5\}$, we have either $a_i=a_{i+1}=-\tfrac{1}{2}$ or $a_i=-a_{i+1}$. This gives three possibilities for each tuple $(a_i,a_{i+1})$ for $i\in\{1,3,5\}$, and $a_8$ is then determined since an even number of the entries of $e$ should be negative. We find $3^3=27$ possibilities.\\
If $e$ has two non-zero entries that are $\pm1$, then $a_7+a_8,\;a_7-a_8\in\{-1,0\}$ implies that either $(a_7,a_8)=(-1,0)$, or $(a_7,a_8)=(0,0)$. Moreover, for $i\in\{1,3,5\}$ we have $\{a_i,a_{i+1}\}=\{-1,0\}$ or $\{a_i,a_{i+1}\}=\{-1,1\}$. It is easy to check that this gives $24$ possibilities.\\
We find that the cardinality of $A$ is 51. As it is too tedious to compute the maximal size of the cliques in $\Graph{-1,0}$ with all vertices in $A$ by hand, we compute this with \texttt{magma}. The maximal size of a clique in $\Graph{-1,0}$ with all vertices in $A$ is five, so the maximal size of a clique in $\Graph{-1,0}$ containing $e_1,\ldots,e_5$ is ten.
\end{proof}

We recall some known Ramsey numbers.

\begin{theorem}\label{Ramsey}(Ramsey Numbers). For two integers $l,k$, let $R(l,k)$ be the least positive integer $n$ such that every undirected graph with $n$ vertices contains either a clique of order four or an independent set of order five. Then we have $R(3,3)=6$, $R(3,4)=9$, and $R(4,5)=25$.
\end{theorem}
\begin{proof}See \cite[Table 4.1]{GrRoSp} for $R(3,3)$ and $R(3,4)$, and \cite{McRa} for $R(4,5)$.
\end{proof}

\begin{proposition}\label{atleast41}
Every clique in $\Graph{-1,0}$ of size bigger than ten contains a subclique of size four with only edges of color $0$. 
\end{proposition}
\begin{proof}
Let $K$ be a clique in $\Graph{-1,0}$ of size bigger than ten. Consider the subgraph $K'$ of $K$ whose vertex set consists of all vertices of $K$, and whose edge set is obtained by taking only the edges in $K$ of color $-1$. We consider different cases depending on the number of connected components of $K'$.\\
If $K'$ has at least four connected components, then we can take four vertices, each from a different connected component, and these vertices form a clique of size four with only edges of color $0$ in $K$.\\
Now assume that $K'$ has at most three connected components. We first show that every connected component of $K'$ that contains a clique of size three is a clique of size three in itself. To this end, assume that $K'$ contains a clique of size three, given by $\{e_1,e_2,e_3\}$. By Lemma \ref{unique-1-1-1}, we have $e_1+e_2+e_3=0$. If $e$ is another vertex of~$K'$, then $e\cdot e_i\in\{-1,0\}$ for $i\in\{1,2,3\}$, and $e\cdot(e_1+e_2+e_3)=0$, from which it follows that $e\cdot e_i=0$ for $i\in\{1,2,3\}$. We conclude that the vertices $e_1,e_2,e_3$ form a connected component of $K'$. Since there are at most three connected components by assumption, and $K'$ has more than ten vertices, we conclude that not all components contain a clique of size three. Now remove a vertex from every connected component in $K'$ that is a clique of size three (of which there are at most two), then we are left with a subgraph of $K'$ with at least 9 vertices, and no cliques of size three left. Hence by Theorem \ref{Ramsey}, there must be a set of four vertices that are pairwise disjoint in~$K'$, meaning that they form a clique with edges of color $0$ in $K$.
\end{proof}

Let $V_3$, $V_4$, $Z$, $\alpha$, $\pi$ and $Y$ be as in the diagram above Lemma \ref{cardfiber}. 

\begin{proposition}\label{maxnopairs}
The following hold.
\begin{itemize}
\item[](i) Let $v=(e_1,e_2,e_3,e_4)$ be an element in $V_4$. Then $e_1,e_2,e_3$ and $e_4$ are contained in a clique of size bigger than ten in $\Graph{-1,0}$ if and only if $v$ is an element of $Y$.
\item[](ii) Every maximal clique of size at least eleven in $\Graph{-1,0}$ is of the form $$\left\{\left\{\begin{array}{c}e_1,\ldots,e_4,\\f_1,\ldots,f_4,\\-e_1-f_1,\ldots,-e_4-f_4\end{array}\right\}\left|\begin{array}{c}\forall i\neq j:e_i\cdot e_j=f_i\cdot f_j=0;\\\forall i:e_i\cdot f_i=-1;\\\forall i\neq j:e_i\cdot f_j=0.\end{array}\right.\right\}.$$ 
\item[](iii) The maximal size of a clique in $\Graph{-1,0}$ is twelve, and there are no maximal cliques of size eleven in $\Graph{-1,0}$. 
\item[](iv) For an element $v\in Y$, there are eight cliques of size twelve in $\Graph{-1,0}$ containing the elements of $v$.
\item[](v) For $K$ a clique of size twelve in $\Graph{-1,0}$, we have $|K^4\cap V_4|=|K^4\cap Y|=1944$. 
\end{itemize}
\end{proposition}
\begin{proof}
Let $K$ be a clique of size bigger than ten in $\Graph{-1,0}$. By Proposition \ref{atleast41}, we know that $K$ contains a subclique of size four with only edges of color~$0$. Let $\{e_1,e_2,e_3,e_4\}$ be such a subclique in $K$. Let $e$ be another element in~$K$. By Lemmas~\ref{no221} and \ref{no51}, there is exactly one $i\in\{1,2,3,4\}$ such that $e\cdot e_i=-1$, and $e\cdot e_j=0$ for $i\neq j\in\{1,2,3,4\}$. It follows that $e\cdot(e_1+e_2+e_3+e_4)=-1$, hence $\sum_{i=1}^4e_i\notin 2\Lambda$. By Proposition \ref{orbits}, this implies that $(e_1,e_2,e_3,e_4)$ is an element in $Y$. Conversely, the tuple $(t_1,t_2,t_3,t_4)$ is an element in $Y$ and it is contained in the clique $T$ (page~\pageref{kliek12}), so by Proposition \ref{orbits}, every element in $Y$ is contained in a clique of size twelve in $\Graph{-1,0}$. This proves (i).\\
Recall the clique $T$ defined above Lemma \ref{blablabla}. We define the following sets for $i$ in $\{1,2,3,4\}$.$$F_i=\left\{e\in E\;\left|\begin{array}{c}e\cdot t_i=-1,\\e\cdot t_j=0\text{ for }j\in\{1,2,3,4\},\;j\neq i\end{array}\right.\right\}.$$ 
Let $K$ be a clique in $\Graph{-1,0}$ of size at least eleven. Such a $K$ exists, since the clique~$T$ is an example. By Proposition \ref{atleast41}, the clique $K$ contains four vertices that form an element of $V_4$, and by part (i) this is an element of $Y$. By Proposition \ref{orbits} we can without loss of generality assume that $K$ contains the four vertices $t_1,t_2,t_3,t_4$. By Lemma \ref{no221} and Lemma~\ref{no51}, for every element $t$ in $K\setminus\{t_1,t_2,t_3,t_4\}$ there is an $i$ in $\{1,2,3,4\}$ such that $t\cdot t_i=-1$ and $t\cdot t_j=0$ for $i\neq j\in\{1,2,3,4\}$. Therefore we have  
%One can easily find the elements in these sets, and we have
%\begin{align*} 
%&F_1=\left\{\left(-\tfrac{1}{2},-\tfrac{1}{2},a_3,a_4,a_5,a_6,a_7,a_8\right)\left|\begin{array}{c}
%\{a_3,a_4\}=\left\{-\tfrac{1}{2},\tfrac{1}{2}\right\},\\
%\{a_5,a_6\}=\left\{-\tfrac{1}{2},\tfrac{1}{2}\right\},\\
%a_7=a_8
%\end{array}\right.\right\}; \\
%&F_2=\left\{\left(a_1,a_2,-\tfrac{1}{2},-\tfrac{1}{2},a_5,a_6,a_7,a_8\right)\left|\begin{array}{c}
%\{a_1,a_2\}=\left\{-\tfrac{1}{2},\tfrac{1}{2}\right\},\\
%\{a_5,a_6\}=\left\{-\tfrac{1}{2},\tfrac{1}{2}\right\},\\
%a_7=a_8
%\end{array}\right.\right\}; \\
%&F_3=\left\{\left(a_1,a_2,a_3,a_4,-\tfrac{1}{2},-\tfrac{1}{2},a_7,a_8\right)\left|\begin{array}{c}
%\{a_1,a_2\}=\left\{-\tfrac{1}{2},\tfrac{1}{2}\right\},\\
%\{a_3,a_4\}=\left\{-\tfrac{1}{2},\tfrac{1}{2}\right\},\\
%a_7=a_8
%\end{array}\right.\right\}; \\
%&F_4=\left\{\left(a_1,a_2,a_3,a_4,a_5,a_6,\tfrac{1}{2},-\tfrac{1}{2}\right)\left|\begin{array}{c}
%\{a_1,a_2\}=\left\{-\tfrac{1}{2},\tfrac{1}{2}\right\},\\
%\{a_3,a_4\}=\left\{-\tfrac{1}{2},\tfrac{1}{2}\right\},\\
%\{a_5,a_6\}=\left\{-\tfrac{1}{2},\tfrac{1}{2}\right\},\\
%\end{array}\right.\right\}.
%\end{align*} 
 $$K\setminus\{t_1,t_2,t_3,t_4\}=\bigcup_{i\in\{1,2,3,4\}}K\cap F_i.$$
Fix $i\in\{1,2,3,4\}$. For an element $f\in F_i$ we have $f\cdot t_i=-1$, so by Lemma~\ref{unique-1-1-1} there is a unique element $g\in E$ such that $f\cdot g=t_i\cdot g=-1$, given by $g~=~-t_i~-~f$. Note that this element is also in $F_i$, since $(-t_i-f)\cdot t_j=0$ for $j\in\{1,2,3,4\}$ with $j~\neq~i$. So for $i\in\{1,2,3,4\}$, the set $F_i$ is the union of different sets $\{f,-t_i-f\}$, and we claim that $K\cap F_i$ is contained in one of these sets. To prove this, fix $i$ and $f\in K\cap F_i$. Assume by contradiction that there is an element $h\in\left( K\cap F_i\right)\setminus\{f,-t_i-f\}$. Then~$h$ is in $F_i$, so $h\cdot f\neq-1$ by uniqueness of $g$. But $h,f$ are both elements in $K$, so this implies $h\cdot f=0$. But then we have $h\cdot t_i=f\cdot t_i=-1$ and $h\cdot f=0$, so by Lemma~\ref{no221}, the clique $K$ has size at most ten, which gives a contradiction. So for $i\in\{1,2,3,4\}$, there are $f_i\in F_i$ such that $K\cap F_i\subseteq\{f_i,-t_i-f_i\}$, and we have $$K\subseteq \bigcup_{i\in\{1,2,3,4\}}\{t_i,f_i,-t_i-f_i\}.$$
Fix such $f_i\in F_i$ for $i\in\{1,2,3,4\}$. We have $f_i\cdot f_j=0$ for $i\neq j\in\{1,2,3,4\}$, because if this were not the case then $K$ would contain a triple $t_i,f_i,f_j$ such that we would have $t_i\cdot f_i=f_i\cdot f_j=-1,\;f_j\cdot t_i=0$, which contradicts the fact that $K$ has size bigger than ten by Lemma \ref{no221}. Hence $\bigcup_{i\in\{1,2,3,4\}}\{t_i,f_i,-t_i-f_i\}$ forms a clique in $\Graph{-1,0}$ of the required form, and if $K$ is maximal, it is equal to this clique. This proves part (ii), and part (iii) follows directly. \\
We proceed by proving (iv). Note that $(t_1,t_2,t_3,t_4)$ is an element in $Y$. We count the number of cliques of size twelve in $\Graph{-1,0}$ containing $t_1,\ldots,t_4$. By (ii), we know that such a clique is of the form $\bigcup_{i\in\{1,2,3,4\}}\{t_i,f_i,-t_i-f_i\}$, where $f_i$ and $-t_i-f_i$ are elements in $F_i$ for $i\in\{1,2,3,4\}$. By simply considering all elements in $E$ we find $$F_1=\left\{\left(-\tfrac{1}{2},-\tfrac{1}{2},a_3,a_4,a_5,a_6,a_7,a_8\right)\left|\begin{array}{c}
\{a_3,a_4\}=\left\{-\tfrac{1}{2},\tfrac{1}{2}\right\},\\
\{a_5,a_6\}=\left\{-\tfrac{1}{2},\tfrac{1}{2}\right\},\\
a_7=a_8
\end{array}\right.\right\}.$$
Since $|F_1|=8$, there are four choices for the set $\{f_1,-t_1-f_1\}$. Fix $f_1$, and write $f_1~=~\left(-\tfrac{1}{2},-\tfrac{1}{2},a_3,\ldots,a_8\right)$. Then $f_2,\;-t_2-f_2$ are elements in $F_2$ that are orthogonal to $f_1$ by (ii). Again, by considering all elements in $E$ we find $$F_2=\left\{\left(b_1,b_2,-\tfrac{1}{2},-\tfrac{1}{2},b_5,b_6,b_7,b_8\right)\left|\begin{array}{c}
\{b_1,b_2\}=\left\{-\tfrac{1}{2},\tfrac{1}{2}\right\},\\
\{b_5,b_6\}=\left\{-\tfrac{1}{2},\tfrac{1}{2}\right\},\\
b_7=b_8
\end{array}\right.\right\}.$$
Let $f=(b_1,\ldots,b_8)$ be an element in $F_2$. Then  $f$ is orthogonal to $f_1$ if and only if $0=\sum_{i=5}^8a_ib_i=2(a_5b_5+a_7b_7)$, which holds if and only if $\frac{b_5}{b_7}=-\frac{a_7}{a_5}$. This gives two choices for the tuple $(b_5,b_7)$, and together with the two choices for $(b_1,b_2)$ we find four elements in $F_2$ that are orthogonal to $f_1$. This gives two choices for the set $\{f_2,-t_2-f_2\}$. Fix one. Then $f_3,\;-t_3-f_3$, and $\;f_4,\;-t_4-f_4$, are elements in $F_3$ and $F_4$ respectively, that are orthogonal to $f_1$ and $f_2$. It is an easy check that this determines the sets $\{f_3,-t_3-f_3\}$ and $\{f_4,-t_4-f_4\}$ uniquely. So for $f_1$ we had four choices, for $f_2$ we had two, and the set $\{f_3,-t_3-f_3,f_4,-t_4-f_4\}$ is determined after choosing $f_1,f_2$. We conclude that there are $4\cdot2=8$ cliques of size twelve in $\Graph{-1,0}$ containing $t_1,\ldots,t_4$. By Proposition \ref{orbits}, this holds for every element in $Y$. This proves (iv). \\
Let $K$ be a clique of size twelve in $\Graph{-1,0}$. Using the notation in (ii), write $$K=\left\{e_1,\ldots,e_4,f_1,\ldots,f_4,-e_1-f_1,\ldots,-e_4-f_4\right\}.$$ It follows from (ii) that the sets of four pairwise orthogonal roots in $K$ are given by $$\{\{a_1,a_2,a_3,a_4\}\;|\;a_i\in\{e_i,f_i,-e_i-f_i\}\mbox{ for }i\in\{1,2,3,4\}\}.$$ This gives $3^4=81$ such sets, and these give rise to $81\cdot 4!=1944$ elements in $K^4\cap V_4$. From (i) it follows that $K^4\cap V_4=K^4\cap Y$. This proves (v).
\end{proof}

\begin{proposition}\label{trans12}
Let $\mathcal{T}$ be the set of all cliques of size twelve in $\Graph{-1,0}$, and $R$ an element in $\mathcal{T}$. The following hold. 
\begin{itemize}
\item[](i) We have $|\mathcal{T}|=179200$, and the group $W$ acts transitively on $\mathcal{T}$.
\item[](ii) The stabilizer $W_R$ in $W$ of $R$ acts transitively on $R^4\cap Y$.
\end{itemize}
\end{proposition}
\begin{proof}Let $T$ be the clique $\{t_1,\ldots,t_{12}\}$, as defined above Lemma~\ref{blablabla}. Define the set $$S=\{((e_1,e_2,e_3,e_4),K)\in Y\times \mathcal{T}\;|\;e_1,\ldots,e_4\in K\}.$$We have projections $\lambda\colon S\longrightarrow Y$ and $\mu\colon S\longrightarrow\mathcal{T}$. \\
From the previous proposition we know that the fibers of $\lambda$ have cardinality~$8$, and the fibers of $\mu$ have cardinality $1944$. Therefore we have $|S|=|Y|\cdot8=348364800$ (Proposition \ref{orbits}), and $|\mathcal{T}|=\frac{|S|}{1944}=179200$. We will show that $W$ acts transitively on $S$, which implies that it acts transitively on $\mathcal{T}$ by the projection $\mu$. Consider the clique $T\in\mathcal{T}$, and set $y=(t_1,t_2,t_3,t_4)\in T^4\cap Y$. Then $(y,T)$ is in the fiber of $\lambda$ above~$y$. The stabilizer $W_y$ in $W$ of $y$ acts on this fiber. We show that this action is transitive, that is, that the orbit $W_{y}T$ is equal to the whole fiber. We have $|W_yT|=\frac{|W_y|}{|W_{y,T}|}$, and $|W_y|=\frac{|W|}{|Wy|}=\frac{|W|}{|Y|}=16$. Note that $t_1,t_2,t_3,t_4$ are all orthogonal to the four roots \begin{align*}
&e_1=(1,-1,0,0,0,0,0,0),\;\;\;\; & & e_2=(0,0,1,-1,0,0,0,0),\\
&e_3=(0,0,0,0,1,-1,0,0),\;\;\;\; & & e_4=(0,0,0,0,0,0,1,1).\end{align*} Therefore, for $i\in\{1,2,3,4\}$, the reflection $r_i$ in the hyperplane orthogonal to~$e_i$ is contained in the stabilizer $W_y$. Since the subgroup generated by these four reflections has cardinality 16, we conclude that this is the whole group~$W_y$. We can now compute that for every element $r$ in $W_y$ we have $rT\neq T$, except for the identity and the composition of all four reflections $r_1,r_2,r_3,r_4$. So $|W_{y,T}|=2$, and we have $|W_yT|=\frac{|W_y|}{|W_{y,T}|}=\frac{16}{2}=8$. Since the fiber of $\lambda$ above $y$ has cardinality $8$, we conclude that $W_y$ acts transitively on this fiber. Since $W$ acts transitively on $Y$, we conclude from Lemma \ref{action} that $W$ acts transitively on $S$. Finally, from the surjective projection $\mu$ and Lemma \ref{action}, it follows that~$W$ acts transitively on $\mathcal{T}$. This proves (i). Since $W$ acts transitively on $S$, the stabilizer~$W_R$ in $W$ of the clique $R$ acts transitively on the fiber~$\mu^{-1}(R)$. Since there is a bijection $\mu^{-1}(R)\longrightarrow R^4\cap Y$ given by the projection~$\lambda$, the group $W$ acts transitively on $R^4\cap Y$ by Lemma \ref{action}. This proves~(ii).
\end{proof} 

\begin{corollary}\label{corollary11}
Let $R$ be a clique of size twelve in $\Graph{-1,0}$. Let $W_R$ be its stabilizer in~$W$. Then $W_R$ acts transitively on $R$.
\end{corollary}
\begin{proof}We have a surjective map $R^4\cap Y\longrightarrow R$ projecting on the first coordinate, so this follows from the previous proposition and Lemma \ref{action}.\end{proof}

Now that we described all the largest cliques (with respect to number of vertices) in $\Graph{-1,0}$, we continue to describe all other maximal cliques. Since the size of the stabilizer of a clique is the same for every two cliques that are in the same orbit, we make the following definition.

\begin{definition}\label{stabilizersize}
The stabilizer size of an orbit is the size of the stabilizer of any of the elements in the orbit.
\end{definition} 

As one can see in the table in Appendix \ref{list}, for a set $c$ that contains $0$ in combination with either $-1$ or $1$, there are many maximal cliques in $\Gamma_{c}$ with small stabilizer sizes, which means large orbits. This means that, even though we use \texttt{magma} to find all cliques and orbits, computations can become very large and time consuming. Therefore we use the following lemma throughout.  

\begin{lemma}\label{HconjOfSsupA}
Let $H$ be a finite group acting on a finite set $X$ and consider its induced action on the power set of $X$. Let $A$ and $S$ be subsets of $X$ and let $m$ denote the number of $H$-conjugates of~$A$ that are contained in $S$. Then the number of $H$-conjugates of $S$ that contain $A$ equals 
\[
\frac{m \cdot |H_A|}{|H_S|},
\]
where $H_A$ and $H_S$ denote the stabilizer subgroups of $A$ and $S$, respectively.
\end{lemma}
\begin{proof}
Let $Z$ denote the $H$-subset of the product $HA \times HS$ consisting of all pairs $(B,T)$ with $B\in HA$ and $T \in HS$  satisfying $B \subset T$. The group $H$ acts transitively on the codomains of the projection maps $\pi\colon  Z \to HA$ and $\rho \colon  Z \to HS$. This implies that all fibers of $\pi$ have the same size, say $r$, as the fiber above $A$, which is the number of $H$-conjugates of $S$ that contain $A$, that is, the number that we are looking for. 
All fibers of $\rho$ have the same size as the fiber above $S$, which equals $m$. Hence, we can express the size of $Z$ as both $|HA|\cdot r$ and $|HS| \cdot m$.
Since the  orbits $HA$ and $HS$ have size $|H|/|H_A|$ and $|H|/|H_S|$, respectively, we find 
\[
r = \frac{m \cdot |HS|}{|HA|} = \frac{m \cdot |H_A|}{|H_S|}.
\]
\end{proof}

Note that for $A = \emptyset$, we recover the well-known fact that the length of the orbit of~$S$ equals the index $[H:H_S]$. 

The following proposition describes all maximal cliques and their orbits in $\Graph{-1,0}$. 

\begin{proposition}\label{G-10}
For two maximal cliques $K_1$ and $K_2$ of the same size in $\Graph{-1,0}$, the following are equivalent.
\begin{itemize}
\item[](i) $K_1$ and $K_2$ are conjugate under the action of $W$.
\item[](ii) $K_1$ and $K_2$ are isomorphic.
\item[](iii) $K_1$ and $K_2$ have the same stabilizer size.
\item[](iv) The automorphism groups of $K_1$ and $K_2$ have the same cardinality, and, if this cardinality is 16 and $K_1$ and $K_2$ have size 9, then $K_1$ and $K_2$ both contain a monochromatic clique of size 7 and color 0, or they both do not.  
\end{itemize}
Moreover, the table in Appendix \ref{list} gives a complete list of representatives of the orbits of the maximal cliques in $\Graph{-1,0}$, as well as for each representative its stabilizer size and the size of its automorphism group.
\end{proposition}
\begin{proof}The implications (i)$\Rightarrow$(ii), (i)$\Rightarrow$(iii), (i) $\Rightarrow$(iv), and (ii)$\Rightarrow$(iv) are immediate. We will show (iii)$\Rightarrow$(i) and (iv)$\Rightarrow$(i), which together with the immediate implications prove all equivalences. To this end, we first show that the table is complete and correct as stated. From Propositions \ref{maxnopairs} and \ref{trans12} we know that the maximal size of all cliques in $\Graph{-1,0}$ is twelve, that there are 179200 cliques of size twelve, and that these cliques form one orbit under the action of $W$, proving the equivalences for $K_1,\;K_2$ of size at least 12. The clique of size 12 in the table is the clique $T$ that is defined above Lemma \ref{blablabla}. The size of its stabilizer in $W$ is $\frac{|W|}{179200}=3888$. From the description of $T$ we see that its automorphism group is isomorphic to the semidirect product $S_3^4\rtimes S_4$, where $S_4$ acts on $S_3^4$ by permuting the four coordinates. This group has order $6^4\cdot24=31104.$ \\
To find maximal cliques in $\Graph{-1,0}$ of size smaller than 12, note that there are no maximal cliques in $\Graph{-1,0}$ of size 11 by Proposition \ref{maxnopairs}, so we only have to look at the cliques of size at most ten. To make computations easier, we first show that every maximal clique in $\Graph{-1,0}$ contains at least one edge of color $0$. We know that the only maximal cliques in $\Graph{-1}$ are the cliques of size three. 
Set $e_1=(1,1,0,0,0,0,0,0,0,0)$, $e_2=(-1,0,1,0,0,0,0,0)$, and $e_3=(0,-1,-1,0,0,0,0,0)$, then $\{e_1,e_2,e_3\}$ is a maximal clique in $\Graph{-1}$. Note that for $e_4=(0,0,0,0,0,0,1,1)$, the set $\{e_1,e_2,e_3,e_4\}$ forms a clique in $\Graph{-1,0}$, hence $\{e_1,e_2,e_3\}$ is not a maximal clique in $\Graph{-1,0}$. Since $W$ acts transitively on the set of maximal cliques in $\Graph{-1}$ (Corollary \ref{trans-1-1-1}), it follows that all maximal cliques in $\Graph{-1}$ are not maximal in $\Graph{-1,0}$. Thus we can assume that the maximal cliques in $\Graph{-1,0}$ contain at least one pair of orthogonal roots. Fix the roots $c_1=(1,1,0,0,0,0,0,0),\;c_2=(0,0,1,1,0,0,0,0)$. Since $W$ acts transitively on the pairs of orthogonal roots, every maximal clique in $\Graph{-1,0}$ is conjugate to a clique containing $c_1,c_2$, so by considering only the maximal cliques in $\Graph{-1,0}$ that contain $c_1$ and $c_2$, we find representatives for all orbits of the maximal cliques in $\Graph{-1,0}$ under the action of $W$. This reduces computations, since there are only 136 roots that have dot product $-1$ or $0$ with both $c_1$ and $c_2$, which is quickly computed with \texttt{magma}, as well as the number of maximal cliques containing $c_1,c_2$. We find the following.

\vspace{11pt}
\begin{center}
\begin{tabular}{c|c}
$r$ & Number of maximal cliques of size $r$\\
&in $\Graph{-1,0}$ containing $c_1$ and $c_2$\\
\hline
$\leq7$ & 0\\ 
\hline
8 & 261600\\
\hline
9 & 2779392\\
\hline
10 & 228408\\
\end{tabular}
\end{center}
\vspace{11pt}

We now turn to the table in the appendix. One can easily check with \texttt{magma} that the sets in the table for $\Graph{-1,0}$ are indeed maximal cliques in $\Graph{-1,0}$; in Remark \ref{vindenklieken}. For each of these cliques we compute the automorphism groups with \texttt{magma}. We see that apart from the cliques $$L_1=\{19,41,48,50,65,150,172,214,240\}$$
and $$L_2=\{41,48,50,55,65,78,178,214,240\}$$ of size 9, which both have an automorphism group of size 16, every two cliques of the same size in the table have a different automorphism group. One can check that $L_2$ contains a subclique with only edges of color zero of size 7, and $L_1$ does not, so $L_1$ and $L_2$ are not isomorphic. This shows that any two cliques of the same size in the table are not isomorphic, and therefore not conjugate. \\
We claim that every maximal clique in $\Graph{-1,0}$ is conjugate to one of these cliques in the table. To this end, set $A=\{c_1,c_2\}$, and let $W_A$ be the stabilizer of $A$ in $W$. From Proposition~\ref{orbitskleinerdan4} it follows that $|W_A|=\frac{|W|}{|WA|}=\frac{|W|}{15120}=46080$. We now show how to proceed for the cliques of size 8, the proof for sizes 9 and 10 goes completely analogously. For each of the five cliques of size 8 in the table we compute the size of its stabilizer (144,128,16,14, and 8) and the number of conjugates of $A$ contained in it (21,20,20,21, and 21, respectively), with \texttt{magma}. Lemma \ref{HconjOfSsupA} now gives us the number of conjugates of each clique that contain $A$. This sums up to the number 261600 we find in the table above, proving our claim.\\
We have showed that the table in the appendix gives exactly one representative for each orbit of the maximal cliques in $\Graph{-1,0}$, so $K_1$ and $K_2$ are both conjugate to an element in the table. If either (iii) or (iv) holds, then by looking at the table we see that this implies that $K_1$ and $K_2$ are conjugate to the same clique in the table, and in particular, they are conjugate to each other, implying (i). This finishes the proof.
\end{proof}

\begin{remark}\label{vindenklieken}
In the proof of Proposition \ref{G-10} we found 261600 cliques of size 8 in $\Graph{-1,0}$ containig $c_1=(1,1,0,0,0,0,0,0)$ and $c_2=(0,0,1,1,0,0,0,0)$. One can check for any two of them whether they are conjugate with \texttt{magma}, but this takes a very long time. To reduce computations, we first sort the cliques by size of their stabilizer. We then go through each set of cliques with the same stabilizer size by taking one clique, and removing all cliques that are conjugate to it from the set.
\end{remark}

\subsection{Maximal cliques of other colors}\label{othercolors}

In this subsection we prove Theorem \ref{main} and \ref{main2} for all maximal cliques in $\Gamma_c$ with $c\in\{\{-1,1\},\{-2,-1,1\},\{0,1\},\{-2,-1,0\},\{-2,0,1\}\}$. We make use of \texttt{magma} in all cases. The following lemma deals with the cases for which this is straightforward. 

\begin{lemma}\label{easy}For $c\in\{\{-1,1\},\{-2,-1,1\}\},$ and for two maximal cliques $K_1$ and $K_2$ of the same size in $\Gamma_c$, the following are equivalent. 
\begin{itemize}
\item[](i) $K_1$ and $K_2$ are conjugate under the action of $W$.
\item[](ii) $K_1$ and $K_2$ are isomorphic. 
\item[](iii) $K_1$ and $K_2$ have the same stabilizer size.
\item[](iv) The automorphism groups of $K_1$ and $K_2$ have the same cardinality.
\end{itemize}
Moreover, for $c\in\{\{-1,1\},\{-2,-1,1\}\},$ the table in Appendix \ref{list} gives a complete list of representatives of the orbits of maximal cliques in $\Gamma_c$, as well as for each representative its stabilizer size and the size of its automorphism group.
\end{lemma}
\begin{proof}
In these two graphs there are not so many maximal cliques, and we can ask \texttt{magma} to compute them, compute the orbits under the action of $W$, and a representative of each orbit directly. The results are in the table. The size of the stabilizers is found by dividing the order of $W$ by the size of the orbit. The automorphism group of the cliques is also easily found with \texttt{magma}. Since cliques of the same size in the table have automorphism groups of different size, they are not isomorphic. The equivalence of the statements (i), (ii), (iii), and (iv) now follows from the table. 
\end{proof} 

\begin{corollary}\label{thm2max-11,-2-11}
For $c\in\{\{-1,1\},\{-2,-1,1\}\},$ let $K_1$ and $K_2$ be two maximal cliques in $\Gamma_c$, and $f\colon K_1\longrightarrow K_2$ an isomorphism between them. Then $f$ extends to an automorphism of $\Lambda$. 
\end{corollary}
\begin{proof}
Since $K_1$ and $K_2$ are isomorphic, from Lemma \ref{easy} it follows that they are both conjugate to the same clique in the table in de appendix; call this clique $H$. Then there are elements $\alpha$, $\beta$ in $W$ such that $\alpha(K_1)=\beta(K_2)=H$. So $\beta\circ f\circ\alpha^{-1}$ is an element in the automorphism group $\Aut(H)$ of $H$. Of course, $f$ extends to an element in $W$ if and only if $\beta\circ f\circ\alpha^{-1}$ does. We conclude that we can reduce to the case where $K_1=K_2=H$, and $f$ is an element in $\Aut(H)$. \\
For each clique $H$ in the table, we construct in \texttt{magma} the map $W_H\longrightarrow \Aut(H)$ from the stabilizer $W_H$ to the automorphism group $\Aut(H)$ given by restriction. For all these cliques, this is a surjective map. It follows that every element in $\Aut(H)$ extends to an element in $W$.
\end{proof}

The final three cases are much more work, because of the large numbers of maximal cliques and their sizes. The most extreme case is that of maximal cliques of size~29 in $\Graph{0,1}$ and $\Graph{-2,0,1}$; we treat this separately in Section \ref{2901}.

\begin{remark}\label{vergelijken}
Recall that the classification of isomorphism classes of maximal cliques in $\Graph{0,1}$ has already been done in \cite{CRS04}, where the authors classify all maximal exceptional graphs (Remark \ref{previouswork}). We compare their methods to ours. For maximal cliques in $\Graph{0,1}$ of size unequal to 29, we find the different isomorphism types by showing that each such clique contains a pair of orthogonal roots, fixing such a pair of orthogonal roots, and computing the set of all maximal cliques in $\Graph{0,1}$ of size unequal to 29 that contain these two roots. We then find the different isomorphism types by cutting this set op into smaller sets using invariants that differentiate the isomorphism type of a clique $C$: the combination of either the stabilizer size with the number of pairs or inverse roots contained in $C$, or the combination of the cardinality of the automorphism group with the number of pairs or inverse roots contained in $C$ (Proposition \ref{hard}). Each isomorphism class turns out to be a full orit under the action of $W$ as well. For the maximal cliques of size 29 we do a similar computation, but in this case we show that such a clique contains a monochromatic 5-clique of color 0, or a monochromatic 4-clique of color 1 for which the sum of the corresponding root is a double root in $\Lambda$, or a monochromatic 4-clique of color 1 for which this sum is not a double root in~$\Lambda$. We fix one clique of each of these three types, and compute the set of all maximal cliques in $\Graph{0,1}$ of size 29 that contain at least one of these fixed cliques. We then cut this big set up in smaller sets using for each clique $C$ the stabilizer size and the number of maximal monochromatic subcliques of color~1 of size $r$, for all $r\in\{1,\ldots,8\}$ contained in $C$. Each of these smaller sets is an isomorphism class, and a full orbit under the action of $W$ (Proposition \ref{29cliques}).\\
In \cite{CRS04}, the authors use a different way to search for all maximal exceptional graphs. They prove that every exceptional graph arises as an extension of an \textsl{exceptional star-complement}, and construct a list of 443 graphs that arise as the exceptional star complements for maximal exceptional graphs. In \cite[Chapter 6]{CRS04}, the authors find all maximal exceptional graphs with a computer search, by extending each of the 443 exceptional star complements. Since an exceptional graph can arise as extensions of different star complements, or as different extensions from the same star complement, being an extension of a certain star complement is not an invariant that differentiates between isomorphism types of graphs. Therefore the authors of \cite{CRS04} do an isomorphism check in all 443 sets of extensions from the 443 start complements (as an example they state that for one star complement there were 1048580 extensions, giving 457 isomorphism types). \\
Since we use different methods, it is nice to see that our results coincide, and an alternative approach for finding all orbits of maximal cliques in $\Graph{0,1}$ could be to use the isomorphism types of these graphs that were already known in \cite{CRS04}, and compute the orbits per isomorphism type. It is not obvious that this would have been faster, however, since we would still have to check if two cliques are conjugate for \textsl{every} two cliques of a certain isomorphism type, which can be many.  
%They show that the exceptional graphs are exactly those that can be \textsl{represented} in $E_8$ (i.e. there is a set of roots in $E_8$, with Gram matrix $2(I+\frac12 A)$, where $A$ is the adjacency matrix of the graph), but are not generalized line graphs \cite[Corollary 3.6.4]{CRS04}. This means that the exceptional graphs are exactly those subgraphs of $\Graph{0,1}$ that are not line graphs. 
\end{remark}

\begin{proposition}\label{-2-10}For two maximal cliques $K_1$ and $K_2$ of the same size in $\Graph{-2,-1,0}$, the following are equivalent. 
\begin{itemize}
\item[](i) $K_1$ and $K_2$ are conjugate under the action of $W$.
\item[](ii) $K_1$ and $K_2$ are isomorphic. 
\item[](iii) $K_1$ and $K_2$ have the same stabilizer size, and, if the stabilizer size is 32 and $K_1$ and $K_2$ have size 10, then $K_1$ and $K_2$ both contain a pair of inverse roots, or they both do not.
\item[](iv) The automorphism groups of $K_1$ and $K_2$ have the same cardinality, and, if this cardinality is 80 and $K_1$ and $K_2$ have size 9, or this cardinality is 64 and $K_1$ and $K_2$ have size 10, then $K_1$ and $K_2$ both contain a pair of inverse roots, or they both do not.
\item[](v) $K_1$ and $K_2$ have the same stabilizer size and their automorphism groups have the same cardinality. 
\end{itemize}
Moreover, the table in Appendix \ref{list} gives a complete list of representatives of the orbits of maximal cliques in $\Graph{-2,-1,0}$, as well as for each representative its stabilizer size and the size of its automorphism group.
\end{proposition}
\begin{proof}This proof follows the same steps as the proof of Proposition \ref{G-10}. See also Remark \ref{vindenklieken} on how we found the representatives of each orbit that are written in the table.\\
Cliques in $\Graph{-2,-1,0}$ without an edge of color 0 are monochromatic and not maximal in $\Graph{-2,-1,0}$ (follows from the results in $\Graph{-2,-1},\Graph{-2,0},\Graph{-1,0}$). Therefore, to find the maximal cliques in $\Graph{-2,-1,0}$, we only consider cliques that contain two orthogonal roots, and we can choose these arbitrarily since $W$ acts transitively on the set of pairs of orthogonal roots. Define the following roots. $$e_1=\left(-\tfrac12, -\tfrac12, -\tfrac12, -\tfrac12, -\tfrac12, -\tfrac12, -\tfrac12, -\tfrac12\right),\;e_2=\left(-\tfrac12, -\tfrac12, -\tfrac12, -\tfrac12, \tfrac12, \tfrac12, \tfrac12, \tfrac12\right).$$ We find the following. 

\begin{center}
\begin{tabular}{c|c}
$r$ & Number of maximal cliques of size $r$\\
&in $\Graph{-2,-1,0}$ containing $e_1$ and $e_2$\\
\hline
$\leq7$ & 0\\
\hline
8 & 192480\\
\hline
9 & 1961088\\
\hline
10 & 743536\\
\hline
11 & 111680\\
\hline
12 & 8290\\
\hline
13 & 2100\\
\hline
14--15 & 0\\
\hline
16 &  15\\
\hline
$\geq17$ & 0 
\end{tabular}
\end{center}
\vspace{11pt} 

We turn to the table in the appendix. One can check that all the sets in the table for $\Graph{-2,-1,0}$ are indeed maximal cliques in $\Graph{-2,-1,0}$. For each of these cliques we compute the automorphism group with \texttt{magma}. As one can see in the table, except from two cliques $$L_1=\{1, 8, 26, 47, 51, 86, 121, 128, 228\},\;\;L_2=\{1, 8, 26, 47, 51, 86, 124, 125, 228\}$$ of size 9 that both have an automorphism group of size 80, and two cliques $$M_1=\{1, 8, 26, 31, 43, 46, 84, 98, 103, 125\},$$$$M_2=\{1, 8, 26, 31, 43, 46, 84, 101, 226, 238\}$$ of size 10 that both have an automorphism group of size 64, any two cliques of the same size have different automorphism groups and are therefore not isomorphic. Moreover, $L_1$ contains the roots 1 and 128, which are each other's inverse, whereas $L_2$ contains no pairs of inverse roots. And $M_1$ contains the roots 26 and 103, which are each other's inverse, and $M_2$ contains no pairs of inverse roots. So also $L_1,L_2,M_1$ and $M_2$ are pairwise not isomorphic. We conclude that any two of the cliques in the table are not isomorphic, hence not conjugate. \\
For each size $r$ in the table above, as we do in the proof of Proposition \ref{G-10}, we compute with Lemma \ref{HconjOfSsupA} and \texttt{magma} the number of maximal cliques of size $r$ containing $e_1$ and $e_2$ that are conjugate to one of the cliques in the table in the appendix. This gives exactly the number of maximal cliques of size $r$ containing $e_1$ and $e_2$ in the table above. So every maximal clique in $\Graph{-2,0,1}$ containing $e_1$ and $e_2$ is conjugate to a clique in the table in the appendix, hence the same holds for every maximal clique in $\Graph{-2,0,1}$. We conclude that the table in the appendix gives a unique representative for each orbit of the set of maximal cliques under the action of~$W$. Finally, for each clique in the table, we compute the size of its stabilizer in~$W$. We see that except for $N_1=\{1,8,26,31,43,86,106,115,224,234\}$ and $N_2=\{1,8,26,31,43,46,84,101,226,238\},$ two cliques of the same size in the table have different stabilizer sizes. In $N_1$, we have roots $43$ and $86$, and these are each other's inverse; in $N_2$, there are no two roots that are each other's inverse. Finally, $N_1$ and~$N_2$ have different automorphism groups. \\
The equivalence of statements (i) -- (v) follows in a similar way as in the proof of Proposition \ref{G-10}. The implications (i)$\Rightarrow$(ii), (i)$\Rightarrow$(iii), (i)$\Rightarrow$(iv), (i)$\Rightarrow$(v) and (ii)$\Rightarrow$(iv) are immediate. Since both $K_1$ and $K_2$ are conjugate to one of the cliques in the table, if any of (iii) -- (v) are true, by looking at the table we see that this implies that $K_1$ and $K_2$ are conjugate to the same clique in the table, and in particular, they are conjugate to each other, implying (i). This proves that all 5 statements are equivalent.\end{proof} 

We can now prove Theorem \ref{main2} for maximal cliques in $\Graph{-1,0}$ and $\Graph{-2,-1,0}$; the statement is the same for these two graphs. Recall the following graphs that are defined in the introduction, where any two disjoint vertices have an edge of color 0 between them. 

\vspace{5pt}

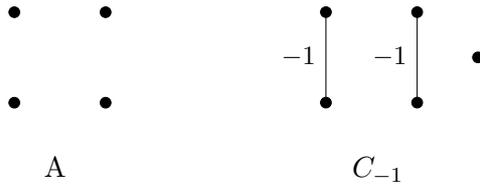
\begin{center}
\begin{tabular}{ccccccccc}
&{
\begin{tikzpicture} [scale=0.3]
 \node [draw,circle,fill,inner sep=0pt,minimum size=4pt](t1) at (0,2) {};
 \node [draw,circle,fill,inner sep=0pt,minimum size=4pt](t2) at (0,6) {};
 \node [draw,circle,fill,inner sep=0pt,minimum size=4pt](t3) at (4,2) {};
 \node [draw,circle,fill,inner sep=0pt,minimum size=4pt](t3) at (4,6) {};
 \end{tikzpicture}}
& & & & & & 
\begin{tikzpicture}[scale=0.4]
\node [draw,circle,fill,inner sep=0pt,minimum size=4pt](t1) at (0,2) {};
 \node [draw,circle,fill,inner sep=0pt,minimum size=4pt](t2) at (0,5) {};
 \node [draw,circle,fill,inner sep=0pt,minimum size=4pt](t3) at (3,2) {};
 \node [draw,circle,fill,inner sep=0pt,minimum size=4pt](t4) at (3,5) {};
 \node [,draw,circle,fill,inner sep=0pt,minimum size=4pt](t5) at (5,3.5) {};
\path[every node/.style={font=\sffamily\small}]
     (0,2) edge node [midway,left]{$-1$} (0,5)
     (3,2) edge node [midway,left]{$-1$} (3,5) ;
\end{tikzpicture} 
&  \\
&&&&&&&&\\
& A &&&&&& $C_{-1}$ &\end{tabular}
\end{center}

\vspace{5pt}

\begin{lemma}\label{thm2max-10en-2-10}
Let $K_1$ and $K_2$ be two maximal cliques, both contained in $\Graph{-1,0}$ or both  contained in $\Graph{-2,-1,0}$, and let $f\colon K_1\longrightarrow K_2$ be an isomorphism between them. The following hold.
\begin{itemize}
\item[](i) The map $f$ extends to an automorphism of $\Lambda$ if and only if for every ordered sequence $S=(e_1,\ldots,e_r)$ of distinct roots in $K_1$ such that the colored graph on them is isomorphic to $A$ or $C_{-1}$, its image $f(S)=(f(e_1),\ldots,f(e_r))$  is conjugate to $S$ under the action of $W$;
\item[](ii) If $S=(e_1,\ldots,e_5)$ is a sequence of distinct roots in $K_1$ such that the colored graph on them is isomorphic to $C_{-1}$ with $e_1\cdot e_4=e_2\cdot e_5=-1$, then $S$ and $f(S)$ are conjugate under the action of $W$ if and only if both $e=e_1+e_2+e_3-e_4-e_5$ and $f(e)$ are in the set $\{2f_1+f_2\;|\;f_1,f_2\in E\}$, or neither are.
\end{itemize}
\end{lemma}
\begin{proof}
Since $K_1$ and $K_2$ are isomorphic, from Propositions \ref{G-10}, and \ref{-2-10} it follows that they are both conjugate to the same clique in the table in the appendix; call this clique~$H$. Then there are elements $\alpha$, $\beta$ in $W$ such that $\alpha(K_1)=\beta(K_2)=H$, so $\beta\circ f\circ\alpha^{-1}$ is an element in the automorphism group $\Aut(H)$ of $H$. Of course, $f$ extends to an element in $W$ if and only if $\beta\circ f\circ\alpha^{-1}$ does. Moreover, for every sequence $S$ as in the statement, $f(S)$ and $(\beta\circ f\circ\alpha^{-1})(S)$ are conjugate. We conclude that we can reduce to the case where $K_1=K_2=H$, and $f$ is an element in $\Aut(H)$. Let $g\colon W_H\longrightarrow \Aut(H)$ be the map from the stabilizer of $H$ to the automorphism group that restricts elements in $W_H$ to $H$, and $T_H$ a set of representatives of the classes in the cokernel of $g$. Since $f$ is a composition of (restrictions of) elements in $W_H$ with an element in $T_H$, we can reduce further to the case where $f$ is an element in $T_H$. \\
For each of the 56 cliques $H$ in the table at $\Graph{-1,0}$ and $\Graph{-2,-1,0}$, we compute the map $g\colon W_H\longrightarrow \Aut(H)$ with \texttt{magma}. In all cases, this map is injective. This means that for all cliques with $|W_H|=|\Aut(H)|$, every element in the automorphism group of~$H$ extends to a unique automorphism of $\Lambda$. We see in the list that this holds for the first five cliques and the $11^{\mbox{th}},\; 12^{\mbox{th}},\;15^{\mbox{th}}$, and $16^{\mbox{th}}$ clique in $\Graph{-1,0}$, and the first five cliques and the $8^{\mbox{th}},\; 10^{\mbox{th}},\;11^{\mbox{th}},\;13^{\mbox{th}},\;17^{\mbox{th}},\; 20^{\mbox{th}},\;23^{\mbox{rd}}$, and $24^{\mbox{th}}$ clique in $\Graph{-2,-1,0}$.\\
For each clique $H$ of the remaining 34 cliques, we compute the following with the function \texttt{CokernelClassesTypeCminus1} \cite{magma}. First, we create a set $T_H$ of representatives of the classes of the cokernel of the map from $W_H$ to $\Aut(H)$. We then check for each $t$ in $T_H$, and for all sequences $S=(e_1,e_2,e_3,e_4,e_5)$ of distinct roots in $H$ such that the colored graph on $S$ is isomorphic to $C_{-1}$ with $e_1\cdot e_4=e_2\cdot e_5=-1$, whether $S$ and $t(S)$ are conjugate. For all $t$ and $S$ for which this is the case, we verify that either $e=e_1+e_2+e_3-e_4-e_5$ is in the set $F=\{2f_1+f_2\;|\;f_1,f_2\in E\}$ and $t(e)$ is not, or vice versa. This proves part (ii).\\ 
For $H$ equal to the $7^{\mbox{th}}-10^{\mbox{th}}\;,13^{\mbox{th}}\;,14^{\mbox{th}},$ and $18^{\mbox{th}}-23^{\mbox{rd}}$ clique in $\Graph{-1,0}$ and the 
$7^{\mbox{th}},\;9^{\mbox{th}},\;12^{\mbox{th}},\;14^{\mbox{th}},\;16^{\mbox{th}}\;,18^{\mbox{th}}\;,19^{\mbox{th}}\;,21^{\mbox{st}}\;,22^{\mbox{nd}}\;,25^{\mbox{th}}-29^{\mbox{th}}$, and $31^{\mbox{st}}$ clique in $\Graph{-2,-1,0}$, the check we just described gives us for all $t$ in $T_H$ a sequence $S$ with distinct roots in $H$ and graph isomorphic to $C_{-1}$, such that $S$ and $t(S)$ are not conjugate. For the remaining 7 cliques in the table, we do an almost analogous check with the function \texttt{CokernelClassesTypeA} in \texttt{magma} \cite{magma}, where $S$ is now a clique whose graph is isomorphic to $A$. For all 7 cliques $H$, for all elements in $T_H$, there exists such an $S$ with $S$ not conjugate to $t(S)$. This finishes the proof of~(i). 
\end{proof}

\begin{proposition}\label{hard}For $c\in\{\{0,1\},\{-2,0,1\}\},$ and $K_1$, $K_2$ two maximal cliques of the same size $r\neq29$ in $\Gamma_c$, the following are equivalent. 
\begin{itemize}
\item[] (i) $K_1$ and $K_2$ are conjugate under the action of $W$.
\item[] (ii) $K_1$ and $K_2$ are isomorphic.
\item[] (iii) $K_1$ and $K_2$ have the same stabilizer size, and they contain the same number of pairs of orthogonal roots. 
\item[](iv) The automorphism groups of $K_1$ and $K_2$ have the same cardinality, and $K_1$ and $K_2$ contain the same number of pairs of orthogonal roots.
\end{itemize}Moreover, the table in Appendix \ref{list} gives a complete list of representatives of the orbits of maximal cliques in $\Gamma_c$, as well as for each representative its stabilizer size and the size of its automorphism group.
\end{proposition}
\begin{proof}We show that the table is correct and complete for each $c$. The steps in the proof are the same as in the proofs of Propositions \ref{G-10} and \ref{-2-10}, and the equivalence of statements (i) -- (iv) follows in the same way as in these propositions. See also Remark \ref{vindenklieken} on how we found the representatives of each orbit that are written in the table.\\
$\bullet$ $c=\{0,1\}$\\
We know that the maximal cliques in $\Graph{1}$ form two orbits: one with cliques of size~7 and one with the cliques of size 8 (Proposition \ref{max178}). Note that the clique of size~7 in $\Graph{1}$ in the table is contained in the clique of size 22 in $\Graph{0,1}$, and the clique of size~8 in $\Graph{1}$ is contained in the clique of size 33 in $\Graph{0,1}$. This means that there are no maximal cliques with only edges of color 1 in $\Graph{0,1}$. We fix two orthogonal roots $e_1=\left(-\tfrac12, -\tfrac12, -\tfrac12, -\tfrac12, \tfrac12, -\tfrac12, -\tfrac12, \tfrac12\right),\;e_2=(-1, 0, 0, 0, -1, 0, 0, 0)$. With \texttt{magma} we compute that there are only 136 roots that have dot product 0 or 1 with $e_1$ and $e_2$, and we find the following. 

\vspace{11pt}
\begin{center}
\begin{tabular}{c|c}
$r$ & Number of maximal cliques of size $r$\\
&in $\Graph{0,1}$ containing $e_1$ and $e_2$\\
\hline
$\leq21$ & 0 \\
\hline
22 &  3120\\
\hline
23--27 & 0\\
\hline
28 & 21120\\
\hline
30 & 16263276\\
\hline
31 & 2792800\\
\hline
32 & 655680\\
\hline
33 & 105120\\
\hline
34 & 18800\\
\hline
35 & 0\\
\hline
36 & 304\\
\hline
$\geq37$ & 0\\
\end{tabular}
\end{center}
\vspace{11pt}

For each set $K$ in the table in Appendix~\ref{list}, one can check that it is indeed a maximal clique in $\Graph{0,1}$. We compute the automorphism groups of all cliques. As we see in the table, for all sizes except 30, two cliques of the same size have a different automorphism group, so they are not isomorphic, hence not conjugate. For size 30, all cliques whose automorphism groups have the same cardinality have a different number of pairs of orthogonal roots that they contain; for example, the cliques of size 30 with stabilizer size 48 contain (in order of appearence in the table) $171,\;179,\;180,\;183,\;198$ subsets of two orthogonal roots. This shows that no two cliques in the table are isomorphic, hence not conjugate. Moreover, using the stabilizer size and the number of subsets of orthogonal roots of each clique $K$ in the list, we can find the number of conjugates of $K$ that contain $\{e_1,e_2\}$ with Lemma~\ref{HconjOfSsupA}. Adding all these numbers up we recover the numbers in the table above, which shows that every maximal clique in $\Graph{-1,0}$ of size unequal to 29 is conjugate to one of the cliques in the list. We conclude that the table in Appendix~\ref{list} is complete. Finally, we see that for each clique in the table, the stabilizer size and the cardinality of the automorphism group is the same. Therefore, by what we showed above, different cliques of the same size and with the same stabilizer size in the table have a different number of subsets of two orthogonal roots. 
\\
$\bullet$ $c=\{-2,0,1\}$\\
We start with cliques in $\Graph{-2,0,1}$ containing an edge of color $-2$. We fix a root $e$ and compute the maximal cliques in $\Graph{-2,0,1}$ containing $e$ and $-e$. We find the following.

\vspace{11pt}
\begin{center}
\begin{tabular}{c|c}
$r$ & Number of maximal cliques of size $r$\\
&in $\Graph{-2,0,1}$ containing $e$ and $-e$\\
\hline
$\leq12$ & 0\\
\hline
13 & 370440\\
\hline
14 & 250236\\
\hline
15 & 0\\
\hline
16 & 77895\\
\hline
17--18 & 0\\
\hline
19 & 7019208\\
\hline 
20 & 861840\\
\hline
21 & 120960\\
\hline
22 & 44352\\
\hline
23 & 0\\
\hline 
24 & 4032\\
\hline
25--28 & 0\\
\hline
$\geq30$ & 0
\end{tabular}
\end{center}

\vspace{11pt}

Since there are no maximal cliques of size bigger than 29 containing an edge of color~$-2$, we conclude that all the maximal cliques in $\Graph{0,1}$ of size at least 29 are also maximal cliques in $\Graph{-2,0,1}$. This leaves us with the maximal cliques in $\Graph{0,1}$ of size 22 and 28. Looking at the table in the appendix, we see that for both sizes there is only one orbit, and it is an easy check that for the listed representatives $L_{22}$ of size 22 and $L_{28}$ of size 28 of both these orbits, there are no roots that can be added to extend the clique in $\Graph{-2,0,1}$. Therefore $L_{22}$ and $L_{28}$ are still maximal in $\Graph{-2,0,1}$. We now turn to the cliques in $\Graph{-2,0,1}$ in the table. First of all, one can check easily with \texttt{magma} that these are indeed maximal cliques in $\Graph{-2,0,1}$. For $K_1$ and $K_2$ of size 28 or $\geq30$, everything is exactly the same as for $\Graph{0,1}$, and we showed that the proposition holds in these cases. For the other cliques, we see that for all sizes except 13, 19, and 20, two different cliques of the same size have different automorphism groups. For sizes 13, 19, and 20, we compute, completely analogously to what we did for $c=\{0,1\}$, that the number of subsets of two orthogonal roots in two different cliques whose automorphism groups have the same cardinality is different. For example, the cliques of size 19 whose automorphism group has size~96, contain (in order of appearence in the table) $91,\; 95,\; 94,\; 98,\; 103$ subsets of two orthogonal roots. This proves that all the cliques in the table are pairwise not isomorphic, hence not conjugate. Again using Lemma \ref{HconjOfSsupA}, we can check that every maximal clique in $\Graph{-2,0,1}$ that is conjugate to one of the cliques in the table, showing that the table is complete. 
Finally, except for the cliques $$L_1=\{1, 8, 12, 14, 15, 20, 22, 23, 36, 38, 39, 128, 136, 137, 138, 139, 149, 160, 169\}$$ and $$L_2=\{1, 8, 12, 14, 50, 68, 70, 74, 128, 136, 137, 154, 169, 170, 176, 177, 181, 182, 215\}$$ of size 19, any two different cliques of the same size that have the same stabilizer size have the same cardinality of their automorphism groups as well. We already showed that this means that they contain a different number of pairs of orthogonal roots. We compute that $L_1$ contains 109 such pairs, and $L_2$ contains 79. Therefore we can conclude that different cliques of the same size and with the same stabilizer size in the table have a different number of subsets of two orthogonal roots. 
\end{proof}

\subsubsection{Cliques of size 29 in \texorpdfstring{$\Graph{0,1}$}{G(0,1)} and \texorpdfstring{$\Graph{-2,0,1}$}{G(-2,0,1)}}\label{2901}

\indent\textbf{Cliques of size 29 in $\Graph{0,1}$}\\
The graph $\Graph{0,1}$ contains a surprisingly large number of maximal cliques of size~$29$, so we will treat this case separately in this section. As before, we say that the stabilizer size of an orbit is the size of the stabilizer of any of the elements in the orbit (Definition \ref{stabilizersize}). 

\begin{proposition}\label{29cliques}
In the graph $\Graph{0,1}$ there are $62825152320$ maximal cliques of size~$29$. They form $432$ orbits under the automorphism group $W$.  The multiset of their stabilizer sizes is 
\begin{align*}
\big\{ 1^{(8)}, &2^{(81)}, 4^{(107)}, 6^{(5)}, 8^{(50)}, 10, 12^{(41)}, 14^{(2)}, 16^{(28)}, 18^{(2)},  20^{(5)}, 24^{(28)}, 32^{(4)}, 36, \\
 &48^{(21)}, 60, 64^{(2)}, 72^{(7)}, 96^{(3)}, 120, 128^{(2)}, 144^{(4)}, 192^{(7)}, 240^{(6)}, 360, 384^{(3)}, \\
 &432^{(2)}, 720^{(2)}, 1152^{(2)}, 1440, 1920, 40320, 51840, 103680 \big\},
\end{align*}
where the superscripts indicate the multiplicity of the elements in the multiset. For two maximal cliques $K_1$ and $K_2$ of size 29 in $\Graph{0,1}$, the following are equivalent. 
\begin{itemize}
\item[](i) $K_1$ and $K_2$ are conjugate under the action of $W$.
\item[](ii) $K_1$ and $K_2$ are isomorphic.
\item[](iii) $K_1$ and $K_2$ have the same stabilizer size, and the same number of maximal monochromatic subcliques of color 1 of size r, for all $r\in\{1,\ldots,8\}$. 
\item[](iv) The automorphism groups of $K_1$ and $K_2$ have the same cardinality, and $K_1$ and $K_2$ have the same number of maximal monochromatic subcliques of color 1 of size r, for all $r\in\{1,\ldots,8\}$.
\end{itemize}
Moreover, the table in Appendix \ref{29in01} gives a complete list of representatives of the orbits of maximal cliques of size 29 in $\Graph{0,1}$.
\end{proposition}

\begin{remark}
The large number of isomorphism types of maximal cliques of size 29 in $\Graph{0,1}$ is also noted in \cite[p.139--140]{CRS04}. The authors show that there are 473 isomorphism types of maximal exceptional graphs (corresponding to our 473 isomorphism types of maximal cliques in $\Graph{0,1}$, see Appendices \ref{list} and \ref{29in01}). The subset $S_1$ of isomorphism types of maximal exceptional graphs that have 29 vertices has cardinality 432 (corresponding to what we find in Proposition \ref{29cliques}). Moreover, the subset $S_2$ of isomorphism types of maximal exceptional graphs that have maximal vertex degree 28 has cardinality 467. The intersection of $S_1$ and $S_2$ consists of 430 isomorphism types of maximal exceptional graphs, and these form their own class (a) in the classification of all maximal exceptional graphs in \cite{CRS04}. They are listed in \cite[Table A6.1]{CRS04}. These 430 graphs arise as cones over graph-switching equivalents of the line graph $L(K_8)$ of the complete graph $K_8$ on 8 vertices \cite[Proposition 6.2.1]{CRS04}. The remaining 2 isomorphism types of maximal exceptional graphs of size 29 have maximal vertex degree less than 28. 
\end{remark}

The number of cliques mentioned in Proposition~\ref{29cliques} is too large to fit in most computers' memory: even if we were to use only 30 bytes per clique to store the vertices in the clique, then all cliques together would still require close to two terabytes of storage. Instead of doing this, we will use the fact that each $29$-clique contains a monochromatic $5$-clique of color $0$ or a monochromatic $4$-clique of color $1$. 
 
\begin{proof}The Ramsey number $R(4,5)$ equals $25$ (Theorem \ref{Ramsey}). This implies 
that a $29$-clique in $\Graph{0,1}$ contains a $5$-clique of edges of color $1$ or a $4$-clique of edges of color~$0$. 
Under the action of the automorphism group $W$ there is only one orbit of $5$-cliques with only edges of color $1$ (see Proposition \ref{transitief}); we call these cliques of type $\typeone$, and there are two orbits of $4$-cliques with pairwise orthogonal roots (see Proposition~\ref{orbitskleinerdan4}); we call the $4$-cliques of which the sum is a double root of type $\typetwoa$ and those of which the sum is not a double of type $\typetwob$. Therefore, if we fix a representative clique for each of these three orbits, then each $29$-clique is conjugate to a $29$-clique that contains one of our three cliques of size 4 or 5. \\
We pick the clique $A=\{1, 2, 129, 130, 131\}$ of type $\typeone$. There are $109$ other vertices that are connected with color $0$ or $1$ to each of the $5$ vertices of $A$. With \texttt{magma}, we count that the graph on these $109$ vertices with only edges of color $0$ or~$1$ has exactly $n_1=127168449$ maximal cliques of size $24$. After adding to each the vertices of $A$, this yields $n_1$ maximal $29$-cliques that contain $A$ in the graph~$\Graph{0,1}$. Similarly, for the cliques $B_1 = \{1,8,26,31\}$ and $B_2=\{1,8,26,43\}$ of type $\typetwoa$ and $\typetwob$, respectively, we count with \texttt{magma} that there are $n_2=16685128$ maximal $29$-cliques in $\Graph{0,1}$ that contain $B_1$, and $n_3=504$ maximal $29$-cliques that contain~$B_2$. \\
One can easily verify with \texttt{magma} that the 432 cliques of size 29 in the table in Appendix \ref{29in01} are maximal cliques in $\Graph{0,1}$. For each clique $K$ of size 29, for each integer $1 \leq r \leq 8$, we can consider the number $\chi_r$ of maximal monochromatic subcliques of $K$ of color $1$ of size $r$. These eight invariants together pin down 430 out of the 432 cliques in the table. Only the sequence $(\chi_1, \chi_2, \ldots, \chi_8) = ( 0, 0, 0, 0, 0, 4, 138, 17 )$ occurs twice: for the $67$-th and $299$-th cliques in the table. These two cliques have 16 and 18 subcliques of type $\typetwoa$, respectively, so they are not isomorphic. We conclude that any two cliques in the table are not isomorphic, hence not conjugate. So there are at least 432 orbits of maximal $29$-cliques. We know that there are $483840$ cliques of size 5 in $\Graph{1}$ from Corollary \ref{numberfacets}, so the stabilizer of $A$ has size $\frac{|W|}{|WA|}=\frac{|W|}{483840}=1440$. The table also lists for each clique~$c$ the number of subcliques of type $\typeone$, as well as the stabilizer size, so we can use Lemma \ref{HconjOfSsupA} to calculate the number of conjugates of $c$ that contain $A$. Summing over all these 432 cliques, we obtain exactly the number $n_1$, so we conclude that all $n_1$ maximal $29$-cliques in $\Graph{0,1}$ that contain~$A$ are accounted for in these 432 orbits. Similarly, the stabilizers of $B_1$ and $B_2$ have sizes $4608$ and $384$, respectively. The table lists the number of subcliques of type $\typetwoa$ and $\typetwob$ for every given clique~$c$, so we can use Lemma~\ref{HconjOfSsupA} again to calculate the number of conjugates of $c$ that contain~$B_i$ for $i=1,2$. Summing over all 432 cliques, we find again that all maximal $29$-cliques containing $B_1$ or $B_2$ are accounted for in these $432$ orbits. \\
We conclude that there are $432$ orbits of $29$-cliques in $\Graph{0,1}$, as claimed, and since no two cliques in the table are isomorphic, this proves (i) $\Leftrightarrow$ (ii). The multiset of stabilizer sizes follows from the table. The length of the orbit of any clique $c$ is $\frac{|W|}{|W_c|}$. Summing over all $432$ cliques in the table, we find that the total number of $29$-cliques is also as claimed. Finally, as we saw before, the invariant $\chi_r$ is different for all cliques except for the $67$-th and $299$-th cliques in the table. These two cliques have stabilizer size $4$ and $8$, respectively, so the stabilizer size, together with the $\chi_r$ form a set of invariants that uniquely determine each of the 432 orbits of maximal $29$-cliques. This proves (i) $\Leftrightarrow$ (iii). The stabilizer of a clique maps to the automorphism group of this clique as a colored graph. In all 432 cases, the clique generates a full rank sublattice of our lattice, so this map is injective. It turns out that in all cases, it is in fact a bijection. This proves (iii) $\Leftrightarrow$ (iv). 
\end{proof}

\begin{corollary}\label{thm2max01}
Let $K_1$ and $K_2$ be two maximal cliques in $\Graph{0,1}$, and $f\colon K_1\longrightarrow~K_2$ an isomorphism between them. Then $f$ extends to a unique automorphism of $\Lambda$. 
\end{corollary}
\begin{proof}
Since $K_1$ and $K_2$ are isomorphic, from Propositions \ref{hard} and \ref{29cliques} it follows that they are conjugate to each other; this means that they are both conjugate to the same clique in the tables in de appendix; call this clique $H$. Then there are elements $\alpha$, $\beta$ in $W$ such that $\alpha(K_1)=\beta(K_2)=H$, so $\beta\circ f\circ\alpha^{-1}$ is an element in the automorphism group $\Aut(H)$ of $H$. Of course, $f$ extends to an element in $W$ if and only if $\beta\circ f\circ\alpha^{-1}$ does. We conclude that we can reduce to the case where $K_1=K_2=H$, and $f$ is an element in $\Aut(H)$. \\
In Proposition \ref{hard} we computed the stabilizers and automorphism groups of all cliques in $\Graph{0,1}$ of size unequal to 29, and in Proposition \ref{29cliques} we did the same for cliques of~29. In \texttt{magma} we construct for each clique in the table the map between the stabilizer and the automorphism group that is given by restriction. In all cases, this is an isomorphism. We conclude that all automorphisms of the cliques in the table extend to an element in $W$.
\end{proof}

The table in Appendix \ref{29in01} contains the results of the previous proposition, with a representative of each orbit. The notation in the table means the following.

\begin{itemize}\label{legenda2}
\item $K$: a clique in $\Graph{0,1}$; we denote vertices by their index as in the notation above Remark \ref{dp12}. 
\item $|W_K|$: the size of the stabilizer of clique $K$ in the group $W$.
\item $\#K_5(1)$: the number of cliques of size 5 with only edges of color 0 in $K$.
\item $\#K_4^a(1)$: the number of cliques in $K$ of four roots that sum up to a double root in $\Lambda$, with only edges of color 1. 
\item $\#K_4^b(1)$: the number of cliques in $K$ of four roots that do not sum up to a double root in $\Lambda$, with only edges of color 1.
\end{itemize}

\begin{remark}
In the proof of Proposition~\ref{29cliques}, we found more than 127 million cliques of size $29$ that contain $A=\{1, 2, 129, 130, 131\}$. To find that they represent exactly 432 different orbits, one might naively try to just verify for each pair whether they are conjugate. This takes too much time; as described in Remark \ref{vindenklieken}, we divided the big set into smaller sets according to the stabilizer sizes.\end{remark}

\textbf{Cliques of size 29 in $\Graph{-2,0,1}$}\\
It is an easy check that all 432 cliques of size 29 in $\Graph{0,1}$ in the table are maximal in~$\Graph{-2,0,1}$ as well. We conclude that the orbits of maximal cliques of size~29 in $\Graph{-2,0,1}$ are exactly the 432 that we found in $\Graph{0,1}$, and the orbits of maximal cliques of size~29 that contain an edge of color $-2$.\\
As we did in Proposition~\ref{hard}, we fix a root $e$ and compute all maximal cliques of size 29 in $\Graph{-2,0,1}$ that contain $e$ and $-e$ with \texttt{magma}. There are 56 of these, and they form one orbit under the action of the stabilizer $W_e$ of $e$. Since $W$ acts transitively on pairs of inverse roots, we conclude that all maximal cliques of size 29 in $\Graph{-2,0,1}$ that contain an edge of color $-2$ are in the same orbit; call this orbit $A$. One can easily check with \texttt{magma} that the clique of size 29 that is written in the table for~$\Graph{-2,0,1}$ is maximal, and moreover, it contains the roots 1 and 128, that are each other's inverse. We conclude that it is a representative of $A$. The stabilizer and automorphism group are computed with \texttt{magma}.

\vspace{11pt}

We finish with the proof of Theorem \ref{main2} for maximal cliques in $\Graph{-2,0,1}$. This is very similar to the proof of Lemma \ref{thm2max-10en-2-10}. Recall the graphs $A$, $C_1$, $D$, and $F$ as defined before Theorem \ref{main2}.

\begin{lemma}\label{thm2max-201}
Let $K_1$ and $K_2$ be two maximal cliques in $\Graph{-2,0,1}$, and $f\colon K_1\longrightarrow K_2$ an isomorphism between them. The following hold. 
\begin{itemize}
\item[](i) The map $f$ extends to an automorphism of $\Lambda$ if and only if for every subclique $S=\{e_1,\ldots,e_r\}$ of $K_1$ that is isomorphic to $A$, $C_1$, $D$, or $F$, its image $f(S)$ in $K_2$ is conjugate to $S$ under the action of $W$.
\item[]Let $S$ be a subclique of $K_1$. 
\item[](ii) If $S$ is isomorphic to $C_{1}$, then $S$ and $f(S)$ are conjugate if and only if both $\sum_{i=1}^5e_i$ and $\sum_{i=1}^5f(e_i)$ are in the set $\{2f_1+f_2\;|\;f_1,f_2\in E\}$, or neither are.
\item[](iii) If $S$ is isomorphic to $D$, then $S$ and $f(S)$ are conjugate if and only if both $\sum_{i=1}^5e_i$ and $\sum_{i=1}^5f(e_i)$ are in the set $\{2f_1+2f_2\;|\;f_1,f_2\in E\}$, or neither are.
\item[](iv) If $S$ is isomorphic to $F$, then $S$ and $f(S)$ are conjugate if and only if both $\sum_{i=1}^5e_i$ and $\sum_{i=1}^6f(e_i)$ are in $2\Lambda$, or neither are.
\end{itemize} 
\end{lemma}
\begin{proof}This proof is very similar to the proof of Lemma \ref{thm2max-10en-2-10}, so will sketch what we did and refer to the other proof for details. \\
We reduce again to the case $K_1=K_2=H$, with $H$ one of the 54 cliques in the list for $\Graph{-2,0,1}$ in the appendix, and $f$ a representative of a class of the cokernel of the map $g\colon W_H\longrightarrow \Aut(H)$, where $W_H$ is the stabilizer of $H$ in $W$, and $\Aut(H)$ is the automorphism group of $H$. \\
For each clique $H$ of those 54 in the table, we check with \texttt{magma} that the map $g\colon W_H\longrightarrow\Aut_H$ is injective; for the $13^{\mbox{th}},\;15^{\mbox{th}}$, and $17^{\mbox{th}}-54^{\mbox{th}}$ clique it is an isomorphism. It follows that for those cliques, every automorphism extends to an element in $W$, so we are done. Here we refer to Lemma \ref{thm2max01} for the cliques that are the same as in $\Graph{0,1}$.  \\
For each clique $H$ of the remaining 14 cliques in the list, we do the following in \texttt{magma} with the functions \texttt{CokernelClassesTypeF}, \texttt{CokernelClassesTypeD}, and \texttt{CokernelClassesTypeC1} \cite{magma}. We construct a set $T_H$ of representatives of the classes of the cokernel of the map from $W_H$ to $\Aut(H)$. We then check for each $t$ in $T_H$, and for all subcliques $S=\{e_1,\ldots,e_r\}$ of $H$ that are isomorphic to $F$ (or $D$, or $C_1$, respectively), whether $S$ and $t(S)$ are not conjugate. For all $t$ and $S$ for which this is the case, we verify that $\sum_{i=1}^re_i$ is in $2\Lambda$ (or in the set $\{2f_1+2f_2\;|\;f_1,f_2\in E\}$, or in the set $\{2f_1+f_2\;|\;f_1,f_2\in E\}$, respectively), and $\sum_{i=1}^rt(e_i)$ is not, or vice versa. This proves (ii), (iii), and (iv).\\
For $H$ equal to the $4^{\mbox{th}},\;8^{\mbox{th}},\;10^{\mbox{th}},\;14^{\mbox{th}}$, and $16^{\mbox{th}}$ clique, for each non-trivial element $t$ in $T_H$ there is a subclique $S$ of $H$ that is isomorphic to $F$, and such that $t$ and $t(S)$ are not conjugate. Similarly, for each clique $H$ of the remaining 9 cliques in the list, for each non-trivial element $t$ of $T_H$, there is a subclique $S$ of $H$ that is isomorphic to either $C_1$, $D$, or $A$, and such that $S$ and $t(S)$ are not conjugate. This finishes the proof of (i).
\end{proof}

\section{Proof of main theorems}\label{proofthm1}

We now put together all the results that form the proofs of Theorem~\ref{main} and Theorem~\ref{main2}, which are both stated in the introduction.

\begin{proofofthm1}
Part (i) is Proposition \ref{orbitskleinerdan4} (iii), and part (ii) is Proposition \ref{max178} (ii). We proceed with (iii). Of course, if $K_1$ and $K_2$ are conjugate under the action of $W$, they are isomorphic as colored graphs, since $W$ respects the dot product. Now assume that $K_1$ and $K_2$ are isomorphic as colored graphs. We will show that they are conjugate under the action of $W$. First of all, by Lemma \ref{step0}, we can assume that there is a type I, II, III, or IV, that both $K_1$ and $K_2$ belong to. Therefore we continue to prove the result per type. \\
For type I, the results for colors $-2$ and $-1$ are at the beginning of Section \ref{monocliques}; the results for color 0 are in Propositions \ref{orbitskleinerdan4}, \ref{gamma0} (iii), and \ref{G0,678}, and the results for color~1 are in Proposition \ref{transitief}. \\
For type II, from Proposition \ref{facets} we know what the cliques look like, and the results are then in Proposition \ref{transitief} and Corollary \ref{transsevenface}. \\
For type III, the results follow from Propositions \ref{A} and \ref{B}.\\
Finally, for type IV, the results follow from Propositions \ref{nomagma}, \ref{G-10}, Lemma~\ref{easy}, Propositions \ref{-2-10} and  \ref{hard}, and Section \ref{2901}.
\end{proofofthm1}

\begin{proofofthm2}
By Lemma \ref{step0}, we can assume that there is a type I, II, III, or IV, that both $K_1$ and $K_2$ belong to. Therefore we continue to prove (i) per type. First of all, if $K_1$ and $K_2$ are of type III, then $f$ always extends; this is shown in Corollary \ref{corB}. \\
If $K_1$ and $K_2$ are of type I, they are monochromatic. If they have color $-2$ or $-1$, then they are of type III (See Section~\ref{monocliques}). For color 0 the proof is in Corollary \ref{thm2mono0}, and for color 1 in Corollary \ref{thm2mono1}. \\
For type II, by Proposition \ref{facets}, $K_1$ and $K_2$ are either monochromatic of color~0, hence of type I, or they are both sets of the vertices of a $7$-crosspolytope, in which case the statement is in 
Corollary \ref{thm2crosspolytopes}.\\
If $K_1$ and $K_2$ are of type IV, they are maximal cliques in a graph $\Gamma_c$, where there are 14 different possibilities for $c$. For $c\in\{\{-2\},\{-1\},\{0\},\{1\}\}$, the cliques $K_1$ and $K_2$ are of type I, which we already covered (note that for $K_1$ and $K_2$ \textsl{maximal} in $\Graph{1}$, there is always an automorphism extending $f$!). For $c$ in $\{\{-2,-1\},\{-2,1\}\}$, the cliques $K_1$ and $K_2$ are of type I as well (Lemma \ref{mono-1-2}). For $c=\{-2,0\}$, the proof is in Lemma \ref{thm2max-20}. For $c\in\{\{-1,1\},\{-2,-1,1\}\}$, an isomorphism of maximal cliques always extends, see Corollary \ref{thm2max-11,-2-11}. The same holds for $c=\{0,1\}$, see Corollary~\ref{thm2max01}.
For $c\in\{\{-1,0\},\{-2,-1,0\}\}$, the statement is in Lemma \ref{thm2max-10en-2-10}.\\
For $c=\{-2,0,1\}$ the statement is Lemma \ref{thm2max-201}.\\
Finally, for $c=\{-2,-1,0,1\}$ there is one clique, which is $\Gamma$ itself, and every automorphism of $\Gamma$ is an element in $W$. This finishes (i). \\
Part (ii) follows from Propositions \ref{orbits} and \ref{orbitskleinerdan4} for type $A$, and it follows from Propositions \ref{transitief} and \ref{max178} for type $B$. Finally, part (iii) is in Lemma \ref{thm2max-10en-2-10}, and part~(iv) is in Lemma~\ref{thm2max-201}.
\end{proofofthm2}

\begin{remark}\label{welkedingenwelkeklieken}
From Theorem \ref{main2} it follows that for an isomorphism $f$ of two cliques $K_1$ and $K_2$ of types I, II, III, or IV, one can determine whether $f$ extends to an automorphism of $\Lambda$ by checking for all subcliques of $K_1$ of the form $A$, $B$, $C_{\alpha}$, $D$, or $F$, if $f$ restricted to an associated ordered sequence extends. However, one never has to check all subcliques of those six forms. The following table shows for each type of $K_1$ and $K_2$ which subcliques are sufficient to check.  

\begin{center}
\resizebox{\columnwidth}{!}{
\begin{tabular}{|c|c|c|c|c|c|c|c|c|}
\hline
Type & Subtype & All isomorphisms extend & A & B & C$_{-1}$ & C$_{1}$ & D & F \\
\hline
I & $\Graph{-2}$ & x & &&&&&\\
I & $\Graph{-1}$ & x & &&&&&\\
I & $\Graph{0}$ & & x &&&&&\\
I & $\Graph{1}$ & &  & x &&&&\\
\hline
II & $k$-simplex, $k\leq7$ &  &x &&&&&\\ 
II & $7$-crosspolytope & & &x&&&&\\
\hline
III & all & x &&&&&&\\
\hline
IV & $\Graph{-2}$  & x & &&&&&\\
IV & $\Graph{-1}$ & x & &&&&&\\
IV & $\Graph{0}$ & & x &&&&&\\
IV & $\Graph{1}$ & x &  & &&&&\\
IV & $\Graph{-2,-1}$ & x & &&&&&\\
IV & $\Graph{-2,0}$ & & x &&&&&\\
IV & $\Graph{-2,1}$ & x & &&&&&\\
IV & $\Graph{-1,0}$ & & x && x &&&\\
IV & $\Graph{-1,1}$ & x & &&&&&\\
IV & $\Graph{0,1}$ & x & &&&&&\\
IV & $\Graph{-2,-1,0}$ &  & x && x &&&\\
IV & $\Graph{-2,-1,1}$ & x & &&&&&\\
IV & $\Graph{-2,0,1}$ &  & x &&&x&x&x\\
IV & $\Graph{-2,-1,0,1}$ & x & &&&&&\\
\hline
\end{tabular}}
\end{center}
\end{remark}

\textbf{Acknowledgements}\\
We want to thank David Madore, who gave us useful references for results on $E_8$ and the action of $W$. Moreover, there is a great interactive view of $E_8$ on his website \url{http://www.madore.org/\textasciitilde david/math/e8w.html}, which has been very insightful. We want to thank Martin Bright and Marta Pieropan for useful comments. We want to thank an anonymous referee for telling us about the work in \cite{CRS04}.

\bibliographystyle{plain} 
\bibliography{WeylGroupE8}

\begin{landscape}
\appendix
\section*{Appendices}
\addcontentsline{toc}{section}{Appendices}
\renewcommand{\thesubsection}{\Alph{subsection}}

\subsection{Table results Section \ref{maximalcliques}}\label{list}

See Section \ref{maximalcliques}, above Remark \ref{dp12}, for an explanation of this table.

\setlength\LTleft{-1in}
\pagestyle{empty}
\footnotesize
\begin{longtable}{cccccl}\label{lijst}
Graph & $|K|$ & $\#O$ & $|W_K|$ & $|\mbox{Aut}(K)|$ & $K$\\
\hline
$\Graph{-2}$&2&1&5806080&2&$\{1,128\}$\\
\hline
$\Graph{-1}$&3&1&311040&6&$\{1,32,240\}$\\
\hline
$\Graph{0}$&8&1&1344&40320&$\{1,8, 26, 31, 43, 46, 52, 53\}$\\
\hline
$\Graph{1}$&7&1&10080&5040&$\{1, 2, 3, 5, 129, 130, 131 \}$\\
&8&1&40320&40320&$\{1, 2, 3, 4, 129, 130, 131, 132\}$\\
\hline
$\Graph{-2,-1}$&2&1&5806080&2&$\{1,128\}$\\
&3&1&311040&6&$\{1, 32, 240\}$\\
\hline
\end{longtable}

\newpage
\subsection{Table cliques of size 29 in \texorpdfstring{$\Graph{0,1}$}{G(0,1)}}\label{29in01}

See Section \ref{2901} for an explanation of this table.

\footnotesize
\begin{center}\begin{longtable}{rrrrrl}
Nr. &$|W_K|$& $\#K_5(1)$ & $\#K_4^a(0)$ & $\#K_4^b(0)$ & $K$ \\
\hline
1&    12 & 1176 & 36 & 0 & \{ 1, 2, 3, 4, 5, 9, 17, 19, 34, 37, 41, 49, 65, 66, 129, 130, 131, 132, 133, 
146, 148, 155, 156, 158, 165, 166, 167, 173, 174 \}\\
2&    48 & 2352 & 8 & 0 & \{ 1, 2, 3, 4, 5, 6, 8, 9, 10, 12, 14, 17, 18, 20, 22, 26, 129, 130, 131, 132, 
133, 134, 137, 138, 143, 144, 145, 148, 149 \}\\

3&    12 & 2008 & 12 & 0 & \{ 1, 2, 3, 4, 5, 6, 8, 9, 10, 17, 34, 42, 66, 129, 130, 131, 132, 133, 137, 143,
145, 146, 150, 155, 156, 157, 160, 165, 166 \}\\

4&    8 & 1256 & 32 & 0 & \{ 1, 2, 3, 4, 5, 6, 9, 12, 14, 17, 34, 42, 50, 65, 66, 129, 130, 131, 132, 133, 
145, 146, 150, 155, 156, 157, 160, 169, 173 \}\\

 5&   96 & 1288 & 36 & 0 & \{ 1, 2, 3, 4, 5, 6, 9, 10, 17, 19, 21, 25, 35, 37, 41, 65, 129, 130, 131, 132, 
133, 134, 143, 147, 148, 158, 165, 166, 173 \}\\

   6& 16 & 1256 & 28 & 0 & \{ 1, 2, 3, 4, 5, 6, 9, 17, 19, 25, 37, 41, 129, 130, 131, 132, 133, 134, 135, 
143, 146, 147, 148, 156, 158, 165, 166, 167, 173 \}\\

 7&   1 & 800 & 45 & 0 & \{ 1, 2, 3, 4, 5, 7, 9, 10, 13, 14, 16, 34, 37, 38, 41, 43, 129, 130, 131, 132, 
133, 138, 139, 142, 155, 158, 159, 160, 161 \}\\
\end{longtable}\end{center}\end{landscape}

\end{document}